\documentclass[11pt]{article}

\usepackage[T1]{fontenc}
\usepackage[utf8]{inputenc}
\usepackage{lmodern}
\usepackage[a4paper,margin=1in]{geometry}
\usepackage{amsmath,amssymb,amsthm,mathtools}
\usepackage{microtype}
\usepackage[colorlinks=true,linkcolor=blue,citecolor=blue,urlcolor=blue]{hyperref}
\usepackage{esint}
\usepackage{comment}
\usepackage[section]{placeins}

\numberwithin{equation}{section}
\theoremstyle{plain}
\newtheorem{theorem}{Theorem}[section]
\newtheorem{proposition}[theorem]{Proposition}
\newtheorem{lemma}[theorem]{Lemma}
\newtheorem{corollary}[theorem]{Corollary}
\theoremstyle{definition}

\theoremstyle{remark}
\newtheorem{remark}[theorem]{Remark}

\usepackage{mathrsfs}

\newcommand{\R}{\mathbb R}
\newcommand{\Z}{\mathbb Z}
\newcommand{\T}{\mathbb T}

\newcommand{\N}{\mathbb N}
\newcommand{\eps}{\varepsilon}
\newcommand{\ii}{\mathrm{i}}

\newcommand{\keywords}[1]{%
  \par\medskip
  \noindent\textbf{Keywords.} #1
}

\usepackage{xcolor}

\newcommand{\MA}[1]{{\color{blue}#1}}

\usepackage{algorithm}
\usepackage{algpseudocode}
\usepackage{float}

\title{Point vortex dynamics in quasi-periodic channels: transporting trajectories and ergodic distribution of equilibria}
\author{}
\date{}

\begin{document}
\author{Mohamed Ali
\and
Taoufik Hmidi
}
\maketitle

\begin{abstract}
We study the dynamics of a single point vortex in an unbounded planar
channel whose interfaces are quasi-periodic in the longitudinal
direction. The motion is governed by the Robin function. We introduce
a hull formulation that lifts the quasi-periodic geometry to a
periodic problem on a finite-dimensional torus and yields an exact
quasi-periodic representation of the Robin function. Under a natural
transversality condition, we construct global quasi-periodic invariant
graphs describing transporting vortex trajectories and show that a
full neighborhood of each boundary component is foliated by such
graphs. Under a Diophantine condition on the spatial frequencies, the
dynamics along each graph can be straightened to a constant drift,
so that the vortex motion is quasi-periodic modulo translation.
A central contribution concerns the distribution of vortex
equilibria. In the quasi-periodic setting, where no
fundamental spatial cell exists, we identify critical points of the
Robin function with crossings of a hypersurface by a Kronecker flow
on the hull torus. Exploiting unique ergodicity, we establish a general
zero-counting theorem, allowing finite-order tangencies, which yields
an explicit geometric flux formula for the asymptotic density of
critical points and their limiting phase distribution. The result is
nonperturbative once the critical hull is constructed. Explicit
periodic and quasi-periodic models reveal bifurcations and phase
transitions in the critical-point distribution, while numerical computations
provide quantitative validation of the analytical results.
\end{abstract}
\tableofcontents

\section{Introduction and main results}
\label{sec-introduction}

Point-vortex systems constitute one of the most classical
finite-dimensional models arising from the incompressible Euler
equations. Their origins go back to the pioneering work of Helmholtz
\cite{Helmholtz1858}, who introduced the vortex equations in the plane
to describe the mutual interaction of concentrated vortices. Their
Hamiltonian structure was subsequently developed by Kirchhoff and
Routh; see \cite{Kirchhoff1876,Routh1881}, as well as the modern
accounts \cite{MajdaBertozzi,MarchioroPulvirenti,Newton, Saffman}.
Beyond their intrinsic interest as finite-dimensional Hamiltonian
systems, point vortices provide an effective description of highly
concentrated vorticity distributions in inviscid incompressible flows.
Roughly speaking, when the vorticity is concentrated in a finite number
of well-separated small regions, each carrying a prescribed
circulation, the evolution of the corresponding vortex centers is
described, at leading order, by the point-vortex system; see, for
instance, \cite{MajdaBertozzi,MarchioroPulvirenti}. In this sense,
point vortices may be viewed as singular particles encoding the dominant
dynamics of concentrated coherent structures in two-dimensional Euler
flows.

Already in the full plane, point-vortex dynamics exhibits a remarkably
rich structure. Depending on the number, strengths, and relative
positions of the vortices, one encounters relative equilibria, rotating
and translating configurations, periodic or quasi-periodic motions, collapse phenomena,
scattering, and, for sufficiently many vortices, genuinely complicated
dynamics; see \cite{Aref2007,Newton,Saffman}. These phenomena make the
point-vortex system a fundamental model at the interface between fluid
dynamics, Hamiltonian systems, and geometric mechanics.

The situation changes substantially in the presence of boundaries.
The geometry of the domain then enters the Hamiltonian through the
regular part of the Green function. In particular, even an isolated
vortex interacts with a solid boundary through the harmonic correction
to the free-space Green function, and this boundary-induced interaction
may generate a nontrivial drift in the absence of any other vortices.
For a single vortex, all pairwise interaction terms vanish, and the
dynamics is therefore governed entirely by the Robin function; see
\cite{Gustafsson1979,Newton}. In this sense, the Robin function acts as
an effective potential encoding the influence of the domain geometry
on the vortex motion.

Boundary-induced point-vortex dynamics has been investigated in a
variety of geometrically nontrivial fluid domains. In particular,
conformal mappings of wavy-wall channels have been used to analyze how
boundary corrugations modify vortex trajectories and transport
\cite{LeeEtAl2010}. More generally, conformal and potential-theoretic
methods provide explicit constructions of Kirchhoff--Routh Hamiltonians
and allow the resulting vortex dynamics to be studied in domains
containing solid obstacles, gaps, or periodic geometric structures;
see, for instance,
\cite{CrowdyMarshall2005,CrowdyMarshall2006,BaddooCrowdy2019}.
These works demonstrate that boundary geometry can profoundly alter
point-vortex dynamics and provide a natural fluid-mechanical motivation
for the present study.

The geometry considered here is of a different nature. We allow the
channel boundaries to vary quasi-periodically in the longitudinal
direction. Consequently, there is in general no finite spatial period
cell to which the problem can be reduced. The wall-induced Hamiltonian,
and in particular the associated Robin function, must instead be
described through a quasi-periodic hull. This formulation makes it
possible to exploit the underlying torus dynamics and to investigate
not only individual vortex trajectories, but also the distribution and
recurrence properties of the critical points generated by the
quasi-periodic boundary geometry.\\
More precisely, let $\Omega\subset\R^2$ be a planar domain and denote by 
$R_\Omega$  its Robin function. The motion
 of a single point vortex $z(t)=(x(t),y(t))$ of circulation $\Gamma$ is
governed by the Hamiltonian system
\begin{equation}
\label{eq-single-vortex-intro}
\dot z(t)
=
\tfrac{\Gamma}{2}
\nabla^\perp R_\Omega(z(t)),
\qquad
\nabla^\perp=(\partial_y,-\partial_x).
\end{equation}
Since the velocity field is tangent to the level sets of $R_\Omega$,
the Robin function is conserved along the vortex motion:
\[
R_\Omega(z(t))=R_\Omega(z(0)),
\]
and every vortex trajectory is contained in a level set of the Robin
function. Away from critical points of $R_\Omega$, each connected
component of a level set is an orbit of the point-vortex flow. Thus, the geometry of the level sets and the
critical-point structure of $R_\Omega$ provide a direct description of
the phase portrait of the single-vortex dynamics.
The connection between the geometry of the domain
and the resulting vortex dynamics is particularly transparent, since
Green functions, and consequently Robin functions, transform naturally
under conformal mappings. This provides a powerful mechanism for
relating geometric properties of $\Omega$ to the structure of the
vortex trajectories; see \cite{Gustafsson1979,Nehari,Ransford}.
Let us emphasize an important distinction between bounded and unbounded
domains. Assume first that $\Omega$ is bounded. By Sard's theorem,
almost every value of $R_\Omega$ is a regular value. For such an energy
$E$, each connected component of
\[
\{z\in\Omega:R_\Omega(z)=E\}
\]
is a compact one-dimensional manifold and hence a closed curve. Since
the Hamiltonian vector field does not vanish on a regular level, the corresponding point-vortex motion is periodic in time. Thus, in bounded domains, regular energy levels generically give rise to periodic
point-vortex trajectories. Related existence results for
periodic solutions of several point-vortex systems in bounded domains, including
solutions with prescribed minimal period, have been obtained in
\cite{BartschSacchet2018}.

The global organization of these trajectories is therefore closely
related to the critical-point structure of the Robin function. This
structure is delicate and remains poorly understood in general
domains. A particularly favorable situation occurs for bounded convex
domains, where the Robin function is strictly convex and possesses a
unique critical point; see
\cite{CaffarelliFriedman,Gustafsson1990}. In that case, the phase portrait
is organized around a distinguished equilibrium and is foliated by
periodic level curves. Beyond such geometrically constrained settings,
considerably less is known about the number and nature of the critical
points of the Robin function.

The situation is fundamentally different in an unbounded domain. A
connected component of a regular level set of $R_\Omega$ need not be
compact, and the associated vortex trajectory may escape every compact
subset of the domain. The relevant dynamical structures are therefore
not necessarily closed orbits but transporting trajectories. This
already illustrates how strongly the geometry of the boundary can
affect the dynamics: while a single point vortex in the whole plane is
stationary, the presence of solid walls may generate equilibria, trapped
motions, separatrices, and unbounded transport.

This contrast suggests a natural change of perspective. In bounded
domains, a central question is the existence and organization of closed
level curves of the Robin function and the corresponding periodic
vortex trajectories. In channel domains, by contrast, one is naturally
led to consider noncompact level curves that cross the domain in the
longitudinal direction and generate transporting vortex trajectories.
The relevant question is therefore no longer primarily one of periodic
recurrence, but rather of the persistence, geometry, and modulation of
such transporting orbits under deformations of the channel boundaries.
\\
The flat strip provides the simplest reference configuration for this
problem. Let
\[
\Omega_0=\R\times(0,h).
\]
Its invariance under horizontal translations implies that the Robin
function depends only on the transverse variable,
\[
R_{\Omega_0}(x,y)=R_0(y)=\tfrac1{2\pi}
\log
\left(
\tfrac{2h}{\pi}
\sin\tfrac{\pi y}{h}
\right).
\]
Accordingly, the single-vortex system
\eqref{eq-single-vortex-intro} reduces to
\begin{equation*}
\label{eq-flat-strip-vortex-intro}
\dot x
=
\tfrac{\Gamma}{2}R_0'(y),
\qquad
\dot y=0.
\end{equation*}
Hence every horizontal line
\[
y=y_0,
\qquad
R_0'(y_0)\neq0,
\]
is an invariant orbit along which the vortex travels with the constant
longitudinal velocity
\[
c(y_0)
=
\tfrac{\Gamma}{2}R_0'(y_0).
\]
The centerline $y=\tfrac{h}{2}$, characterized by
$R_0'(\tfrac{h}{2})=0$, is precisely the set of stationary vortices.
Thus the flat strip is completely integrable and foliated, away from
the centerline, by horizontal transporting orbits parametrized by the
transverse coordinate.

This simple structure provides the natural starting point for studying
deformed channels. The purpose of the present work is to understand how
the horizontal transporting trajectories of the flat strip are modified
when translation invariance is broken by quasi-periodic boundary
deformations. More precisely, we consider channels of the form
\begin{equation*}
\label{eq-intro-channel}
\Omega
=
\left\{
(x,y)\in\R^2:
f_-(\omega_-x)<y<f_+(\omega_+x)
\right\},
\end{equation*}
where the two interfaces are smooth quasi-periodic graphs and remain
uniformly separated, see Section \ref{sec-two} for more details. Our aim is to determine how the geometry of these
interfaces is encoded in the Robin function and, consequently, how it
governs the persistence and modulation of transporting vortex
trajectories. {The qualitative transition from the flat-strip dynamics to periodic and
quasi-periodic boundary perturbations is summarized schematically in
Figure~\ref{fig:intro_schematic}.
}

\begin{figure}[t]
    \centering
    \includegraphics[width=\textwidth]{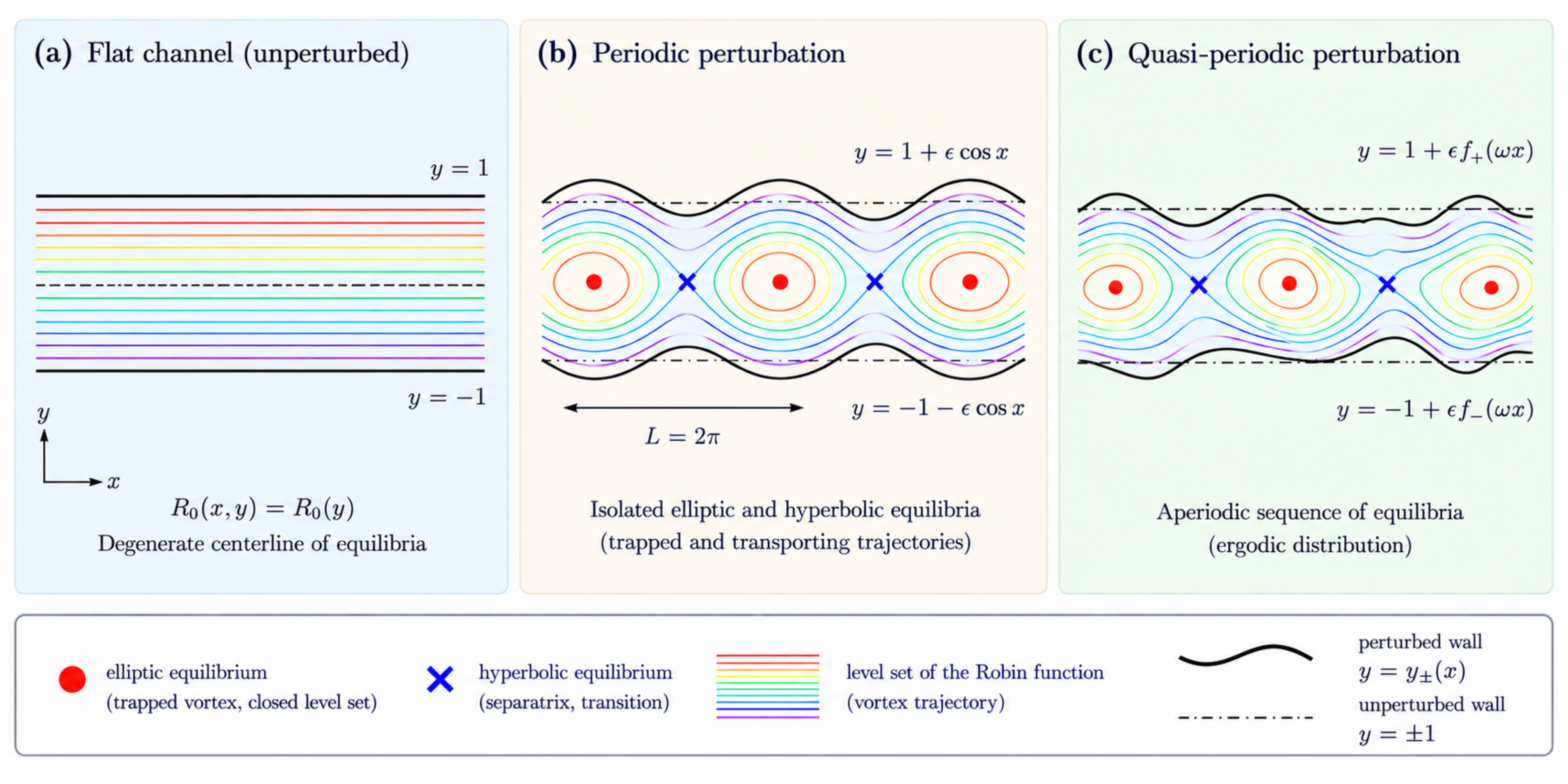}
    \caption{
Schematic of the three dynamical regimes considered in this work.
Left: in the flat channel, regular Robin levels are horizontal
transporting trajectories and the centerline is a degenerate family of
equilibria. Center: a periodic boundary perturbation may split this
critical line into isolated elliptic and hyperbolic equilibria,
producing trapped islands, separatrix-type levels, and transporting
trajectories. Right: for quasi-periodic boundary perturbations, analogous
local structures occur without a fundamental spatial period; the
equilibria form an aperiodic sequence whose large-scale distribution is
described by the ergodic theory developed below. The figure is
schematic and is not based on numerical data.
    }
    \label{fig:intro_schematic}
\end{figure}

This raises two complementary problems. The first concerns the dynamics: \textit{if an invariant trajectory is quasi-periodic as a spatial graph, does
its time parametrization remain quasi-periodic?} The second concerns the
critical-point structure of the Robin function: \textit{in the absence of a
spatial period, how should its critical points be described and
distributed?}

Both questions are naturally formulated through the compact hull
associated with the channel, in which longitudinal translations are
encoded by a linear flow on a finite-dimensional torus. For
transporting trajectories, the hull provides the appropriate
finite-dimensional description of their spatial and temporal
quasi-periodicity. For critical points, it replaces the unavailable
period cell and provides the framework for their statistical
distribution. We first address the dynamics of transporting
trajectories and return subsequently to the critical-point problem.

\subsection{Quasi-periodic channels and transporting trajectories}
We first investigate the persistence of transporting trajectories in quasi-periodic channels. We show that such channels possess a robust family of transporting orbits near each boundary. This phenomenon is nonperturbative: it follows from the universal logarithmic behavior of the Robin function near a smooth boundary. In particular, sufficiently close regular level sets can be represented as global graphs over the longitudinal variable and inherit the quasi-periodic structure of the channel.
This structure is naturally described through the compact hull of the channel. If $\theta\in\T^d$ denotes the phase associated with a longitudinal translation, then the Robin function can be represented as
\begin{align}\label{Robin-Intro}
R_\Omega(x,y)=\mathcal R(\omega x,y),
\end{align}
where $\mathcal R$ is periodic in $\theta$. The corresponding transporting level curves can then be written as
\[
y=\mathcal Y(\omega x),
\]
where $\mathcal Y$ is periodic on $\T^d$.
It is important, however, to distinguish the geometry of an orbit from its time parametrization. The existence and quasi-periodic graph structure of these trajectories require no arithmetic  assumption on $\omega$, whereas the motion along such a graph is generally nonuniform. Showing that its time parametrization is itself quasi-periodic amounts to straightening this variable-speed flow, leading to a cohomological equation on $\T^d$. Under a suitable Diophantine condition on $\omega$, this equation can be solved, and the vortex motion can be expressed as a constant longitudinal drift modulated by a bounded quasi-periodic oscillation.
We summarize these geometric and dynamical conclusions in the following qualitative statement.
\begin{theorem}[Quasi-periodic dynamics near the boundary]
\label{thm-intro-boundary}
Consider a channel whose upper and lower interfaces are smooth
quasi-periodic graphs in the horizontal variable with nonresonant frequency $\omega\in\R^d$. Then, sufficiently
close to either boundary, the support of every point-vortex orbit is
an invariant curve given by the graph of a quasi-periodic function of frequency $\omega.$
\\
Moreover, under a suitable Diophantine condition on  $\omega$, the dynamics along each such graph is
quasi-periodic modulo a uniform horizontal drift. More precisely, the
corresponding trajectory has the form
\[
z(t)=c_\star\,t\,e_1+Z(\nu t),
\]
where $c_\star\neq0,\,\nu=c_\star\omega$ and  $Z$ is periodic on a finite-dimensional torus $\T^d$.

\end{theorem}
The proof combines elliptic, geometric, and dynamical arguments. The
first step, developed in Section~\ref{sec-two}, is to lift the quasi-periodic
geometry to a periodic family of channels parametrized by
$\theta\in\T^d$. Translation covariance of the Green function then
yields the exact hull representation \eqref{Robin-Intro}.
A key point is to establish sufficient regularity of the hull function
$\mathcal R$ with respect to both the phase and transverse variables.
For this purpose, the family of hull domains is flattened onto a fixed
strip, reducing the problem to a family of uniformly elliptic
divergence-form equations whose coefficients depend periodically on
$\theta$. Elliptic regularity and parameter-dependent Schauder
estimates then yield the required regularity of $\mathcal R$. Since the
coefficients of the transformed operator involve first derivatives of
the boundary parametrizations, this procedure naturally entails the
loss of one derivative in the phase variables.

The geometric part of the argument is developed in
Section~\ref{sec-3.1} and is more general than the near-boundary statement of
Theorem~\ref{thm-intro-boundary}. The essential hypothesis is the
uniform transversality condition
\[
|\partial_y\mathcal R(\theta,y)|\geq c_0>0
\]
on a suitable phase-dependent region of the hull domain. Under this
condition, the implicit function theorem, together with the compactness
of $\T^d$, produces global periodic level graphs
$y=\mathcal Y(\theta)$; restricting them to the Kronecker orbit
$\theta=\omega x$ gives the quasi-periodic invariant graphs in the
physical channel. Thus, proximity to the boundary is only one mechanism
ensuring transversality. The general result applies equally to regions
in the interior of the channel whenever the same condition can be
verified; see Theorem~\ref{thm-exact-qp-orbits}.

The time parametrization requires a different argument. Once an
invariant graph has been constructed, the vortex equation reduces to a
scalar quasi-periodic equation for the longitudinal variable. In
Section~\ref{sec-3.2}, this variable-speed flow is straightened by solving a
cohomological equation on $\T^d$. This is the only stage at which an
arithmetic assumption enters: a Diophantine condition on $\omega$
controls the small divisors and yields a constant drift plus a bounded
quasi-periodic correction.
The near-boundary conclusion of
Theorem~\ref{thm-intro-boundary} is obtained by verifying the
transversality condition uniformly close to either interface, using
the universal boundary asymptotics of the Robin function; see
Subsection~\ref{subsec-qp-motion-near-boundary}. The underlying
transversality principle is, however, not restricted to a boundary
layer and applies in any region of the channel where
$\partial_y\mathcal R$ remains uniformly away from zero.

This observation is particularly useful for small quasi-periodic
perturbations of the flat strip. Away from the degenerate centerline
$y=h/2$, the vertical derivative of the flat-strip Robin function does
not vanish, and this property persists under sufficiently small
deformations of the interfaces. Consequently, the horizontal
transporting orbits of the flat strip persist as quasi-periodic
invariant graphs. In this perturbative regime, one can moreover go
beyond the qualitative theory: Hadamard's variational formula yields
explicit first-order expansions of the Robin function, the invariant
graphs, and the corresponding vortex trajectories. These results are
developed in Section~\ref{sec-4} and are summarized as follows.

\begin{theorem}[Persistence away from the centerline]
\label{thm-main-intro}
Consider a sufficiently small quasi-periodic perturbation, with nonresonant frequency $\omega,$ of the flat
channel
\(
\R\times(0,h).
\)
Then every horizontal point-vortex orbit of the flat channel whose
height $y_\star$ remains bounded away from the boundary and satisfies
$y_\star\neq h/2$ persists, for $|\eps|$ sufficiently small, as a
quasi-periodic invariant graph
\[
y=Y_{\eps,y_\star}(x),
\qquad
Y_{\eps,y_\star}(x)
=
y_\star+O(\eps),
\]
with the same spatial frequency vector as the channel.
\\
If, in addition, the frequency vector satisfies a suitable
Diophantine condition, the dynamics on this graph is quasi-periodic
modulo a constant horizontal drift. More precisely, the corresponding
trajectory can be written as
\[
z(t)
=
c_{\eps,y_\star}t\,e_1
+
Z_{\eps,y_\star}(\nu_{\eps,y_\star}t),
\]
where $Z_{\eps,y_\star}$ is periodic on $\T^d$ and
\[
\nu_{\eps,y_\star}
=
c_{\eps,y_\star}\omega,
\qquad
c_{\eps,y_\star}
=
\tfrac{\Gamma}{4h}
\cot\left(\tfrac{\pi y_\star}{h}\right)
+
O(\eps).
\]
\end{theorem}

The preceding results describe the persistence of the regular
transporting levels of the flat strip. A different phenomenon occurs
near the centerline $y=h/2$, where the flat Robin function loses its
transverse nondegeneracy and the whole horizontal line consists of
equilibria. Small boundary oscillations may break this degenerate
critical manifold into isolated critical points, producing elliptic
and hyperbolic vortex equilibria and changing the topology of the
nearby level sets. Sections~\ref{subsec-explicit-conformal-example} and~\ref{sec-boundary-oscillation-critical-points} are devoted to understanding this
symmetry-breaking mechanism and its dependence on the geometry of the
boundary oscillations.

We first consider in Section~\ref{subsec-explicit-conformal-example} an explicit conformal deformation of
the flat strip. This example plays two complementary roles. On the one
hand, conformal covariance gives an exact representation of the Robin
function and allows the critical points and the associated vortex
orbits to be analyzed directly. On the other hand, it provides an
independent test of the perturbative theory developed in Section~\ref{sec-4}:
expanding the exact conformal formula reproduces the first variation
obtained from Hadamard's formula and yields explicit expressions for
the deformation and time parametrization of the invariant curves. In
particular, the continuous family of equilibria forming the flat
centerline is destroyed by the perturbation and replaced, in this
model, by an alternating periodic family of isolated elliptic and
hyperbolic equilibria.

Section~\ref{sec-boundary-oscillation-critical-points} reveals a further effect which is not visible in a
single-mode perturbation: the critical-point structure may undergo
bifurcations as different spatial scales compete. We consider a
two-mode periodic deformation and use the first-order Robin expansion
to reduce the critical-point equation near the centerline to an
explicit finite-dimensional problem. The interaction between the two
Fourier modes is encoded by a Chebyshev polynomial, and the relative
amplitude of the modes becomes a bifurcation parameter. As this
parameter crosses distinguished values, pairs of non-axial critical
points are created or annihilated. In particular, for a perturbation
involving the modes $1$ and $n$, the number of nondegenerate critical
points in one spatial period may vary from two to $2n$.

These examples highlight an important feature of the problem:
boundary oscillations do not merely deform pre-existing vortex
trajectories; they may also reorganize the critical-point structure of
the Robin Hamiltonian and produce  bifurcations of the vortex
phase portrait. In the periodic setting, this structure can be
described by counting the critical points in a fundamental cell. This
naturally raises the question of how such a description should be
formulated for  quasi-periodic channels, where no spatial
period cell is available. This leads us to the statistical and
ergodic viewpoint developed in the next subsection.

\subsection{Ergodic distribution of vortex equilibria}

We now turn to the second main theme of the paper, which concerns the
spatial organization of the critical points of the Robin function in
 quasi-periodic channels. The absence of a fundamental
spatial cell changes the nature of the problem. In the periodic
setting, the critical-point structure can be described by counting
equilibria over one period. For a  quasi-periodic geometry,
such a finite-cell description is no longer available: isolated
critical points may occur indefinitely along the channel and form an
aperiodic set. The relevant quantity is therefore their frequency of
occurrence on large spatial scales.
More precisely, if $\{x_n\}_{n\in\mathbb Z}$ denotes the longitudinal
positions of a family of positive critical points, we are naturally led to the
asymptotic density
\begin{equation*}
\label{eq-density-intro}
\rho
=
\lim_{L\to\infty}
\tfrac{1}{L}
\#
\Big\{
n\in\mathbb Z:\ x_n\in[0,L]
\Big\},
\end{equation*}
provided that the limit exists. The existence of this limit is not a
local perturbative question. It depends on the recurrence of the
quasi-periodic geometry over arbitrarily large longitudinal distances
and therefore naturally brings ergodic theory into the problem.

The hull formulation provides the appropriate compact setting. After
solving the vertical criticality equation, the remaining condition
determining the longitudinal positions of the critical points takes
the form
\begin{equation*}
\label{eq-critical-hull-intro}
\mathcal S(\omega x)=0,
\end{equation*}
where $\mathcal S:\T^d\to\R$ is smooth and  periodic. Introducing the critical
hypersurface
\begin{equation*}
\label{eq-Sigma-intro}
\Sigma
=
\left\{
\theta\in\T^d:\ \mathcal S(\theta)=0
\right\},
\end{equation*}
we see that the physical critical points correspond to the spatial
times at which the Kronecker orbit
\[
x\longmapsto\omega x\pmod{2\pi}
\]
intersects $\Sigma$. Thus, in the physical channel, the problem reduces to studying the intersections between a linear flow and a hypersurface in a compact torus.
This correspondence is central to our approach. It separates the two
mechanisms governing the distribution of equilibria: the channel
geometry determines the hypersurface $\Sigma$, whereas the frequency
vector $\omega$ determines how $\Sigma$ is sampled by the Kronecker
flow. When the components of $\omega$ are rationally independent, the
latter flow is uniquely ergodic. Ergodicity is therefore not merely a
technical ingredient in the proof; it is the mechanism through which
the quasi-periodic geometry produces deterministic large-scale
statistics.
There is nevertheless an important obstruction to applying unique
ergodicity directly. The number of intersections with $\Sigma$ is a
singular observable rather than a continuous function on the torus.
Moreover, the Kronecker flow may become tangent to $\Sigma$, so that
the usual local description of a transverse crossing degenerates. The
main ergodic result of the paper shows that these difficulties can be
overcome and, in particular, that isolated tangencies of finite order
do not alter the leading asymptotic distribution.

The result is formulated at the level of a general periodic function
$\mathcal S$ and a nonresonant Kronecker flow, independently of its
specific origin from the Robin function. For $L>0$, we denote by
\begin{equation*}
N_{\mathcal S,\omega}(L)
:=
\#
\Big\{
x\in[0,L]:
\mathcal S(x\omega)=0
\Big\}
\end{equation*}
the number of intersections, counted without multiplicity, of the
orbit $x\mapsto x\omega$ with the zero hypersurface of $\mathcal S$.
Whenever the limit exists, we define the corresponding asymptotic
crossing density by
\begin{equation*}
\rho_{\mathcal S,\omega}
:=
\lim_{L\to\infty}
\tfrac{N_{\mathcal S,\omega}(L)}{L}.
\end{equation*}
The following theorem gives the existence of this limit and identifies
it intrinsically in terms of the geometry of the zero set and the
direction of the Kronecker flow. A complete statement is given in Theorem \ref{thm-critical-point-ergodic-distribution} and Theorem \ref{thm-density-finite-order-tangencies}.

\begin{theorem}
\label{thm-intro-critical-density-l1}
Let $\omega\in\R^d$ satisfy the nonresonance condition
and assume that $\Sigma$ is a nonempty smooth hypersurface. Assume that the intersections of
the Kronecker flow with $\Sigma$ are transverse, or more generally
that the tangencies satisfy the finite-order nondegeneracy conditions
specified below.
Then the asymptotic density $\rho_{\mathcal S,\omega}$ is well-defined and takes the form
\begin{align*}
\label{eq-density-delta-critical}
\rho_{\mathcal S,\omega}
&=
\fint_{\mathbb T^d}
\delta_0\bigl(\mathcal S(\theta)\bigr)
\left|
D_\omega\mathcal S(\theta)
\right|
\,d\theta
\\
\nonumber
&=
\frac{1}{(2\pi)^d}
\int_\Sigma
\frac{
|D_\omega\mathcal S(\theta)|
}{
|\nabla_\theta\mathcal S(\theta)|
}
\,d\sigma(\theta),\quad D_\omega:
=
\omega\cdot\nabla_\theta,
\end{align*}
where $\delta_0$ is the Dirac mass.

\end{theorem}

The formulas above admit a simple geometric interpretation. The
density is the mean flux of the Kronecker flow through the critical
hypersurface. A portion of $\Sigma$ is weighted by the normal velocity
with which the flow crosses it, namely by
$|D_\omega\mathcal S|/|\nabla_\theta\mathcal S|$. Consequently, the critical phases are not, in general, uniformly distributed on $\Sigma$: regions where the flow has a larger normal component carry a larger asymptotic weight. As we shall see in Theorem~\ref{thm-critical-point-ergodic-distribution}, this yields substantially more than a counting law. It defines a canonical probability measure on the critical hypersurface, describing the asymptotic distribution in the hull of the phases associated with vortex equilibria in the physical channel. 

An important feature of Theorem~\ref{thm-intro-critical-density-l1} is
that it is nonperturbative. Once the critical equation admits a
sufficiently regular hull representation and the geometric hypotheses
on $\Sigma$ are satisfied, neither the density formula nor the
limiting crossing measure requires the channel to be close to a flat
strip. Smallness is needed only in applications where perturbation
theory is used to construct the relevant critical graph, identify the
hypersurface $\Sigma$, or obtain explicit approximations of its
geometry. This separates the perturbative analysis of the Robin
function from the ergodic mechanism governing the large-scale
distribution of its critical points.

The treatment of tangencies is particularly relevant in this respect.
At a transverse intersection, the zero of
$\mathcal S(\omega x)$ is locally stable and standard regularization
arguments can be combined with unique ergodicity. At a tangency this
local mechanism degenerates, and a direct approximation of the
crossing count becomes more delicate. We show that isolated
finite-order tangencies can nevertheless be incorporated without
altering the flux formula. Thus global transversality, while
convenient, is not essential for the asymptotic density.

This point also becomes important when the geometry depends on an
additional parameter. As the parameter varies, transverse crossings
may merge into tangencies and subsequently disappear or split into
new crossings. The density may consequently undergo a loss of
regularity even when the underlying boundary deformation depends
smoothly on the parameter. The explicit two-frequency models studied in Section \ref{sec-two-frequency-cosine-case}
 exhibit this mechanism and lead to sharp regularity transitions
for the asymptotic density.

From a broader perspective, the novelty here lies in combining the
Robin-function hull with ergodic zero-counting to describe the
equilibrium set of a point vortex in a quasi-periodic geometry. Ergodicity
and equidistribution of Kronecker flows are classical, but the
observable arising in the present problem is singular and geometric:
one must count intersections with a critical hypersurface rather than
average a prescribed continuous function. The resulting framework
links the geometry of the channel, the arithmetic of its frequencies,
and the spatial statistics of vortex equilibria. In this way, the
finite-cell counting problem of periodic channels is replaced, in the
genuinely quasi-periodic setting, by an intrinsic density and a
canonical distribution of critical phases.

The paper is organized as follows. In Section~\ref{sec-two} we introduce the hull
formulation for channels with quasi-periodic interfaces and establish
the corresponding representation and regularity properties of the
Robin function. In Section~\ref{sec-03} we study the quasi-periodic geometry of
its regular level sets, the dynamics along the associated invariant
graphs, and the resulting foliation near the boundary. Section~\ref{sec-4} is
devoted to small quasi-periodic perturbations of the flat strip, where
Hadamard's variational formula is used to derive first-order expansions
of the Robin function, the invariant graphs, and the corresponding
vortex trajectories.

The subsequent sections turn to the vortex equilibria problem.
In Section~\ref{subsec-explicit-conformal-example} we study an explicit conformal perturbation of the flat
strip, which provides an exact model for the breaking of the degenerate
critical centerline into isolated equilibria and an explicit
illustration of the perturbative theory. Section~\ref{sec-boundary-oscillation-critical-points} investigates this
symmetry-breaking mechanism for two-mode periodic interfaces and
describes the resulting amplitude-dependent bifurcations in the number
and location of the critical points. Section~\ref{sec-quasi-periodic-interfaces-ergodicity} develops the ergodic
theory underlying the distribution of critical points in 
quasi-periodic channels. We first establish a general zero-counting
theory along Kronecker flows, including isolated finite-order
tangencies, and then apply it to the Robin function to obtain the
asymptotic density and phase distribution of vortex equilibria.
Particular attention is given to two-frequency models, where the
critical geometry, tangencies, and the resulting transitions in the
density can be analyzed explicitly.

\section{Single-vortex dynamics in quasi-periodic channels}\label{sec-two}

We consider the dynamics of a single point vortex in a stationary channel
bounded by two graph interfaces which are genuinely quasi-periodic in the
longitudinal variable. The small-amplitude perturbation of the flat strip
will be considered later.
Let
\begin{equation}
\Omega
=
\left\{
(x,y)\in\R^2:
b_-(x)<y<b_+(x)
\right\},
\label{eq-domain-general}
\end{equation}
where we impose the uniform separation condition
\begin{equation}\label{separation1}
\inf_{x\in\mathbb R}b_+(x)-\sup_{x\in\mathbb R}b_-(x)=h_0>0.
\end{equation}
In particular, this condition guarantees a uniform positive separation between the two interfaces under arbitrary phase shifts.
We assume that the two boundary graphs are quasi-periodic and have the form
\begin{equation}
b_+(x)=f_+(\omega_+x),
\qquad
b_-(x)=f_-(\omega_-x),
\label{eq-qp-boundaries-general}
\end{equation}
where 
\[
\omega_-\in\R^{d_-},
\qquad
\omega_+\in\R^{d_+},
\]
and the functions
\[
f_\pm:\T^{d_\pm}\to\R
\]
are at least continuous  and periodic in each angular variable. We assume that the full frequency vector
\[
\omega=(\omega_-,\omega_+)\in\R^d,
\qquad
d=d_-+d_+,
\]
is nonresonant, namely,
\[
\ell\cdot\omega\neq0
\qquad
\text{for every }\ell\in\Z^d\setminus\{0\}.
\]
Let $G_\Omega(z,\zeta)$ be the Dirichlet Green function normalized by
\[
-\Delta_zG_\Omega(z,\zeta)=\delta_\zeta(z),
\qquad
G_\Omega(z,\zeta)=0,
\quad z\in\partial\Omega,
\]
and define the Robin function
\begin{equation*}
R_\Omega(z)
=
\lim_{\zeta\to z}
\left(
G_\Omega(z,\zeta)
+
\frac1{2\pi}\log|z-\zeta|
\right).
\label{eq-robin-general}
\end{equation*}
For a point vortex of circulation $\Gamma\neq0$, the equation of motion is
\begin{equation}
\dot z
=
\tfrac{\Gamma}{2}\nabla^\perp R_\Omega(z),
\qquad
\nabla^\perp=(\partial_y,-\partial_x),
\label{eq-vortex-general}
\end{equation}
that is, $z(t)=(x(t),y(t))$ satisfies the ODE
\begin{equation*}
\dot x
=
\tfrac{\Gamma}{2}\partial_yR_\Omega(x,y),
\qquad
\dot y
=
-\tfrac{\Gamma}{2}\partial_xR_\Omega(x,y).
\label{eq-vortex-components}
\end{equation*}
The system \eqref{eq-vortex-general} is an autonomous Hamiltonian system,
with Hamiltonian given, up to a multiplicative constant, by the Robin
function \(R_\Omega\). In particular,
\[
\tfrac{d}{dt}R_\Omega(z(t))
=
\tfrac{\Gamma}{2}
\nabla R_\Omega(z(t))
\cdot
\nabla^\perp R_\Omega(z(t))
=
0.
\]
Hence, the Robin function is conserved along the flow, and every
point-vortex trajectory is contained in a connected component of a level
set of \(R_\Omega\). Consequently, the qualitative dynamics of the point
vortex is determined by the geometry and topology of the level sets of
the Robin function. In particular, regular closed components correspond
to periodic trajectories, while critical points and critical level sets
may give rise to equilibria and separatrix-type orbits. This can be summarized in this classical result.

\begin{proposition}
Let $z_0\in\Omega$ and set
\[
E=R_\Omega(z_0).
\]
Then the orbit through $z_0$ is contained in the connected component of
\[
\{z\in\Omega:R_\Omega(z)=E\}.
\]
If this connected component contains no critical point
of $R_\Omega$, then it is a one-dimensional regular curve and the
point-vortex trajectory parametrizes this curve.
More precisely:
\begin{enumerate}
\item critical points of $R_\Omega$ are precisely the equilibrium point-vortex
configurations;
\item a compact regular connected component of a level set is a periodic
vortex orbit;
\item a noncompact regular component is an unbounded vortex orbit;
\item level sets containing critical points may contain separatrices or
homoclinic/heteroclinic connections.
\end{enumerate}
\end{proposition}

\subsection{Exact hull representation of the Robin function}
\label{subsec-exact-hull-Robin}

Let $\Omega$ be the domain defined by \eqref{eq-domain-general}, whose
boundaries satisfy the quasi-periodicity assumption
\eqref{eq-qp-boundaries-general} and the uniform separation condition
\eqref{separation1}. We introduce the associated hull family of domains
\begin{equation}
\label{eq-hull-domain}
\Omega_{\theta}
=
\left\{
(x,y)\in\R^2:
f_-(\omega_-x+\theta_-)<y<
f_+(\omega_+x+\theta_+)
\right\},
\end{equation}
where
\[
\omega:=(\omega_-,\omega_+)\in
\R^{d_-}\times\R^{d_+},
\qquad
\theta:=(\theta_-,\theta_+)
\in
\R^{d_-}\times\R^{d_+}
=
\R^d,\, d=d_{-}+d_+.
\]
Since the hull functions $f_\pm$ are periodic in their phase variables,
the family $\Omega_\theta$ is periodic with respect to
$\theta\in\T^d$.
Associated with the family $\{\Omega_\theta\}_{\theta\in\T^d}$, we
introduce the natural domain
\begin{equation}
\label{eq-domain-Robin-hull}
\mathcal D
:=
\left\{
(\theta,y)\in\T^d\times\R:
f_-(\theta_-)<y<f_+(\theta_+)
\right\}.
\end{equation}
By construction, the domain \(\mathcal D\subset\T^d\times\mathbb R\) is periodic in the angular variable $\theta;$ equivalently, its lift to $\mathbb R^d\times\mathbb R$ is $2\pi\mathbb Z^d$-periodic.
Notice that
\[
(\theta,y)\in\mathcal D
\quad\Longleftrightarrow\quad
(0,y)\in\Omega_\theta.
\]
We define the Robin hull function by
\begin{equation}
\label{eq-robin-hull}
\mathscr R(\theta,y)
:=
R_{\Omega_\theta}(0,y),
\qquad
(\theta,y)\in\mathcal D,
\end{equation}
where $R_{\Omega_\theta}$ denotes the Robin function associated with
the domain $\Omega_\theta$.
\\
The purpose of introducing the family \eqref{eq-hull-domain} is to
replace the quasi-periodic dependence on the unbounded spatial variable
$x$ by a periodic dependence on the compact phase space $\T^d$.
Horizontal translations of the original quasi-periodic geometry are
therefore encoded by translations of the phase variable $\theta$.

\begin{proposition}
\label{prop-exact-hull-Robin}
The Robin hull function
\[
\mathscr R:\mathcal D\longrightarrow\R
\]
defined by \eqref{eq-robin-hull} is periodic with respect to each component of $\theta$. Moreover, for
every $\theta\in\T^d$ and every $(x,y)\in\Omega_\theta$, one has
\begin{equation}
\label{eq-robin-hull-shift}
R_{\Omega_\theta}(x,y)
=
\mathscr R(\theta+x\omega,y).
\end{equation}
In particular,
\begin{equation}
\label{eq-robin-exact-hull}
R_\Omega(x,y)
=
\mathscr R(x\omega,y),
\qquad
(x,y)\in\Omega.
\end{equation}
Consequently, the dependence of the Robin function on the longitudinal
variable is obtained by restricting the periodic hull function
$\mathscr R$ to the linear flow
\[
x\mapsto x\omega
\]
on $\T^d$. In particular, along every horizontal line contained in the
channel, the map $x\mapsto R_\Omega(x,y)$ is quasi-periodic with
frequency vector $\omega$.
\end{proposition}

\begin{proof}
We first establish the relation between horizontal translations in the
physical domain and translations on the phase torus. Let
{$\theta=(\theta_-,\theta_+)\in\R^d$} and $a\in\R$. By definition,
$(X,Y)\in\Omega_\theta-ae_1$ if and only if
$(X+a,Y)\in\Omega_\theta$. Hence,
\[
f_-\bigl(\omega_-(X+a)+\theta_-\bigr)
<
Y
<
f_+\bigl(\omega_+(X+a)+\theta_+\bigr).
\]
Equivalently,
\[
f_-\bigl(\omega_-X+\theta_-+a\omega_-\bigr)
<
Y
<
f_+\bigl(\omega_+X+\theta_++a\omega_+\bigr).
\]
Therefore,
\begin{equation}
\label{eq-domain-phase-shift}
\Omega_\theta-ae_1
=
\Omega_{\theta+a\omega}.
\end{equation}
Thus a horizontal translation by $a$ in physical space is exactly
equivalent to a translation by $a\omega$ in the phase variables.
\\
We next use the translation covariance of the Green and Robin
functions. If $D\subset\R^2$ is a domain and $a\in\R^2$, then the
Dirichlet Green function satisfies
\[
G_{D-a}(z-a,\zeta-a)
=
G_D(z,\zeta).
\]
Passing to the regular part on the diagonal gives the corresponding
identity for the Robin function,
\begin{equation}
\label{eq-Robin-translation-covariance}
R_{D-a}(z-a)
=
R_D(z).
\end{equation}
We apply \eqref{eq-Robin-translation-covariance} with
\[
D=\Omega_\theta,
\qquad
a=xe_1,
\qquad
z=(x,y).
\]
It follows that
\[
R_{\Omega_\theta}(x,y)
=
R_{\Omega_\theta-xe_1}(0,y).
\]
Using the phase-shift identity \eqref{eq-domain-phase-shift}, we obtain
\[
R_{\Omega_\theta}(x,y)
=
R_{\Omega_{\theta+x\omega}}(0,y).
\]
Since $(x,y)\in\Omega_\theta$, we have
\[
f_-(\omega_-x+\theta_-)
<
y
<
f_+(\omega_+x+\theta_+).
\]
Equivalently,
\[
(\theta+x\omega,y)\in\mathcal D.
\]
Therefore, by the definition \eqref{eq-robin-hull} of the Robin hull
function,
\[
R_{\Omega_\theta}(x,y)
=
\mathscr R(\theta+x\omega,y),
\]
which proves \eqref{eq-robin-hull-shift}.
It remains to verify the periodicity in the phase variables. Since
$f_\pm$ are periodic on their respective tori, for every
$m\in\Z^d$ one has
\[
\Omega_{\theta+2\pi m}
=
\Omega_\theta,
\]
where we use the convention
$\T^d=\R^d/(2\pi\Z)^d$. Moreover,
\[
(\theta,y)\in\mathcal D
\quad\Longleftrightarrow\quad
(\theta+2\pi m,y)\in\mathcal D.
\]
Consequently,
\[
\mathscr R(\theta+2\pi m,y)
=
R_{\Omega_{\theta+2\pi m}}(0,y)
=
R_{\Omega_\theta}(0,y)
=
\mathscr R(\theta,y).
\]
Thus $\mathscr R$ is periodic in the phase variables and is
well defined on $\mathcal D$.\\
Finally, taking $\theta=0$ in \eqref{eq-robin-hull-shift} and using
$\Omega_0=\Omega$ yields
\[
R_\Omega(x,y)
=
\mathscr R(x\omega,y),
\]
which is \eqref{eq-robin-exact-hull}. Since $\mathscr R$ is periodic
in its phase variable, this identity gives the desired exact
quasi-periodic representation of the Robin function.
\end{proof}

The proposition shows that the quasi-periodicity of the geometry is
inherited exactly by the Robin function throughout the channel: the
spatial dependence of $R_\Omega$ is obtained by restricting the
periodic hull function $\mathscr R$ to the linear flow
$x\mapsto x\omega$ on $\T^d$.

\subsection{Regularity of the Robin hull function}

We now investigate the regularity of the Robin hull function with
respect to both the phase variable $\theta$ and the vertical variable
$y$. Although the Robin function is smooth in the interior of each
fixed domain $\Omega_\theta$, its dependence on $\theta$ is governed
by the regularity of the moving boundary. The natural way to analyze
this dependence is to flatten the family
$\{\Omega_\theta\}_{\theta\in\mathbb T^d}$ onto a fixed strip. The
resulting elliptic operator has coefficients involving one derivative
of the boundary parametrization. This leads naturally to the loss of
one derivative with respect to the phase variables.
A second point has to be taken into account. After flattening, the
Robin function is obtained by evaluating the regular part of the Green
function at a point of the fixed strip which itself depends on
$(\theta,y)$. Thus, one has to control simultaneously the dependence
of the pulled-back Green function on the parameters $(\theta,y)$ and
its elliptic regularity with respect to the spatial variable. Since
the Robin hull is now defined on its natural domain $\mathcal D$, the
corresponding regularity statement is local in $\mathcal D$. Our
result reads as follows.

\begin{proposition}
\label{prop-regularity-Robin-hull}
Let $m\geq1$ be an integer and $\alpha\in(0,1)$. Assume that
\begin{align}
\label{reg-m}
f_\pm\in C^{m,\alpha}(\mathbb T^{d_\pm}),
\end{align}
and that the uniform separation condition \eqref{separation1} holds.
Then the Robin hull function defined in \eqref{eq-robin-hull}
satisfies
\begin{equation}
\label{eq-Robin-hull-regularity}
\mathscr R\in C_{\mathrm{loc}}^{m-1}(\mathcal D).
\end{equation}
\end{proposition}

\begin{proof}
We divide the proof into two steps.

\medskip
\noindent
\textbf{Step 1.} \textit{Flattening the family of domains.}
Introduce the flat horizontal strip
\[
S:=\mathbb R\times(0,1).
\]
For every $\theta=(\theta_-,\theta_+)\in\mathbb T^d$, define
\[
\Phi_\theta:S\longrightarrow\Omega_\theta
\]
by
\begin{equation}
\label{eq-flattening-Robin}
\Phi_\theta(x,s)
=
\left(
x,\,
(1-s)f_-(\theta_-+\omega_-x)
+
sf_+(\theta_++\omega_+x)
\right).
\end{equation}
By setting
\[
h_\theta(x)
:=
f_+(\theta_++\omega_+x)
-
f_-(\theta_-+\omega_-x),
\]
we get  from direct computations
\[
\det D\Phi_\theta(x,s)=h_\theta(x).
\]
Using the uniform separation assumption \eqref{separation1},
\[
h_\theta(x)\geq h_0>0,
\qquad
(\theta,x)\in\mathbb T^d\times\mathbb R.
\]
Consequently, $\Phi_\theta$ is a diffeomorphism from $S$ onto
$\Omega_\theta$, uniformly with respect to $\theta$.
\\
Since
\[
f_\pm\in C^{m,\alpha},
\]
the flattening maps depend $C^m$ on the phase variables, locally in the
spatial variable. More precisely,
\[
\theta\longmapsto\Phi_\theta
\]
is $C^m$ with values in the appropriate local Hölder spaces.
\\
Under the change of variables $z=\Phi_\theta(q)$, the Laplace equation
is transformed into a divergence-form equation on the fixed strip $S$.
The corresponding coefficient matrix is
\begin{equation}
\label{eq-A-theta-flattening}
A_\theta
=
J_\theta
D\Phi_\theta^{-1}D\Phi_\theta^{-T},
\qquad
J_\theta:=\det D\Phi_\theta.
\end{equation}
Since $A_\theta$ depends algebraically on $D\Phi_\theta$ and
$(D\Phi_\theta)^{-1}$, it follows that
\(
\theta\longmapsto A_\theta
\)
is of class $C^{m-1}$, locally with values in
$C^{0,\alpha}(\overline S)$. Moreover, the family $A_\theta$ is
uniformly elliptic:
\[
\lambda|\xi|^2
\leq
A_\theta(q)\xi\cdot\xi
\leq
\Lambda|\xi|^2
\]
for some constants $0<\lambda\leq\Lambda$ independent of
$\theta$, $q$, and $\xi$.\\
This is precisely where one derivative is lost: the coefficients of
the pulled-back elliptic operator involve the first derivatives of the
boundary parametrization.

\medskip
\noindent
\textbf{Step 2.} \textit{Regularity of the pulled-back Green function.}\\
Let $G_\theta(z,\zeta)$ denote the Dirichlet Green function of
$\Omega_\theta$. We use the decomposition
\begin{equation}
\label{eq-Green-regular-decomposition}
G_\theta(z,\zeta)
=
-\tfrac{1}{2\pi}\log|z-\zeta|
+
H_\theta(z,\zeta),
\end{equation}
where $H_\theta$ denotes the regular part. Thus, by definition
\begin{equation}
\label{eq-Robin-H-diagonal}
R_{\Omega_\theta}(\zeta)
=
H_\theta(\zeta,\zeta).
\end{equation}
Fix a compact set
\(
\mathcal K\Subset\mathcal D.
\)
Then, there exists
$\delta_{\mathcal K}>0$ such that
\begin{equation}
\label{eq-uniform-interior-hull}
\operatorname{dist}
\bigl((0,y),\partial\Omega_\theta\bigr)
\geq
\delta_{\mathcal K},
\qquad
(\theta,y)\in\mathcal K.
\end{equation}
Thus, although the Robin hull is defined on the whole of $\mathcal D$,
its restriction to $\mathcal K$ enjoys the same uniform interior
separation that was previously imposed globally.
For $(\theta,y)\in\mathcal K$, the function
\[
z\longmapsto H_\theta(z,(0,y))
\]
solves
\begin{equation}
\label{eq-H-physical-problem}
\begin{cases}
\Delta_z H_\theta(z,(0,y))=0,
&z\in\Omega_\theta,\\[1mm]
\displaystyle
H_\theta(z,(0,y))
=
\tfrac{1}{2\pi}\log|z-(0,y)|,
&z\in\partial\Omega_\theta.
\end{cases}
\end{equation}
Set
\[
\widetilde H_{\theta,y}(q)
:=
H_\theta\bigl(\Phi_\theta(q),(0,y)\bigr).
\]
Pulling back \eqref{eq-H-physical-problem} through $\Phi_\theta$ gives
the fixed-domain problem
\begin{equation}
\label{eq-H-pulled-back-problem}
\begin{cases}
\operatorname{div}
\bigl(A_\theta\nabla\widetilde H_{\theta,y}\bigr)=0,
&q\in S,\\[1mm]
\displaystyle
\widetilde H_{\theta,y}(q)
=
g_{\theta,y}(q),
&q\in\partial S,
\end{cases}
\end{equation}
where
\begin{equation}
\label{eq-boundary-data-H}
g_{\theta,y}(q)
:=
\tfrac{1}{2\pi}
\log\left|\Phi_\theta(q)-(0,y)\right|.
\end{equation}
By \eqref{eq-uniform-interior-hull}, there exists
$c_{\mathcal K}>0$ such that
\[
\left|\Phi_\theta(q)-(0,y)\right|
\geq c_{\mathcal K},
\qquad
q\in\partial S,\quad
(\theta,y)\in\mathcal K.
\]
Consequently, no singularity occurs in the boundary datum
\eqref{eq-boundary-data-H}. Since $\Phi_\theta$ depends $C^m$ on
$\theta$, it follows that the map
\[
(\theta,y)\longmapsto g_{\theta,y}
\]
is of class $C^m$ with respect to the parameter variables, locally
uniformly for $(\theta,y)\in\mathcal K$.
\\
Next, we identify the point of the fixed strip corresponding to
$(0,y)$. We have
\[
\Phi_\theta(0,s)
=
\left(
0,\,
(1-s)f_-(\theta_-)+sf_+(\theta_+)
\right).
\]
Hence the unique $s=\sigma(\theta,y)$ satisfying
\[
\Phi_\theta(0,\sigma(\theta,y))=(0,y)
\]
is given explicitly by
\begin{equation}
\label{eq-sigma-Robin}
\sigma(\theta,y)
=
\tfrac{y-f_-(\theta_-)}
{f_+(\theta_+)-f_-(\theta_-)}.
\end{equation}
Using \eqref{eq-Robin-H-diagonal} and the definition of
$\widetilde H_{\theta,y}$, we therefore obtain
\begin{align}
\mathscr R(\theta,y)
&=
R_{\Omega_\theta}(0,y)
\nonumber\\
&=
H_\theta((0,y),(0,y))
\nonumber\\
&=
\widetilde H_{\theta,y}
\bigl(0,\sigma(\theta,y)\bigr).
\label{eq-Robin-hull-evaluation}
\end{align}
Thus, the regularity of $\mathscr R$ follows from that of
$\widetilde H_{\theta,y}$. It is useful to distinguish between two
different types of regularity. The map
\(
(\theta,y)\longmapsto\widetilde H_{\theta,y}
\)
describes the dependence of the solution on the parameters
$(\theta,y)$, whereas
\(
q\longmapsto\widetilde H_{\theta,y}(q)
\)
describes its elliptic regularity with respect to the spatial variable
$q\in S$.
For every fixed $(\theta,y)\in\mathcal K$, the function
$\widetilde H_{\theta,y}$ solves \eqref{eq-H-pulled-back-problem}.
Since
\[
f_\pm\in C^{m,\alpha}(\mathbb T^{d_\pm}),
\]
the coefficients of the pulled-back operator satisfy
\[
A_\theta
\in
C_{\mathrm{loc}}^{m-1,\alpha}(S)
\]
uniformly with respect to $\theta\in\mathbb T^d$. Therefore, by
interior elliptic regularity,
\begin{equation}
\label{eq-H-spatial-regularity}
\widetilde H_{\theta,y}
\in
C_{\mathrm{loc}}^{m,\alpha}(S).
\end{equation}
This spatial regularity has to be distinguished from the regularity
with respect to the parameters. The latter concerns the map
\[
(\theta,y)
\longmapsto
\widetilde H_{\theta,y}|_K
\]
with values in a suitable Hölder space on $K\Subset S$.
\\
For $m=1$, the coefficients $A_\theta$ and the boundary data
$g_{\theta,y}$ depend continuously on $(\theta,y)\in\mathcal K$.
In this case, the flattening maps satisfy
\[
\Phi_\theta\in C_{\mathrm{loc}}^{1,\alpha}(S),
\]
and the coefficient matrices
\[
A_\theta
=
J_\theta
D\Phi_\theta^{-1}D\Phi_\theta^{-T}
\]
belong to $C_{\mathrm{loc}}^{0,\alpha}(\overline S)$ and are uniformly
elliptic. By the continuous dependence of solutions of uniformly
elliptic Dirichlet problems with respect to the coefficients and the
boundary data, it follows that
\[
(\theta,y)
\longmapsto
\widetilde H_{\theta,y}
\]
is continuous on $\mathcal K$, locally with respect to the spatial
variable. Since $\sigma$ is continuous on $\mathcal K$, the evaluation
formula \eqref{eq-Robin-hull-evaluation} yields
\[
\mathscr R\in C^0(\mathcal K).
\]
Since $\mathcal K\Subset\mathcal D$ is arbitrary, this proves
\[
\mathscr R\in C_{\mathrm{loc}}^0(\mathcal D)
\]
when $m=1$.
Assume now that $m\geq2$. The higher-order parameter regularity follows
by differentiating the elliptic problem with respect to
\[
p:=(\theta,y).
\]
For a multi-index $\beta$ satisfying
\[
1\leq|\beta|\leq m-1,
\]
set
\[
U_\beta
:=
\partial_p^\beta\widetilde H_{\theta,y}.
\]
Differentiating the elliptic equation in \eqref{eq-H-pulled-back-problem}
gives
\begin{equation}
\label{eq-differentiated-H}
\operatorname{div}
\bigl(A_\theta\nabla U_\beta\bigr)
=
-
\operatorname{div}
\Bigg(
\sum_{\substack{0<\gamma\leq\beta}}
\binom{\beta}{\gamma}
(\partial_p^\gamma A_\theta)
\nabla
\partial_p^{\beta-\gamma}
\widetilde H_{\theta,y}
\Bigg).
\end{equation}
Since $A_\theta$ is independent of $y$, only derivatives in the
$\theta$ variables contribute to $\partial_p^\gamma A_\theta$. On the
boundary,
\begin{equation}
\label{eq-differentiated-H-boundary}
U_\beta
=
\partial_p^\beta g_{\theta,y}
\qquad\text{on }\partial S.
\end{equation}
Since
\(
\partial_\theta^\gamma A_\theta
\)
exists and depends continuously on $\theta$ for every
$|\gamma|\leq m-1$, we may apply interior Schauder estimates to the
elliptic equation \eqref{eq-differentiated-H}. More precisely,
if
\[
\operatorname{div}(A\nabla u)=\operatorname{div}F
\qquad\text{in }S,
\]
where $A$ is uniformly elliptic and
\[
A,F\in C_{\mathrm{loc}}^{k,\alpha}(S),
\]
then, for every $K\Subset K'\Subset S$,
\begin{equation}
\label{eq-divergence-Schauder}
\|u\|_{C^{k+1,\alpha}(K)}
\leq
C
\left(
\|u\|_{L^\infty(K')}
+
\|F\|_{C^{k,\alpha}(K')}
\right).
\end{equation}
Applying \eqref{eq-divergence-Schauder} to
\eqref{eq-differentiated-H} and arguing inductively with respect to
$|\beta|$, we obtain, for every 
$|\beta|\leq m-1$,
\begin{equation}
\label{eq-H-mixed-regularity}
\partial_{(\theta,y)}^\beta\widetilde H_{\theta,y}
\in
C_{\mathrm{loc}}^{m-|\beta|,\alpha}(S),
\end{equation}
with continuous dependence on $(\theta,y)\in\mathcal K$. Hence, for
every compact set $K\Subset S$, the parameter derivatives
\[
(\theta,y)
\longmapsto
\partial_{(\theta,y)}^\beta
\widetilde H_{\theta,y}\big|_K
\]
are continuous with values in
$C^{m-|\beta|,\alpha}(K)$.
\\
The uniform separation condition \eqref{separation1} guarantees that
the denominator in \eqref{eq-sigma-Robin} is bounded away from zero.
Consequently,
\[
\sigma\in C_{\mathrm{loc}}^{m,\alpha}(\mathcal D).
\]
Moreover, since $\mathcal K\Subset\mathcal D$, there exists
$\eta_{\mathcal K}>0$ such that
\[
\eta_{\mathcal K}
\leq
\sigma(\theta,y)
\leq
1-\eta_{\mathcal K},
\qquad
(\theta,y)\in\mathcal K.
\]
Thus the evaluation points
\[
\bigl(0,\sigma(\theta,y)\bigr),
\qquad
(\theta,y)\in\mathcal K,
\]
remain in a compact subset of $S$. Combining this fact with
\eqref{eq-H-mixed-regularity} and the evaluation formula
\eqref{eq-Robin-hull-evaluation}, we conclude that
\[
\mathscr R\in C^{m-1}(\mathcal K).
\]
Since $\mathcal K\Subset\mathcal D$ was arbitrary, we finally obtain
\[
{
\mathscr R\in C_{\mathrm{loc}}^{m-1}(\mathcal D).
}
\]
This completes the proof.
\end{proof}

\section{Quasi-periodic structure of the level sets}\label{sec-03}

We now turn from the construction and regularity of the Robin hull to
the dynamics it generates. The analysis naturally separates into two
parts: the geometry of the invariant level sets and the time
parametrization of the motion along them. The first is governed by a
transversality condition and does not involve any arithmetic assumption
on the frequencies. The second leads to a cohomological equation on
$\T^d$, for which a Diophantine condition is introduced to control the
small divisors. This distinction between spatial and temporal
quasi-periodicity is the main theme of this section.
\subsection{Quasi-periodic invariant graphs and regularity}\label{sec-3.1}

The hull representation developed in the previous section allows us to describe the level sets of the Robin function through a periodic function on the compact phase space $\T^d$. Under a suitable transversality condition, these level sets can be represented as global periodic graphs whose restrictions to the linear flow $\theta=x\omega$ yield quasi-periodic invariant graphs in the physical channel. The purpose of this subsection is to establish this geometric representation and to investigate the regularity of the resulting graphs. In particular, the construction is purely geometric and does not require any arithmetic assumption on the frequency vector $\omega$.
The following theorem makes this construction precise.
Before stating it,
we introduce some notation adapted to the natural domain $\mathcal D$ of
the Robin hull. Let
\[
a,b:\T^d\longrightarrow\R
\]
be two periodic functions such that
\begin{equation}
\label{eq-admissible-variable-strip}
f_-(\theta_-)
<
a(\theta)
<
b(\theta)
<
f_+(\theta_+),
\qquad
\theta=(\theta_-,\theta_+)\in\T^d.
\end{equation}
For every $\theta\in\T^d$, we define the interval
\[
J_\theta
:=
[a(\theta),b(\theta)]
\]
and set
\[
\mathscr R(\theta,J_\theta)
:=
\left\{
\mathscr R(\theta,y):y\in J_\theta
\right\}.
\]
We then introduce the common energy range
\begin{equation}
\label{eq-common-energy-range}
\mathcal E_{a,b}^{\mathrm{glob}}
:=
\bigcap_{\theta\in\T^d}
\mathscr R(\theta,J_\theta).
\end{equation}
This formulation allows the vertical interval $J_\theta$ to vary with
the phase and, in particular, to follow either boundary of the channel.
For example, choosing
\[
a(\theta)=f_-(\theta_-)+\delta,
\qquad
b(\theta)=f_-(\theta_-)+2\delta,
\]
with $\delta>0$ sufficiently small, describes a region lying at distance
of order $\delta$ from the lower boundary. An analogous choice can be
made near the upper boundary.
The main result reads as follows.
\begin{theorem}
\label{thm-exact-qp-orbits}
Let $f_\pm$ satisfy \eqref{reg-m} with $m\geq2$, together with \eqref{separation1}, and let
\[
a,b\in C^{m-1}(\T^d)
\]
satisfy \eqref{eq-admissible-variable-strip}. Assume that
\[
\mathcal E_{a,b}^{\mathrm{glob}}\neq\varnothing
\]
and that the following transversality condition holds: there exists $c_0>0$ such that
\begin{equation}
\label{eq-uniform-transversality-general}
\left|
\partial_y\mathscr R(\theta,y)
\right|
\geq c_0,
\qquad
\theta\in\T^d,
\quad
y\in J_\theta.
\end{equation}
Then, for every $E\in \mathcal E_{a,b}^{\mathrm{glob}}$,
there exists a unique periodic function
\(
\mathscr Y
\in
C^{m-1}(\T^d)
\)
such that
\begin{equation}
\label{eq-hull-level-equation-general}
a(\theta)
\leq
\mathscr Y(\theta)
\leq
b(\theta),
\qquad
\mathscr R
\bigl(
\theta,
\mathscr Y(\theta)
\bigr)
=
E,
\qquad
\theta\in\T^d.
\end{equation}
In addition, the function
\begin{equation}
\label{eq-exact-qp-orbit-general}
Y(x)
:=
\mathscr Y(x\omega)
\end{equation}
is quasi-periodic in $x$ and satisfies
\begin{equation}
\label{eq-physical-level-set-general}
R_\Omega(x,Y(x))
=
E,
\qquad
x\in\R.
\end{equation}
Moreover, the graph
\[
\Gamma_E
=
\left\{
(x,Y(x)):x\in\R
\right\}
\]
is invariant under the point-vortex flow
\begin{align}
\label{Orbit-th}
\dot z
=
\tfrac{\Gamma}{2}
\nabla^\perp R_\Omega(z),
\qquad
z(0)\in\Gamma_E.
\end{align}
\end{theorem}

\begin{proof}
Consider the variable strip
\[
\mathcal J_{a,b}
:=
\left\{
(\theta,y)\in\T^d\times\R:
a(\theta)<y<b(\theta)
\right\}.
\]
By \eqref{eq-admissible-variable-strip},
\[
\mathcal J_{a,b}\subset\mathcal D.
\]
Moreover, since $a$ and $b$ are continuous and $a<b$, the set
$\mathcal J_{a,b}$ is connected.
Set
\[
F(\theta,y)
:=
\mathscr R(\theta,y)-E,\,\, (\theta,y)\in\mathcal{D}.
\]
By Proposition~\ref{prop-regularity-Robin-hull},
\[
\mathscr R
\in
C_{\mathrm{loc}}^{m-1}(\mathcal D).
\]
Since the closed variable strip \(\overline{\mathcal J}_{a,b}\)
is a compact subset of $\mathcal D$, it follows that
\[
\mathscr R
\in
C^{m-1}(\overline{\mathcal J}_{a,b}).
\]
In particular, since $m\geq2$, the derivative
$\partial_y\mathscr R$ is continuous there.
\\
By the transversality assumption \eqref{eq-uniform-transversality-general},
$\partial_y\mathscr R$ never vanishes on
$\overline{\mathcal J}_{a,b}$. Since $\mathcal J_{a,b}$ is connected,
$\partial_y\mathscr R$ has a constant sign. Consequently, for every
fixed $\theta\in\T^d$, the map
\[
y\longmapsto\mathscr R(\theta,y)
\]
is strictly monotone on $J_\theta$.
Now, pick  $E\in \mathcal E_{a,b}^{\mathrm{glob}}$. Then, for every
$\theta\in\T^d$,
\[
E
\in
\mathscr R
\bigl(
\theta,[a(\theta),b(\theta)]
\bigr).
\]
Hence, there exists
\[
\exists y\in[a(\theta),b(\theta)]
\quad\hbox{
such that}\quad
F(\theta,y)=0.
\]
Strict monotonicity implies that this zero is unique and denoted by $\mathscr Y(\theta)$. This defines a function
\[
\mathscr Y:\T^d\longrightarrow\R
\]
such that
\[
a(\theta)
\leq
\mathscr Y(\theta)
\leq
b(\theta)
\]
and
\[
F
\bigl(
\theta,\mathscr Y(\theta)
\bigr)
=
0,
\qquad
\forall\,\theta\in\T^d.
\]
Moreover,
\[
\left|
\partial_yF
\bigl(
\theta,\mathscr Y(\theta)
\bigr)
\right|
=
\left|
\partial_y\mathscr R
\bigl(
\theta,\mathscr Y(\theta)
\bigr)
\right|
\geq c_0.
\]
The implicit function theorem therefore gives a local
$C^{m-1}$ representation of the solution around every
$\theta\in\T^d$. Since the solution is unique in each fiber
$J_\theta$, these local representations agree on their overlaps and
define a global periodic function
\(
\mathscr Y\in C^{m-1}(\T^d).
\)
Differentiating the equation
\[
\mathscr R
\bigl(
\theta,\mathscr Y(\theta)
\bigr)
=
E
\]
with respect to $\theta_j$, we obtain
\begin{equation}
\label{eq-derivative-Y-hull}
\partial_{\theta_j}\mathscr Y(\theta)
=
-
\frac{
\partial_{\theta_j}\mathscr R
\bigl(
\theta,\mathscr Y(\theta)
\bigr)
}{
\partial_y\mathscr R
\bigl(
\theta,\mathscr Y(\theta)
\bigr)
},
\qquad
j=1,\ldots,d.
\end{equation}
We now return to the physical channel. By the exact hull representation
established in Proposition~\ref{prop-exact-hull-Robin},
\[
R_\Omega(x,y)
=
\mathscr R(x\omega,y),
\qquad
(x,y)\in\Omega.
\]
Setting
\[
Y(x)
:=
\mathscr Y(x\omega),
\]
we have, by \eqref{eq-admissible-variable-strip},
\[
f_-(\omega_-x)
<
a(x\omega)
\le
Y(x)
\le
b(x\omega)
<
f_+(\omega_+x),
\]
and therefore
\[
(x,Y(x))\in\Omega,
\qquad
x\in\R.
\]
Furthermore,
\[
R_\Omega(x,Y(x))
=
\mathscr R
\bigl(
x\omega,
\mathscr Y(x\omega)
\bigr)
=
E.
\]
Hence the global graph
\[
\Gamma_E
=
\left\{
(x,Y(x)):x\in\R
\right\}
\]
is contained in the level set
$\{R_\Omega=E\}$.\\
It remains to prove its invariance by the point vortex flow. 
Let $z(t)=(x(t),y(t))$ be a solution of the point-vortex equation \eqref{Orbit-th} with initial condition
\[
z(0)=\bigl(x_0,Y(x_0)\bigr)\in\Gamma_E.
\]
Along the trajectory, we have
\begin{align*}
\frac{d}{dt}R_\Omega(z(t))
&=
\nabla R_\Omega(z(t))\cdot\dot z(t)
\\
&=
\frac{\Gamma}{2}
\nabla R_\Omega(z(t))
\cdot
\nabla^\perp R_\Omega(z(t))
\\
&=0.
\end{align*}
Hence
\[
R_\Omega(z(t))
=
R_\Omega(z(0))
=
E.
\]
Therefore the trajectory remains in the level set $\{R_\Omega=E\}$. Since, in the region under consideration, this level set is uniquely represented by the graph
\[
\Gamma_E
=
\left\{
(x,Y(x)):x\in\mathbb R
\right\},
\]
we necessarily have
\[
y(t)=Y(x(t)).
\]
Consequently,
\[
z(t)=\bigl(x(t),Y(x(t))\bigr)\in\Gamma_E,
\]
and thus $\Gamma_E$ is invariant under the point-vortex flow \eqref{Orbit-th}.
 This completes the proof.
\end{proof}

\subsection{Diophantine frequencies and quasi-periodic dynamics}\label{sec-3.2}

The previous subsection provides a geometric description of the
point-vortex trajectories as invariant quasi-periodic graphs. The next
natural question is whether this spatial quasi-periodicity is also
reflected in the time evolution of the vortex. This is not immediate:
even though the trajectory lies on a quasi-periodic graph, the vortex
does not move along this graph with constant horizontal velocity.
Consequently, the quasi-periodicity of the orbit as a set does not
directly imply quasi-periodicity of its time parametrization.\\
The purpose of this subsection is to show that the time evolution can
nevertheless be reduced to a uniform horizontal translation, up to a
quasi-periodic correction. More precisely, we shall prove that the
motion admits a decomposition into a linear drift and a bounded
quasi-periodic oscillation. The frequency vector of this oscillation is
obtained from the spatial frequencies of the channel after multiplication
by the effective drift velocity.\\
As we shall see below, the argument relies on the fact that, along an invariant graph, the
point-vortex system reduces to a scalar equation for the horizontal
position. The transversality condition ensures that the corresponding
horizontal velocity never vanishes, so that the motion is monotone and
can be reparametrized. Straightening this nonconstant quasi-periodic
velocity then leads naturally to a cohomological equation on the torus.\\
It is precisely at this stage that an arithmetic condition on the
frequency vector enters the analysis. Indeed, solving the cohomological
equation involves a small-divisor problem. We therefore impose a
Diophantine condition on $\omega$, which provides a quantitative lower
bound on the corresponding divisors. Combined with sufficient regularity
of the channel boundaries, this condition allows us to solve the
cohomological equation with a controlled loss of derivatives and,
consequently, to obtain a quasi-periodic correction with the desired
regularity.\\
This reduction shows, in particular, that the geometry of the
quasi-periodic channel determines not only the spatial structure of the
invariant trajectories, but also their temporal frequency content. The
following theorem gives the resulting quasi-periodic decomposition of
the point-vortex dynamics modulo horizontal translation.

\begin{theorem}
\label{thm-exact-qp-orbits-1}
Under the assumptions of Theorem~\ref{thm-exact-qp-orbits}, assume in
addition that the frequency vector $\omega$ is Diophantine, namely, that
there exist constants $\gamma>0$ and $\tau>d-1$ such that
\begin{equation}
\label{eq-Diophantine-k-orbit}
|\omega\cdot\ell|
\geq
\frac{\gamma}{\langle\ell\rangle^\tau},
\qquad
\forall\,\ell\in\Z^d\setminus\{0\}.
\end{equation}
Let $r\geq1$ be an integer and assume that
\[
f_\pm\in C^{m,\alpha}(\mathbb T^{d_\pm}),
\qquad
\alpha\in(0,1),
\]
for some integer $m$ satisfying
\begin{equation}
\label{eq-regularity-threshold-time}
m>r+\tau+d+2.
\end{equation}
Then every point-vortex trajectory contained in $\Gamma_E$ is
quasi-periodic modulo a horizontal translation. More precisely, for
every initial point $z(0)\in\Gamma_E$, the corresponding orbit of
\eqref{Orbit-th} can be parametrized as
\begin{equation}
\label{eq-point-vortex-QP-modulo-translation}
z(t)
=
c_\star t\,e_1
+
Z_\star(c_\star t\omega),
\end{equation}
where
\begin{equation*}
\label{eq-effective-drift-theorem}
c_\star
:=
\left(
\fint_{\T^d}
\frac{d\theta}{\mathscr V(\theta)}
\right)^{-1},
\qquad
\mathscr V(\theta)
:=
\tfrac{\Gamma}{2}
\partial_y\mathscr R
\bigl(
\theta,\mathscr Y(\theta)
\bigr),
\end{equation*}
and $Z_\star$ is a periodic function on $\T^d$ with 
\[
Z_\star\in C^r(\T^d;\R^2).
\]
 Furthermore, for every
$\Psi\in C(\mathbb R^2)$,
\begin{equation*}
\label{eq-ergodic-boundary-orbit}
\lim_{T\to\infty}
\frac1T\int_0^T
\Psi\bigl(Z_{\star}(c_\star\omega t)\bigr)\,dt
=
\fint_{\mathbb T^d}
\Psi\bigl(Z_\star(\theta)\bigr)\,d\theta.
\end{equation*}
Throughout the paper, we use the notation
$$
\fint_{\T^d}:=\frac1{(2\pi)^d}
\int_{\mathbb T^d}$$
\end{theorem}

\begin{proof}
Let
\[
z(t)=(x(t),y(t))
\]
be a solution of \eqref{Orbit-th} contained in $\Gamma_E$. Then
\begin{equation}
\label{eq-point-vortex-components}
\dot x(t)
=
\tfrac{\Gamma}{2}
\partial_yR_\Omega(x(t),y(t)),
\qquad
\dot y(t)
=
-\tfrac{\Gamma}{2}
\partial_xR_\Omega(x(t),y(t)).
\end{equation}
On the invariant graph
\[
y=Y(x),
\]
we therefore obtain the scalar equation
\begin{equation}
\label{eq-scalar-flow-on-graph}
\dot x
=
V(x),
\qquad
V(x)
:=
\tfrac{\Gamma}{2}
\partial_yR_\Omega(x,Y(x)).
\end{equation}
Using the hull representation of the Robin function, define
\begin{equation}
\label{eq-V-hull}
\mathscr V(\theta)
:=
\tfrac{\Gamma}{2}
\partial_y\mathscr R
\bigl(
\theta,\mathscr Y(\theta)
\bigr).
\end{equation}
Then, by \eqref{eq-robin-exact-hull} and
\eqref{eq-exact-qp-orbit-general},
\begin{equation}
\label{eq-V-QP}
V(x)
=
\mathscr V(x\omega).
\end{equation}
Moreover, by the transversality assumption
\eqref{eq-uniform-transversality-general},
\[
|\mathscr V(\theta)|
\geq
\tfrac{|\Gamma|}{2}c_0,
\qquad
\forall\,\theta\in\T^d.
\]
Since $\mathscr V$ is continuous and does not vanish on the connected
torus $\T^d$, it has a constant sign. Hence the horizontal component
$x(t)$ is strictly monotone.
\\
We now straighten the scalar quasi-periodic vector field
\[
\dot x=\mathscr V(x\omega).
\]
Introduce the effective drift
\begin{equation}
\label{eq-effective-drift-c}
c_\star
:=
\left(
\fint_{\T^d}
\frac{d\theta}{\mathscr V(\theta)}
\right)^{-1}.
\end{equation}
Since $\mathscr V$ has a constant sign and is bounded away from zero,
the quantity $c_\star$ is well defined, nonzero, and has the same sign
as $\mathscr V$.
Introduce the periodic real-valued function
\begin{equation}
\label{eq-g-cohomological}
g(\theta)
:=
\frac{c_\star}{\mathscr V(\theta)}-1.
\end{equation}
By the definition of $c_\star$,
\[
\fint_{\T^d}g(\theta)\,d\theta=0.
\]
We solve the cohomological equation
\begin{equation}
\label{eq-cohomological-orbit}
\omega\cdot\partial_\theta u(\theta)
=
g(\theta),
\qquad
\fint_{\T^d}u(\theta)\,d\theta=0.
\end{equation}
In Fourier series, its zero-average solution is given by
\begin{equation}
\label{eq-u-Fourier-orbit}
u(\theta)
=
\sum_{\ell\in\Z^d\setminus\{0\}}
\frac{\widehat g_\ell}
{i\,\omega\cdot\ell}
e^{i\ell\cdot\theta}.
\end{equation}
Since $g$ is real-valued, the function $u$ is real-valued as well.
Define
\begin{equation}
\label{eq-straightening-map-x}
\Phi(x)
:=
x+u(x\omega).
\end{equation}
Then
\begin{align*}
\Phi'(x)
&=
1+
\omega\cdot\partial_\theta u(x\omega)
\\
&=
1+g(x\omega)
\\
&=
\frac{c_\star}{\mathscr V(x\omega)}.
\end{align*}
Since $c_\star$ and $\mathscr V$ have the same sign,
\[
\Phi'(x)>0,
\qquad
x\in\R.
\]
Moreover, $u$ is bounded, and therefore
\[
\Phi(x)-x=u(x\omega)
\]
is bounded. Consequently,
\[
\lim_{x\to\pm\infty}\Phi(x)=\pm\infty,
\]
and $\Phi:\R\to\R$ is an orientation-preserving diffeomorphism.\\
Along a solution of \eqref{eq-scalar-flow-on-graph}, we may write
\begin{align*}
\tfrac{d}{dt}\Phi(x(t))
&=
\Phi'(x(t))\dot x(t)
\\
&=
\frac{c_\star}{\mathscr V(x(t)\omega)}
\mathscr V(x(t)\omega)
\\
&=
c_\star.
\end{align*}
Consequently,
\begin{equation}
\label{eq-Phi-linear-motion}
\Phi(x(t))
=
c_\star t+\Phi(x_0),
\qquad
x_0:=x(0).
\end{equation}
We now describe the inverse of $\Phi$ in quasi-periodic form. Introduce
the torus map
\begin{equation}
\label{eq-torus-map-Theta}
\Theta:\T^d\to\T^d,
\qquad
\Theta(\theta)
=
\theta+u(\theta)\omega.
\end{equation}
Its differential is
\[
D\Theta(\theta)
=
\operatorname{Id}_d
+
\omega\otimes\partial_\theta u(\theta),
\]
where
\[
\bigl(
\omega\otimes\partial_\theta u(\theta)
\bigr)_{ij}
=
\omega_i\partial_{\theta_j}u(\theta).
\]
Since this is a rank-one perturbation of the identity, the matrix
determinant lemma gives
\begin{equation}
\label{eq-det-DTheta}
\det D\Theta(\theta)
=
1+\omega\cdot\partial_\theta u(\theta).
\end{equation}
Using the cohomological equation
\eqref{eq-cohomological-orbit}, we obtain
\[
\det D\Theta(\theta)
=
\frac{c_\star}{\mathscr V(\theta)}
>0,
\qquad
\forall\,\theta\in\T^d.
\]
Thus $\Theta$ is a local diffeomorphism.
We will show that $\Theta$ is in fact a global diffeomorphism.\\
Since
$\Theta$ is a local diffeomorphism, $\Theta(\T^d)$ is open in $\T^d$.
On the other hand, $\T^d$ is compact and $\Theta$ is continuous, so
$\Theta(\T^d)$ is compact and therefore closed in $\T^d$. By
connectedness of $\T^d$,
\[
\Theta(\T^d)=\T^d,
\]
and hence $\Theta$ is surjective.
\\
To prove injectivity, consider the homotopy
\[
\Theta_s(\theta)
:=
\theta+s u(\theta)\omega,
\qquad
s\in[0,1].
\]
It joins the identity to $\Theta$:
\[
\Theta_0=\operatorname{Id}_{\T^d},
\qquad
\Theta_1=\Theta.
\]
By homotopy invariance of the topological degree,
\[
\deg(\Theta)
=
\deg(\operatorname{Id}_{\T^d})
=
1.
\]
We recall below the definition of the degree 
\[
\deg(\Theta)
=
\sum_{\theta\in\Theta^{-1}(\vartheta)}
\operatorname{sgn}
\bigl(\det D\Theta(\theta)\bigr).
\]
Since $\Theta$ is a local diffeomorphism and
\[
\det D\Theta(\theta)>0
\]
everywhere, every preimage of a point contributes $+1$ to the degree.
Therefore, for every $\vartheta\in\T^d$,
\[
1
=
\deg(\Theta)
=
\#\Theta^{-1}(\vartheta).
\]
Thus $\Theta$ is injective, and consequently
\[
\Theta:\T^d\longrightarrow\T^d
\]
is a global diffeomorphism.
We claim that its inverse has the same special structure. \\Indeed, let
$\varphi\in\T^d$ and set
\[
\theta=\Theta^{-1}(\varphi).
\]
Since
\[
\varphi
=
\theta+u(\theta)\omega,
\]
we obtain
\[
\theta
=
\varphi-u(\theta)\omega.
\]
Defining
\[
q(\varphi)
:=
-u\bigl(\Theta^{-1}(\varphi)\bigr),
\]
we get
\begin{equation}
\label{eq-inverse-Theta1}
\Theta^{-1}(\varphi)
=
\varphi+q(\varphi)\omega.
\end{equation}
Equivalently, $q$ satisfies
\[
q(\varphi)
+
u\bigl(\varphi+q(\varphi)\omega\bigr)
=
0,
\qquad
\varphi\in\T^d.
\]
We next recover the inverse of the scalar map $\Phi$. From
\[
\Phi(x)=x+u(x\omega)
\]
we have
\[
\Phi(x)\omega
=
x\omega+u(x\omega)\omega
=
\Theta(x\omega).
\]
Therefore, if $s=\Phi(x)$, then we get in view of \eqref{eq-inverse-Theta1}
\[
x\omega
=
\Theta^{-1}(s\omega)
=
s\omega+q(s\omega)\omega.
\]
It follows that
\begin{equation}
\label{eq-inverse-Phi}
\Phi^{-1}(s)
=
s+q(s\omega),
\qquad
s\in\R.
\end{equation}
Set
\[
\beta:=\Phi(x_0),
\qquad
\nu:=c_\star\omega.
\]
From \eqref{eq-Phi-linear-motion} and
\eqref{eq-inverse-Phi}, we obtain
\begin{align}
x(t)
&=
\Phi^{-1}(c_\star t+\beta)
\nonumber\\
&=
c_\star t+\beta
+
q\bigl(
c_\star t\omega+\beta\omega
\bigr)
\nonumber\\
&=
c_\star t+q_\star(\nu t),
\label{eq-x-QP-modulo-drift}
\end{align}
where
\begin{equation}
\label{eq-q-star-orbit}
q_\star(\varphi)
:=
\beta+q(\varphi+\beta\omega).
\end{equation}
For the vertical component, we write
\[
y(t)
=
Y(x(t))
=
\mathscr Y(x(t)\omega).
\]
Moreover,
\begin{align*}
x(t)\omega
&=\Theta^{-1}\big( \Phi(x(t))\omega\big)\\
&=
\Theta^{-1}
\bigl(
\nu t+\beta\omega
\bigr).
\end{align*}
Hence
\[
y(t)
=
\mathscr Y
\left(
\Theta^{-1}
\bigl(
\nu t+\beta\omega
\bigr)
\right).
\]
Define
\begin{equation}
\label{eq-P-vector-orbit}
Z_\star(\varphi)
:=
\left(
q_\star(\varphi),
\,
\mathscr Y
\left(
\Theta^{-1}
\bigl(
\varphi+\beta\omega
\bigr)
\right)
\right).
\end{equation}
Then $Z_\star$ is periodic on $\T^d$ and
\[
z(t)
=
\bigl(
x(t),y(t)
\bigr)
=
c_\star t\,e_1
+
Z_\star(\nu t).
\]
This proves \eqref{eq-point-vortex-QP-modulo-translation}.
It remains to establish the regularity of $Z_\star$. By
Proposition~\ref{prop-regularity-Robin-hull}, the Robin hull satisfies
\[
\mathscr R
\in
C_{\mathrm{loc}}^{m-1}(\mathcal D).
\]
The graph
\[
\left\{
\bigl(
\theta,\mathscr Y(\theta)
\bigr):
\theta\in\T^d
\right\}
\]
is a compact subset of $\mathcal D$. Hence $\mathscr R$ is
$C^{m-1}$ in a neighborhood of this graph. Combining
\eqref{eq-V-hull} with the regularity of $\mathscr Y$ established in
Theorem~\ref{thm-exact-qp-orbits}, the standard composition rules give
\[
\mathscr V\in C^{m-2}(\T^d).
\]
Since $\mathscr V$ is bounded away from zero,
\[
g\in C^{m-2}(\T^d).
\]
Applying Lemma~\ref{lem-regularity-u} to the Fourier representation
\eqref{eq-u-Fourier-orbit}, we obtain
\[
u\in C^r(\T^d)
\qquad\text{provided that}\qquad
r<m-2-\tau-d.
\]
Condition \eqref{eq-regularity-threshold-time} guarantees precisely this
inequality. Therefore
\[
u,q\in C^r(\T^d),
\qquad
\Theta^{\pm1}\in C^r(\T^d).
\]
Since
\[
\mathscr Y\in C^{m-1}(\T^d)
\]
and $m-1>r$, the composition formula
\eqref{eq-P-vector-orbit} finally yields
\[
Z_\star\in C^r(\T^d;\R^2).
\]
Since $c_\star\neq0$ and $\omega$ is Diophantine, the vector
\[
\nu=c_\star\omega
\]
is nonresonant. 
It is known,  see for instance \cite{Walters1982,KatokHasselblatt1995}, that  the Kronecker flow
\[
\theta\longmapsto \theta+\nu t
\qquad\text{on }\mathbb T^d
\]
is uniquely ergodic, with unique invariant probability measure given by
the Haar measure
\(
\frac{d\theta}{(2\pi)^d}.
\)
Now, for every $\Psi\in C(\mathbb R^2)$, the function
\(
\theta\in\T^d\longmapsto \Psi\bigl(Z_\star(\theta)\bigr)
\)
is continuous on $\mathbb T^d$. Then, the unique ergodicity of the Kronecker flow
yields
\[
\lim_{T\to\infty}
\frac1T\int_0^T
\Psi\bigl(Z_\star(\nu t)\bigr)\,dt
=
\frac1{(2\pi)^d}
\int_{\mathbb T^d}
\Psi\bigl(Z_\star(\theta)\bigr)\,d\theta,
\]
which proves the desired result and  completes the proof of the theorem.
\end{proof}

\begin{remark}
\label{rem-diophantine-exponent}
The restriction $\tau>d-1$ is not needed for the inversion of the
cohomological operator once a fixed vector $\omega$ is known to satisfy
the Diophantine estimate
\[
|\omega\cdot\ell|
\geq
\frac{\gamma}{\langle\ell\rangle^\tau},
\qquad
\ell\in\Z^d\setminus\{0\}.
\]
It is imposed in order to ensure that the set of such Diophantine
vectors has full Lebesgue measure; see
Lemma~\ref{lem-full-measure-diophantine}.
\end{remark}

\subsection{Quasi-periodic motions near the boundary}
\label{subsec-qp-motion-near-boundary}

We now apply Theorems~\ref{thm-exact-qp-orbits} and~\ref{thm-exact-qp-orbits-1} to the dynamics of point vortices located close to the boundary of the channel. The logarithmic behavior of the Robin function near the boundary provides both the transversality and the common energy range required to construct quasi-periodic invariant graphs. We show that every point sufficiently close to either boundary belongs to one of these graphs. Moreover, under a Diophantine condition on the frequency vector, the dynamics on each invariant graph can be straightened into a uniform horizontal drift. Consequently, after removing this drift, the corresponding point-vortex motion is quasi-periodic in time. Before stating this result, we recall a few standard facts concerning the distance map to the boundary and the behavior of the Robin function near the boundary. \\
For each boundary component $\Gamma_\theta$ of $\Omega_\theta$, parametrized by $\gamma_\theta$, there exists a tubular neighborhood of uniform width $\rho_{\rm tub}>0$ in which every point admits a unique representation
\[
z=\gamma_\theta(s)+r\,n_\theta(s),\qquad |r|<\rho_{\rm tub},
\]
where $n_\theta$ denotes the unit normal to $\Gamma_\theta$. The signed distance $r=r_\theta(z)$ has the regularity of the corresponding normal coordinates and, on the interior side of the channel, coincides with
\[
d_\theta(z):=\operatorname{dist}(z,\partial\Omega_\theta)=|r_\theta(z)|.
\]
Moreover,
\[
\nabla d_\theta(z)=n_{\theta,\rm in}(\pi_\theta(z)),
\]
where $\pi_\theta(z)$ denotes the unique nearest point on the boundary. Since the boundary components are graphs with uniformly bounded slopes, the vertical direction is uniformly transversal to the boundary. Hence, after possibly decreasing $\rho_{\rm tub}$, there exists $\kappa_0>0$, independent of $\theta\in\T^d$, such that
\begin{align}
\label{eq-distance-vertical-transversality}
\partial_y d_\theta(z)\leq-\kappa_0<0&
\quad\text{near the upper boundary},\\
\nonumber \partial_y d_\theta(z)\geq\kappa_0>0
\quad&\text{near the lower boundary}.
\end{align}
In particular, $d_\theta$ is strictly monotone along the vertical fibers in the corresponding tubular neighborhoods.
\\
We now recall the standard boundary asymptotics of the Robin function. For $(\theta,y)\in\mathcal D$, let
\[
d_\theta(0,y):=\operatorname{dist}\bigl((0,y),\partial\Omega_\theta\bigr).
\]
Then, uniformly with respect to $\theta\in\T^d$, as $(\theta,y)\in\mathcal D$ approaches the boundary, namely as
\[
d_\theta(0,y)\to0,
\]
we have
\begin{equation}
\label{eq-Robin-boundary-asymptotics}
\mathscr R(\theta,y)
=
\frac{1}{2\pi}
\log\bigl(2d_\theta(0,y)\bigr)
+
O\bigl(d_\theta(0,y)\bigr),
\end{equation}
and
\begin{equation}
\label{eq-Robin-gradient-boundary-asymptotics}
\partial_y\mathscr R(\theta,y)
=
\frac{1}{2\pi}
\frac{\partial_y d_\theta(0,y)}
{d_\theta(0,y)}
+
O(1).
\end{equation}
The remainder terms are uniform in $\theta$.
Combining \eqref{eq-distance-vertical-transversality} and
\eqref{eq-Robin-gradient-boundary-asymptotics}, we deduce that, for
$d_\theta(0,y)$ sufficiently small,
\begin{equation}
\label{eq-Robin-strong-transversality-boundary}
\left|
\partial_y\mathscr R(\theta,y)
\right|
\geq
\frac{c}{d_\theta(0,y)}
\end{equation}
for some $c>0$ independent of $\theta$. Thus the transversality condition required in Theorem~\ref{thm-exact-qp-orbits} becomes stronger as the boundary is approached.
We can now state the consequence for the point-vortex dynamics.
\begin{corollary}[Quasi-periodic motions near the boundary]
\label{cor-qp-motion-near-boundary}
Assume that the interfaces $f_\pm$ satisfy \eqref{separation1} and \eqref{reg-m} with $m\geq2$. Then there exists $\rho_\star>0$ such that every point
\(
z_0=(x_0,y_0)\in\Omega
\)
satisfying
\begin{equation}
\label{eq-point-close-boundary}
0<\operatorname{dist}(z_0,\partial\Omega)<\rho_\star
\end{equation}
belongs to a quasi-periodic invariant graph of the Robin function as stated in Theorem $\ref{thm-exact-qp-orbits}.$ Assume in addition that $\omega$ is Diophantine,
\begin{equation}
\label{eq-Diophantine-boundary-orbits}
|\omega\cdot\ell|
\geq
\frac{\gamma}{\langle\ell\rangle^\tau},
\qquad
\ell\in\Z^d\setminus\{0\},
\end{equation}
and $r,m$  satisfying
$$
r\geq1, m>r+\tau+d+2.
$$ Then  the point-vortex trajectory issued from $z_0$ can be written in the form
\begin{equation}
\label{eq-boundary-drift-qp-motion}
z(t)=c_{\star}t\,e_1+Z_{\star}(\nu t),
\qquad
\nu=c_{\star}\omega,
\end{equation}
where $c_\star,Z_{\star}$ are as in Theorem $\ref{thm-exact-qp-orbits-1}. $

\end{corollary}

\begin{proof}
We give the proof near the upper boundary, the argument near the lower
boundary being identical up to the orientation of the vertical
coordinate.
Let
\[
\rho_0
:=
\operatorname{dist}(z_0,\partial\Omega).
\]
We assume that $\rho_0>0$ is sufficiently small and that $z_0$ belongs
to the tubular neighborhood of the upper boundary. By the hull
identity in Proposition \ref{prop-exact-hull-Robin},
\[
R_\Omega(x,y)
=
\mathscr R(x\omega,y),
\]
and therefore, 
we have
\begin{equation}
\label{eq-E0-near-boundary}
E_0
=
\mathscr R(\theta_0,y_0)\quad\hbox{with}\quad \theta_0:=x_0\omega.
\end{equation}
Since
\[
d_{\theta_0}(0,y_0)=\rho_0,
\]
the boundary expansion
\eqref{eq-Robin-boundary-asymptotics} gives
\begin{equation}
\label{eq-E0-boundary-expansion}
E_0
=
\frac{1}{2\pi}\log(2\rho_0)
+
O(\rho_0).
\end{equation}
We now construct two phase-dependent barriers following the upper
boundary. For every $\theta\in\T^d$, let $a_{\rho_0}(\theta)$ and
$b_{\rho_0}(\theta)$ be uniquely determined, in the upper tubular
neighborhood, by
\begin{equation}
\label{eq-boundary-barriers-distance}
d_\theta
\bigl(
0,a_{\rho_0}(\theta)
\bigr)
=
2\rho_0,
\qquad
d_\theta
\bigl(
0,b_{\rho_0}(\theta)
\bigr)
=
\frac{\rho_0}{2}.
\end{equation}
For $\rho_0$ sufficiently small these functions are well defined and
belong to $C^{m-1}(\T^d)$. Since we are considering the upper
boundary, the point at distance $2\rho_0$ lies below the point at
distance $\rho_0/2$, and hence
\[
a_{\rho_0}(\theta)
<
b_{\rho_0}(\theta).
\]
We introduce the variable strip
\[
J_\theta^{\rho_0}
=
\left[
a_{\rho_0}(\theta),
b_{\rho_0}(\theta)
\right].
\]
Applying \eqref{eq-Robin-boundary-asymptotics} at the two endpoints
gives, uniformly in $\theta$,
\begin{align*}
\mathscr R
\bigl(
\theta,a_{\rho_0}(\theta)
\bigr)
&=
\frac{1}{2\pi}
\log(4\rho_0)
+
O(\rho_0),
\\
\mathscr R
\bigl(
\theta,b_{\rho_0}(\theta)
\bigr)
&=
\frac{1}{2\pi}
\log(\rho_0)
+
O(\rho_0).
\end{align*}
On the other hand,
\eqref{eq-E0-boundary-expansion} yields
\begin{align*}
\mathscr R
\bigl(
\theta,a_{\rho_0}(\theta)
\bigr)
-
E_0
&=
\frac{\log2}{2\pi}
+
O(\rho_0),
\\
E_0
-
\mathscr R
\bigl(
\theta,b_{\rho_0}(\theta)
\bigr)
&=
\frac{\log2}{2\pi}
+
O(\rho_0).
\end{align*}
Hence, after decreasing $\rho_\star$ if necessary,
\begin{equation}
\label{eq-E0-strict-common-range}
\mathscr R
\bigl(
\theta,b_{\rho_0}(\theta)
\bigr)
<
E_0
<
\mathscr R
\bigl(
\theta,a_{\rho_0}(\theta)
\bigr),
\qquad
\theta\in\T^d.
\end{equation}
In particular,
\[
E_0
\in
\bigcap_{\theta\in\T^d}
\mathscr R
\bigl(
\theta,
(a_{\rho_0}(\theta),b_{\rho_0}(\theta))
\bigr).
\]
Thus the common-energy condition of
Theorem~\ref{thm-exact-qp-orbits} is satisfied.
\\
It remains to verify the transversality assumption. For
\(
y\in J_\theta^{\rho_0},
\)
we have from the monotonicity of $y\mapsto d_\theta(0,y) $
\[
\frac{\rho_0}{2}
\leq
d_\theta(0,y)
\leq
2\rho_0.
\]
By \eqref{eq-Robin-gradient-boundary-asymptotics} and
\eqref{eq-distance-vertical-transversality}, for $\rho_0$ sufficiently
small we obtain
\[
\left|
\partial_y\mathscr R(\theta,y)
\right|
\geq
\frac{c}{\rho_0},
\qquad
\theta\in\T^d,
\quad
y\in J_\theta^{\rho_0}.
\]
Therefore the uniform transversality assumption
\eqref{eq-uniform-transversality-general} is also satisfied.
\\
All the hypotheses of Theorem~\ref{thm-exact-qp-orbits} are now
fulfilled. Hence there exists a unique
\[
\mathscr Y_{E_0}\in C^{m-1}(\T^d)
\]
such that
\[
a_{\rho_0}(\theta)
\leq
\mathscr Y_{E_0}(\theta)
\leq
b_{\rho_0}(\theta)
\]
and
\[
\mathscr R
\bigl(
\theta,\mathscr Y_{E_0}(\theta)
\bigr)
=
E_0.
\]
The corresponding physical graph
\[
Y_{E_0}(x)
=
\mathscr Y_{E_0}(x\omega)
\]
is therefore quasi-periodic and invariant under the point-vortex
flow.
Finally, since
\[
\mathscr R(\theta_0,y_0)
=
E_0
\]
and
\[
d_{\theta_0}(0,y_0)=\rho_0
\in
\left(
\frac{\rho_0}{2},2\rho_0
\right),
\]
we have
\[
a_{\rho_0}(\theta_0)
<
y_0
<
b_{\rho_0}(\theta_0).
\]
The uniqueness part of Theorem~\ref{thm-exact-qp-orbits} therefore
implies
\[
y_0
=
\mathscr Y_{E_0}(\theta_0).
\]
Since $\theta_0=x_0\omega$, it follows that
\[
y_0
=
Y_{E_0}(x_0),
\]
and hence
\[
z_0\in\Gamma_{E_0}.
\]
This completes the proof of the first point. The second part follows directly from Theorem~\ref{thm-exact-qp-orbits-1}, since all its assumptions are satisfied by the invariant graph constructed above.
\end{proof}

\begin{remark}
\label{rem-boundary-qp-foliation}
Corollary~\ref{cor-qp-motion-near-boundary} shows more than the
existence of isolated quasi-periodic trajectories near the boundary.
A sufficiently thin neighborhood of either boundary is foliated by
quasi-periodic level graphs of the Robin function. Indeed,
\eqref{eq-Robin-gradient-boundary-asymptotics} implies that the Robin
function is strictly monotone along each vertical fiber in this
region. Thus every point sufficiently close to the boundary belongs
to exactly one of these invariant graphs.
\end{remark}

\section{Small-amplitude perturbations of the flat strip}\label{sec-4}

We now specialize the general theory developed above to small
quasi-periodic perturbations of the flat strip. This setting provides a
natural model in which the conclusions of
Theorems~\ref{thm-exact-qp-orbits} and
\ref{thm-exact-qp-orbits-1} can be made more explicit and their
geometric content can be examined in detail.

In the unperturbed flat strip, the Robin function depends only on the
transverse variable, and its regular level sets form a simple family of
horizontal vortex trajectories. Our purpose is to understand how this
family is deformed when the two interfaces are subjected to small
quasi-periodic perturbations. More precisely, we show that the regular
level sets of the flat Robin function persist, for sufficiently small
deformations, as quasi-periodic invariant graphs of the perturbed
vortex dynamics.
We restrict our analysis to levels lying away from the centerline
$y=h/2$, where the flat Robin function is nondegenerate in the
transverse direction and the transversality condition required by the
general theory persists under sufficiently small perturbations. This
allows us to obtain a family of quasi-periodic invariant trajectories
continuing the horizontal trajectories of the flat geometry. Finally,
using the Hadamard expansion of the Robin function, we determine the
leading-order deformation of these invariant graphs and thereby make
explicit the first-order effect of the boundary oscillations on the
vortex trajectories.
\\
Here, we consider a small quasi-periodic perturbation of the flat
channel of the form
\begin{equation}
\Omega_\eps
=
\left\{
(x,y)\in\R^2:
\eps g_-(\omega_-x)<y<
h+\eps g_+(\omega_+x)
\right\},
\qquad
|\eps|\ll1,
\label{eq-small-qp-domain}
\end{equation}
where
\[
g_\pm:\T^{d_\pm}\to\R
\]
are smooth periodic functions, $h>0$, and
\[
\omega:=(\omega_-,\omega_+)
\in
\R^{d_-}\times\R^{d_+}
=
\R^d.
\]
The corresponding hull family is
\begin{equation}
\label{eq-small-qp-hull-domain}
\Omega_{\eps,\theta}
=
\left\{
(x,y)\in\R^2:
\eps g_-(\omega_-x+\theta_-)<y<
h+\eps g_+(\omega_+x+\theta_+)
\right\},
\end{equation}
and its natural Robin-hull domain is
\begin{equation}
\label{eq-small-qp-Robin-domain}
\mathcal D_\eps
:=
\left\{
(\theta,y)\in\T^d\times\R:
\eps g_-(\theta_-)<y<
h+\eps g_+(\theta_+)
\right\}.
\end{equation}
According to Proposition \ref{prop-exact-hull-Robin}, the hull Robin function denoted here by $\mathscr R_\eps$ satisfies 
\[
\mathscr R_\eps(\theta,y)
=
R_{\Omega_{\eps,\theta}}(0,y),
\qquad
\forall\,(\theta,y)\in\mathcal D_\eps.
\]
\subsection{The flat-strip Robin function}

In the absence of perturbations in \eqref{eq-small-qp-domain}, the
domain reduces to the flat strip
\[
\Omega_0=\R\times(0,h).
\]
The associated Dirichlet Green function is explicitly given by
\begin{equation}
\label{eq-flat-green}
G_0(z,\zeta)
=
\frac{1}{4\pi}
\log\left(
\frac{
\cosh\left(\frac{\pi(x-\xi)}{h}\right)
-
\cos\left(\frac{\pi(y+\eta)}{h}\right)
}{
\cosh\left(\frac{\pi(x-\xi)}{h}\right)
-
\cos\left(\frac{\pi(y-\eta)}{h}\right)
}
\right),
\end{equation}
where $z=(x,y)$ and $\zeta=(\xi,\eta)$ belong to $\Omega_0$.
Therefore, by a straightforward computation, the Robin function
depends only on the transverse variable $y$ and takes the form
\begin{equation}
R_0(y)
=
\frac1{2\pi}
\log
\left(
\frac{2h}{\pi}
\sin\frac{\pi y}{h}
\right).
\label{eq-flat-robin}
\end{equation}
Thus
\begin{equation}
R_0'(y)
=
\frac1{2h}
\cot\frac{\pi y}{h},
\qquad
R_0''(y)
=
-\frac{\pi}{2h^2}
\csc^2\frac{\pi y}{h}.
\label{eq-flat-robin-derivatives}
\end{equation}
Hence $y=h/2$ is the unique critical horizontal level. The unperturbed
vortex dynamics is
\begin{equation}
\label{eq-flat-vortex}
\dot x
=
\frac{\Gamma}{4h}
\cot\frac{\pi y}{h},
\qquad
\dot y=0.
\end{equation}
Thus every horizontal line $y=y_\star\neq h/2$ is an invariant orbit
on which the vortex moves with constant nonzero horizontal velocity.
The direction of propagation changes across the centerline $y=h/2$,
which is the unique critical level and consists entirely of stationary
vortex configurations. Hence the flat strip possesses a one-parameter
family of translating invariant orbits, and the purpose of the
perturbative analysis below is to understand how these orbits deform
under small quasi-periodic perturbations of the channel boundaries.
Another point to mention is that the geometry is invariant under horizontal
translations and, consequently, the associated hull family of domains
is independent of the phase variable $\theta$. The Robin hull function
does not depend on $\theta$ either. More precisely,
\[
\mathcal D_0
=
\T^d\times(0,h)
\]
and
\[
\mathscr R_0(\theta,y)
=
R_0(y),
\qquad
(\theta,y)\in\mathcal D_0.
\]
Thus the level sets of the Robin hull are simply the horizontal graphs
\[
y=y_\star,
\]
and the phase dependence appears only after the quasi-periodic
perturbation of the interfaces is introduced.

\subsection{Robin function expansion via the Hadamard variational formula}
\label{subsec-Hadamard-Robin}
We now derive the first-order variation of the Robin function under a small quasi-periodic deformation of the flat strip. The natural tool for this purpose is the classical Hadamard variational formula, which describes the shape derivative of the Dirichlet Green function in terms of the normal displacement of the boundary and the normal derivatives of the unperturbed Green function. In particular, the first-order effect of a boundary deformation is encoded entirely by a boundary integral. We refer, for instance, to \cite{Hadamard,HenrotPierre,Kozono,Peetre, SokolowskiZolesio, SuzukiTsuchiya2016} for the general theory of shape derivatives and Hadamard variational formulas.
We intend to apply the general Hadamard variational formula to the
quasi-periodic perturbation of the flat strip
\[
\Omega_0=\mathbb R\times(0,h).
\]
The boundary of the perturbed domain $\Omega_\eps$ defined in
\eqref{eq-small-qp-domain} consists of the two interfaces
\[
y=\eps g_-(\omega_-x),
\qquad
y=h+\eps g_+(\omega_+x),
\]
where
\[
g_\pm:\T^{d_\pm}\to\R
\]
are smooth periodic functions. Let $R_\eps$ denote the Robin function
of $\Omega_\eps$. The purpose of this subsection is to obtain the
first-order expansion of $R_\eps$ with a remainder of order
$O(\eps^2)$, locally uniformly in the flat strip, together with the
corresponding expansion of the Robin hull on its natural domain. Our result reads as follows.

\begin{lemma}
\label{lem-robin-expansion-perturbed}
Assume that
\[
g_\pm\in C^{m,\alpha}(\T^{d_\pm}),
\qquad
\alpha\in(0,1),
\qquad
m\ge2.
\]
Let
\(
K\subset(0,h)
\)
be compact. Then there exists $\eps_0>0$ such that, for every
$|\eps|\leq\eps_0$,
\begin{equation}
\label{eq-Robin-hull-final-expansion}
\mathscr R_\eps(\theta,y)
=
R_0(y)
+
\eps\mathcal R_1(\theta,y)
+
\eps^2\mathcal R_{2,\eps}(\theta,y),
\end{equation}
for $(\theta,y)\in\T^d\times K$, where
\begin{equation}
\label{eq-R1-hull-final}
\mathcal R_1(\theta,y)
=
\int_{\R}
g_+(\theta_++s\omega_+)
P_+^2(s,y)\,ds
-
\int_{\R}
g_-(\theta_-+s\omega_-)
P_-^2(s,y)\,ds,
\end{equation}
and $\mathcal R_{2,\eps}:\T^d\times(0,h)\to\R$ satisfies
\begin{equation*}
\label{eq-R2-uniform-Cm1}
\sup_{|\eps|\leq\eps_0}
\|
\mathcal R_{2,\eps}
\|_{C^{m-1}(\T^d\times K)}
\leq C_K.
\end{equation*}
Here
\begin{equation*}
\label{eq-Ppm-flat}
P_{\pm}(s,y)
=
\frac{1}{2h}
\frac{
\sin(\pi y/h)
}{
\cosh(\pi s/h)\pm\cos(\pi y/h)
}.
\end{equation*}
In addition,
\begin{equation*}
\label{eq-Robin-physical-expansion}
R_\eps(x,y)
=
R_0(y)
+
\eps R_1(x,y)
+
\eps^2R_{2,\eps}(x,y),
\end{equation*}
where
\begin{equation*}
\label{eq-R1-explicit}
R_1(x,y)
=
\int_{\R}
g_+\bigl((s+x)\omega_+\bigr)
P_+^2(s,y)\,ds
-
\int_{\R}
g_-\bigl((s+x)\omega_-\bigr)
P_-^2(s,y)\,ds,
\end{equation*}
and
\[
R_{2,\eps}(x,y)
=
\mathcal R_{2,\eps}(x\omega,y).
\]
\end{lemma}

\begin{proof}
We first work at the level of the Robin hull. For
\(
\theta=(\theta_-,\theta_+)\in\T^d,
\)
let $\Omega_{\eps,\theta}$ denote the corresponding phase-shifted
channel,
\[
\Omega_{\eps,\theta}
=
\left\{
(x,y)\in\R^2:
\eps g_-(\theta_-+\omega_-x)
<
y
<
h+\eps g_+(\theta_++\omega_+x)
\right\},
\]
and recall from Proposition \ref{prop-exact-hull-Robin}  that
\[
\mathscr R_\eps(\theta,y)
=
R_{\Omega_{\eps,\theta}}(0,y).
\]
At $\eps=0$, the domain is the flat strip
independently of $\theta$, and therefore
\[
\mathscr R_0(\theta,y)=R_0(y).
\]
We begin by computing the first variation. For fixed
$\theta\in\T^d$, let $G_{\eps,\theta}$ denote the Dirichlet Green
function of $\Omega_{\eps,\theta}$. To simplify the notation in this
part of the proof, we write $G_\eps=G_{\eps,\theta}$.
By the classical Hadamard variational formula for the Dirichlet Green
function, see for instance \cite{Hadamard,SuzukiTsuchiya2016}, we have, for
$z,\zeta\in\Omega_0$,
\begin{equation}
\label{eq-hadamard-green-proof}
\left.
\frac{d}{d\eps}G_\eps(z,\zeta)
\right|_{\eps=0}
=
\int_{\partial\Omega_0}
\rho(\xi)
\partial_{n_\xi}G_0(z,\xi)
\partial_{n_\xi}G_0(\zeta,\xi)
\,d\sigma(\xi),
\end{equation}
where $\rho$ denotes the signed normal velocity of the boundary
deformation, and $
\partial_{n_\xi}$
is the normal derivative, with respect to the boundary variable $\xi$.
\\
The boundary of the flat strip has the two components
\[
\Gamma_-=\R\times\{0\},
\qquad
\Gamma_+=\R\times\{h\},
\]
with outward unit normals
\[
n_-=(0,-1),
\qquad
n_+=(0,1).
\]
For the family $\Omega_{\eps,\theta}$, the corresponding signed normal
velocities at $\eps=0$ are
\[
\rho_-(s)
=
-g_-(\theta_-+s\omega_-),
\qquad
\rho_+(s)
=
g_+(\theta_++s\omega_+).
\]
Consequently, \eqref{eq-hadamard-green-proof} becomes
\begin{align}
\left.
\frac{d}{d\eps}G_\eps(z,\zeta)
\right|_{\eps=0}
&=
\int_{\R}
g_+(\theta_++s\omega_+)
\partial_{n_+}G_0(z,(s,h))
\partial_{n_+}G_0(\zeta,(s,h))
\,ds
\nonumber\\
&\quad
-
\int_{\R}
g_-(\theta_-+s\omega_-)
\partial_{n_-}G_0(z,(s,0))
\partial_{n_-}G_0(\zeta,(s,0))
\,ds.
\label{eq-hadamard-two-boundaries-proof}
\end{align}
We now take
\[
z=\zeta=(0,y),
\qquad y\in K.
\]
Since the logarithmic singular part of the Green function is universal
and does not depend on the domain deformation, differentiating the
regular part and then restricting to the diagonal gives the
corresponding variation of the Robin function. For the flat strip,
the normal derivatives of $G_0$ along the two boundary components are
the Poisson kernels, namely
\[
\partial_{n_+}G_0((0,y),(s,h))
=
P_+(s,y),
\qquad
\partial_{n_-}G_0((0,y),(s,0))
=
P_-(s,y).
\]
 Hence
\begin{align}
\left.
\partial_\eps
\mathscr R_\eps(\theta,y)
\right|_{\eps=0}
&=
\int_{\R}
g_+(\theta_++s\omega_+)
P_+^2(s,y)\,ds
-
\int_{\R}
g_-(\theta_-+s\omega_-)
P_-^2(s,y)\,ds,
\label{eq-first-variation-robin-hull-proof}
\end{align}
which gives the formula of $\mathcal R_1$ stated  in
\eqref{eq-R1-hull-final}.
Since $K\Subset(0,h)$, the kernels
\[
P_{\pm}(s,y)
=
\frac{1}{2h}
\frac{\sin(\pi y/h)}
{\cosh(\pi s/h)\pm\cos(\pi y/h)}
\]
and all the $y$-derivatives required below decay exponentially as
$|s|\to\infty$, uniformly for $y\in K$. Differentiation with respect
to $\theta_\pm$ acts only on $g_\pm$. Thus the assumption
\[
g_\pm\in C^{m,\alpha}(\T^{d_\pm})
\]
allows differentiation under the integral sign and gives the required
regularity of $\mathcal R_1$ on $\T^d\times K$.
\\
It remains to control the second-order remainder. We briefly recall
the standard fixed-domain argument. Choose a family of
diffeomorphisms, as in the proof of Proposition \ref{prop-regularity-Robin-hull},
\[
\Phi_{\eps,\theta}:\Omega_0\longrightarrow\Omega_{\eps,\theta}
\]
which maps the two horizontal components of $\partial\Omega_0$ onto
the corresponding perturbed interfaces and depends smoothly on
$(\eps,\theta)$. For $|\eps|$ sufficiently small,
$\Phi_{\eps,\theta}$ and its inverse are uniformly controlled.
Pulling the Dirichlet problem back by $\Phi_{\eps,\theta}$ transforms
the Laplacian on $\Omega_{\eps,\theta}$ into a uniformly elliptic
divergence-form operator on the fixed strip,
\[
L_{\eps,\theta}
=
\nabla\cdot
\bigl(
A_{\eps,\theta}\nabla
\bigr),
\]
whose coefficients depend smoothly on $\eps$. Since
$g_\pm\in C^{m,\alpha}$, the coefficients and the derivatives in
$\eps$ needed below are uniformly controlled with the corresponding
$C^{m-1,\alpha}$ regularity in the phase and spatial variables.
\\
Differentiating the pulled-back elliptic problem twice with respect to
$\eps$ gives uniformly elliptic equations for the first and second
$\eps$-derivatives of the pulled-back Green function. Standard
Schauder estimates on compact subsets, together with the subtraction
of the universal logarithmic singularity before restriction to the
diagonal, yield
\begin{equation}
\label{eq-second-epsilon-Robin-bound}
\sup_{|\eps|\leq\eps_0}
\left\|
\partial_\eps^2\mathscr R_\eps
\right\|_{C^{m-1}(\T^d\times K)}
\leq C_K.
\end{equation}
We do not reproduce the standard elliptic shape-differentiability
argument in further detail. It is quite similar to the proof of Proposition \ref{prop-regularity-Robin-hull}.
Taylor's formula with integral remainder now gives
\begin{align*}
\mathscr R_\eps(\theta,y)
&=
\mathscr R_0(\theta,y)
+
\eps
\left.
\partial_\eps\mathscr R_\eps(\theta,y)
\right|_{\eps=0}
\\
&\quad
+
\eps^2
\int_0^1
(1-t)
\partial_\eps^2
\mathscr R_{t\eps}(\theta,y)
\,dt.
\end{align*}
Defining
\[
\mathcal R_{2,\eps}(\theta,y)
:=
\int_0^1
(1-t)
\partial_\eps^2
\mathscr R_{t\eps}(\theta,y)
\,dt,
\]
we obtain
\[
\mathscr R_\eps(\theta,y)
=
R_0(y)
+
\eps\mathcal R_1(\theta,y)
+
\eps^2\mathcal R_{2,\eps}(\theta,y),
\]
and \eqref{eq-second-epsilon-Robin-bound} implies
\[
\sup_{|\eps|\leq\eps_0}
\|
\mathcal R_{2,\eps}
\|_{C^{m-1}(\T^d\times K)}
\leq C_K.
\]
Finally, the exact hull identity  established in Proposition \ref{prop-exact-hull-Robin} gives
\[
R_\eps(x,y)
=
\mathscr R_\eps(x\omega,y),
\qquad
x\omega=(x\omega_-,x\omega_+).
\]
It follows that
\[
R_\eps(x,y)
=
R_0(y)
+
\eps R_1(x,y)
+
\eps^2R_{2,\eps}(x,y),
\]
where
\[
R_{2,\eps}(x,y)
=
\mathcal R_{2,\eps}(x\omega,y)
\]
and
\begin{align*}
R_1(x,y)
&=
\mathcal R_1(x\omega,y)
\\
&=
\int_{\R}
g_+\bigl((s+x)\omega_+\bigr)
P_+^2(s,y)\,ds
-
\int_{\R}
g_-\bigl((s+x)\omega_-\bigr)
P_-^2(s,y)\,ds.
\end{align*}
Thus $R_{2,\eps}$ is quasi-periodic in $x$. This concludes the proof.
\end{proof}

\subsection{First-order expansion of the quasi-periodic graphs}

We now aim at a more quantitative description of the trajectories
constructed in Theorems~\ref{thm-exact-qp-orbits} and
\ref{thm-exact-qp-orbits-1} in the small-amplitude regime
$|\eps|\ll1$. In the flat channel, the invariant orbits are the
horizontal lines $y=y_\star$, and the natural question is to determine
how these orbits deform under a quasi-periodic perturbation of the
interfaces. Our goal is therefore to expand the corresponding invariant
graph around $y_\star$ and identify explicitly its first-order
quasi-periodic correction. Combining the expansion of the Robin hull
with the level-set equation will provide this correction in terms of
the boundary profiles $g_\pm$ and the Poisson kernels of the flat
strip.

\begin{proposition}
\label{cor-small-qp-exact-orbits}
Let
\[
J=[a,b]\subset(0,h),\,\hbox{with}\,\,
\tfrac h2\notin J.
\]
Then there exists $\varepsilon_J>0$ such that, for every
$|\varepsilon|\le\varepsilon_J$, we have 
\[
\T^d\times J\subset\mathcal D_\varepsilon
\quad\hbox{and}\quad
\inf_{(\theta,y)\in\T^d\times J}
\left|
\partial_y\mathscr R_\varepsilon(\theta,y)
\right|
>0.
\]
Define
\begin{equation}
\label{eq-common-energy-eps-J}
\mathcal E_{\varepsilon,J}^{\mathrm{glob}}
:=
\bigcap_{\theta\in\T^d}
\mathscr R_\varepsilon(\theta,J).
\end{equation}
Then, for $|\varepsilon|\le\varepsilon_J,$
\[
\operatorname{int}
\mathcal E_{\varepsilon,J}^{\mathrm{glob}}
\neq\varnothing,
\]
and  for every
\[
E\in
\operatorname{int}
\mathcal E_{\varepsilon,J}^{\mathrm{glob}},
\]
there exists a unique quasi-periodic invariant graph
\[
\Gamma_E
=
\Big\{
(x,Y_{\varepsilon,E}(x)):x\in\R
\Big\},
\]
with
\[
Y_{\varepsilon,E}(x)
=
\mathscr Y_{\varepsilon,E}(x\omega),
\qquad
\mathscr Y_{\varepsilon,E}
\in
C^1(\T^d;(a,b)),
\]
satisfying
\[
R_{\Omega_\varepsilon}
\bigl(
x,Y_{\varepsilon,E}(x)
\bigr)
=
E.
\]
Furthermore, for every $y_\star\in J$, the invariant graph associated
with the energy
\[
E=R_0(y_\star)
\]
is well defined and admits the expansion
\begin{align}
\mathscr Y_{\varepsilon,E}(\theta)
&=
y_\star
-
2h\varepsilon
\tan\left(\tfrac{\pi y_\star}{h}\right)
\mathcal R_1(\theta,y_\star)
+
O_J(\varepsilon^2),
\label{eq-Y-leading-general}
\end{align}
uniformly for
\[
(\theta,y_\star)\in\T^d\times J,
\]
where $\mathcal R_1$ is defined by
\eqref{eq-R1-hull-final}.
\end{proposition}

\begin{proof}
For $\varepsilon=0$, the Robin function of the flat strip
$\R\times(0,h)$ depends only on $y$, and in view of
\eqref{eq-flat-robin-derivatives} we have
\[
R_0'(y)
=
\tfrac{1}{2h}
\cot\left(\tfrac{\pi y}{h}\right).
\]
In particular, its only critical point in $(0,h)$ is $y=h/2$. By assumption, the interval $J$
lies entirely on one side of $h/2$, implying that
\[
m_J
:=
\min_{y\in J}|R_0'(y)|
>0.
\]
Since the functions $g_\pm$ are bounded, then 
for $|\varepsilon|$ sufficiently small we have
\[
\varepsilon g_-(\theta_-)
<
a
\leq y\leq
b
<
h+\varepsilon g_+(\theta_+),\,\forall (\theta,y)\in\T^d\times J.
\]
 Thus
\[
\T^d\times J
\subset
\mathcal D_\varepsilon.
\]
From Lemma \ref{lem-robin-expansion-perturbed}, we deduce that
\[
\|
\mathscr R_\varepsilon-R_0
\|_{C^1(\T^d\times J)}
\leq
C_J|\varepsilon|.
\]
Hence, for $|\varepsilon|$ sufficiently small,
\[
\left|
\partial_y\mathscr R_\varepsilon(\theta,y)
\right|
\geq
\tfrac{m_J}{2}
>0,
\qquad
(\theta,y)\in\T^d\times J.
\]
Thus the uniform transversality condition of
Theorem~\ref{thm-exact-qp-orbits} is satisfied on the constant
fiberwise interval
\(
J_\theta=J.
\)
We next verify that the corresponding common energy range is nonempty
and describe its behavior as $\varepsilon\to0$. For $\varepsilon=0$,
the Robin hull is independent of $\theta$, and therefore
\[
\mathcal E_{0,J}^{\mathrm{glob}}
=
R_0(J).
\]
Since $R_0'$ has a constant sign on $J$, this is a nondegenerate
interval.
\\
For $|\varepsilon|$ sufficiently small, the same monotonicity persists.
For each $\theta\in\T^d$, set
\[
e_{\varepsilon,J}^-(\theta)
:=
\min
\left\{
\mathscr R_\varepsilon(\theta,a),
\mathscr R_\varepsilon(\theta,b)
\right\}\quad\hbox{and}\quad 
e_{\varepsilon,J}^+(\theta)
:=
\max
\left\{
\mathscr R_\varepsilon(\theta,a),
\mathscr R_\varepsilon(\theta,b)
\right\}.
\]
Then
\[
\mathscr R_\varepsilon(\theta,J)
=
\left[
e_{\varepsilon,J}^-(\theta),
e_{\varepsilon,J}^+(\theta)
\right],
\]
and hence
\[
\mathcal E_{\varepsilon,J}^{\mathrm{glob}}
=
\left[
E_{\varepsilon,J}^-,
E_{\varepsilon,J}^+
\right],
\]
where
\[
E_{\varepsilon,J}^-
:=
\max_{\theta\in\T^d}
e_{\varepsilon,J}^-(\theta),
\qquad
E_{\varepsilon,J}^+
:=
\min_{\theta\in\T^d}
e_{\varepsilon,J}^+(\theta).
\]
The estimate
\[
\|
\mathscr R_\varepsilon-R_0
\|_{C^0(\T^d\times J)}
\leq
C_J|\varepsilon|
\]
implies
\[
E_{\varepsilon,J}^-
=
\min\{R_0(a),R_0(b)\}
+
O_J(|\varepsilon|)
\quad\textnormal{and}\quad
E_{\varepsilon,J}^+
=
\max\{R_0(a),R_0(b)\}
+
O_J(|\varepsilon|).
\]
Consequently,
\[
E_{\varepsilon,J}^+
-
E_{\varepsilon,J}^-
=
|R_0(b)-R_0(a)|
+
O_J(|\varepsilon|)
>0
\]
for $|\varepsilon|$ sufficiently small. Therefore
\[
\operatorname{int}
\mathcal E_{\varepsilon,J}^{\mathrm{glob}}
\neq\varnothing.
\]
More generally, every compact interval
\[
K
\Subset
\operatorname{int}R_0(J)
\]
satisfies
\[
K
\subset
\operatorname{int}
\mathcal E_{\varepsilon,J}^{\mathrm{glob}}
\]
for all sufficiently small $|\varepsilon|$. Thus the interior energy
range of the flat strip persists locally uniformly away from its
endpoints.
\\
We have therefore verified both the uniform transversality condition
and the nonemptiness of the common energy range. The conclusion follows
from Theorem~\ref{thm-exact-qp-orbits}. In particular, for every
\[
E
\in
\operatorname{int}
\mathcal E_{\varepsilon,J}^{\mathrm{glob}},
\]
there exists a unique global quasi-periodic invariant graph contained
in $\R\times(a,b)$ and lying on the level
\[
R_{\Omega_\varepsilon}=E.
\]
We now derive the first-order expansion. Since the estimate is required
uniformly for $y_\star\in J$, we choose a slightly larger compact
interval
\[
\widetilde J=[\widetilde a,\widetilde b]
\Subset(0,h)
\]
such that
\[
J\Subset\operatorname{int}\widetilde J
\quad\hbox{and}\quad
\tfrac h2\notin\widetilde J.
\]
The preceding argument applied to $\widetilde J$ shows that
\[
\operatorname{int}
\mathcal E_{\varepsilon,\widetilde J}^{\mathrm{glob}}
\neq\varnothing
\]
for $|\varepsilon|$ sufficiently small.
Since
\[
R_0(J)
\Subset
\operatorname{int}R_0(\widetilde J),
\]
the persistence of the common energy range gives
\begin{equation}
\label{eq-R0J-inside-energy}
R_0(J)
\subset
\operatorname{int}
\mathcal E_{\varepsilon,\widetilde J}^{\mathrm{glob}}
\end{equation}
for all sufficiently small $|\varepsilon|$. Consequently, for every
$y_\star\in J$, the energy
\(
E=R_0(y_\star)
\)
gives a unique invariant graph
\[
\mathscr Y_{\varepsilon,E}:
\T^d\longrightarrow\widetilde J.
\]
By \eqref{eq-hull-level-equation-general},
\begin{equation}
\label{eq-level-equation-Yeps}
\mathscr R_\varepsilon
\bigl(
\theta,\mathscr Y_{\varepsilon,E}(\theta)
\bigr)
=
R_0(y_\star),
\qquad
\theta\in\T^d.
\end{equation}
From the expansion of the Robin hull obtained in
\eqref{eq-Robin-hull-final-expansion}, applied on the compact set
\[
\T^d\times\widetilde J
\Subset
\mathcal D_0,
\]
we have
\[
\mathscr R_\varepsilon(\theta,y)
=
R_0(y)
+
\varepsilon\mathcal R_1(\theta,y)
+
O_{\widetilde J}(\varepsilon^2).
\]
Hence
\begin{equation}
\label{eq-level-equation-expanded}
R_0
\bigl(
\mathscr Y_{\varepsilon,E}(\theta)
\bigr)
+
\varepsilon
\mathcal R_1
\bigl(
\theta,
\mathscr Y_{\varepsilon,E}(\theta)
\bigr)
+
O_{\widetilde J}(\varepsilon^2)
=
R_0(y_\star).
\end{equation}
Since $\widetilde J$ does not intersect the critical level $h/2$,
\[
\inf_{y\in\widetilde J}|R_0'(y)|>0.
\]
The implicit function theorem therefore gives
\[
\mathscr Y_{\varepsilon,E}(\theta)-y_\star
=
O_J(\varepsilon),
\quad\hbox{
uniformly for}\quad
(\theta,y_\star)\in\T^d\times J.
\]
Set
\[
\delta_\varepsilon(\theta)
:=
\mathscr Y_{\varepsilon,E}(\theta)-y_\star.
\]
Then
\[
\delta_\varepsilon
=
O_J(\varepsilon).
\]
Expanding $R_0$ at $y_\star$, we obtain
\[
R_0(y_\star+\delta_\varepsilon)
=
R_0(y_\star)
+
R_0'(y_\star)\delta_\varepsilon
+
O_J(\delta_\varepsilon^2),
\]
and therefore
\[
R_0(y_\star+\delta_\varepsilon)
=
R_0(y_\star)
+
R_0'(y_\star)\delta_\varepsilon
+
O_J(\varepsilon^2).
\]
Similarly, using the regularity of $\mathcal R_1$ on
$\T^d\times\widetilde J$,
\[
\mathcal R_1
\bigl(
\theta,y_\star+\delta_\varepsilon
\bigr)
=
\mathcal R_1(\theta,y_\star)
+
O_J(\varepsilon).
\]
Consequently,
\[
\varepsilon
\mathcal R_1
\bigl(
\theta,y_\star+\delta_\varepsilon
\bigr)
=
\varepsilon
\mathcal R_1(\theta,y_\star)
+
O_J(\varepsilon^2).
\]
Substituting these expansions into
\eqref{eq-level-equation-expanded}, we obtain
\[
R_0'(y_\star)\delta_\varepsilon(\theta)
+
\varepsilon
\mathcal R_1(\theta,y_\star)
+
O_J(\varepsilon^2)
=
0.
\]
From the explicit formula
\[
\tfrac{1}{R_0'(y_\star)}
=
2h
\tan\left(\tfrac{\pi y_\star}{h}\right),
\]
we deduce that
\[
\delta_\varepsilon(\theta)
=
-
2h\varepsilon
\tan\left(\tfrac{\pi y_\star}{h}\right)
\mathcal R_1(\theta,y_\star)
+
O_J(\varepsilon^2).
\]
Recalling that
\[
\mathscr Y_{\varepsilon,E}(\theta)
=
y_\star+\delta_\varepsilon(\theta),
\]
we conclude that
\begin{align*}
\mathscr Y_{\varepsilon,E}(\theta)
&=
y_\star
-
2h\varepsilon
\tan\left(\tfrac{\pi y_\star}{h}\right)
\mathcal R_1(\theta,y_\star)
+
O_J(\varepsilon^2),
\end{align*}
uniformly for
\[
(\theta,y_\star)\in\T^d\times J.
\]
This proves \eqref{eq-Y-leading-general}.
\end{proof}

\subsection{First-order expansion of the quasi-periodic orbits}

The previous analysis provides a first-order description of the geometry
of the invariant graphs arising from the horizontal trajectories of the
flat channel. We now complement this spatial description by studying the
time parametrization of the corresponding point-vortex trajectories.
More precisely, we aim to determine how the constant horizontal drift of
the flat case is modified by the quasi-periodic perturbation and to
identify the leading oscillatory correction to the motion. To obtain a
description which is uniform for all times, we keep the effective drift
$c_\eps$ unexpanded in the linear part of the trajectory, while deriving
separately its first-order expansion in $\eps$. This leads to a global
representation of the vortex trajectory as the superposition of an
effective horizontal drift and a bounded quasi-periodic modulation.

\begin{proposition}
\label{prop-first-order-orbit-small-channel}
Assume the hypotheses of Proposition~$\ref{cor-small-qp-exact-orbits}$
and the Diophantine and regularity assumptions of
Theorem~$\ref{thm-exact-qp-orbits-1}$. Let
\[
y_\star\in(0,h)\setminus\left\{\tfrac h2\right\}
\]
and consider the invariant graph associated with the energy
\(
E=R_0(y_\star).
\)
Then the effective drift satisfies
\begin{equation*}
\label{eq-drift-first-order}
c_\eps
=
c_0+\eps c_1+O_{y_\star}(\eps^2),
\end{equation*}
where
\begin{equation}
\label{eq-c1-small-channel}
c_0
=
\tfrac{\Gamma}{4h}
\cot\left(\tfrac{\pi y_\star}{h}\right),\quad c_1
=
\fint_{\T^d}V_1(\theta)\,d\theta,
\end{equation}
with
\begin{equation*}
\label{eq-V1-small-channel}
V_1(\theta)
=
\tfrac{\Gamma}{2}
\left[
\partial_y\mathcal R_1(\theta,y_\star)
+
\tfrac{2\pi\mathcal R_1(\theta,y_\star)}
{h\sin(2\pi y_\star/h)}
\right].
\end{equation*}
Let  $u_1$ denote the unique real-valued zero-average solution of
\begin{equation}
\label{eq-u1-small-channel}
\omega\cdot\partial_\theta u_1(\theta)
=
\tfrac{c_1-V_1(\theta)}{c_0},
\qquad
\fint_{\T^d}u_1(\theta)\,d\theta=0.
\end{equation}
Then every point-vortex trajectory contained in the corresponding
invariant graph admits the global-in-time expansion
\begin{equation}
\label{eq-orbit-first-order-small-channel}
\begin{aligned}
x(t)
&=
x_0+c_\eps t
+
\eps
\left[
u_1(\theta_0)
-
u_1(\theta_0+c_\eps t\omega)
\right]
+
O_{y_\star}(\eps^2),
\\
y(t)
&=
y_\star
-
2h\eps
\tan\left(\tfrac{\pi y_\star}{h}\right)
\mathcal R_1
\bigl(
\theta_0+c_\eps t\omega,y_\star
\bigr)
+
O_{y_\star}(\eps^2),
\end{aligned}
\end{equation}
where
\[
\theta_0=x_0\omega.
\]
The remainders in \eqref{eq-orbit-first-order-small-channel} are
uniform for $t\in\R$.
\end{proposition}

\begin{proof}
From the first-order expansion of the invariant graph obtained in
\eqref{eq-Y-leading-general},
\begin{equation}
\label{eq-Y-expansion-proof-orbit}
\mathscr Y_{\eps,E}(\theta)
=
y_\star+\eps Y_1(\theta)+O_{y_\star}(\eps^2),
\end{equation}
where
\begin{equation}
\label{eq-Y1-proof-orbit}
Y_1(\theta)
=
-2h
\tan\left(\tfrac{\pi y_\star}{h}\right)
\mathcal R_1(\theta,y_\star)
=
-\tfrac{\mathcal R_1(\theta,y_\star)}
{R_0'(y_\star)}.
\end{equation}
Moreover, from the expansion of the Robin hull
\eqref{eq-Robin-hull-final-expansion}, locally around $y_\star$,
\[
\mathscr R_\eps(\theta,y)
=
R_0(y)
+
\eps\mathcal R_1(\theta,y)
+
O_{y_\star}(\eps^2),
\]
and hence
\[
\partial_y\mathscr R_\eps(\theta,y)
=
R_0'(y)
+
\eps\partial_y\mathcal R_1(\theta,y)
+
O_{y_\star}(\eps^2).
\]
Recall from Theorem~\ref{thm-exact-qp-orbits-1} that the horizontal
velocity along the invariant graph is
\[
\mathscr V_\eps(\theta)
=
\tfrac{\Gamma}{2}
\partial_y\mathscr R_\eps
\bigl(
\theta,\mathscr Y_{\eps,E}(\theta)
\bigr).
\]
Taylor expansion around $y=y_\star$ yields
\[
\mathscr V_\eps(\theta)
=
c_0+\eps V_1(\theta)+O_{y_\star}(\eps^2),
\]
where
\[
c_0
=
\tfrac{\Gamma}{2}R_0'(y_\star)
=
\tfrac{\Gamma}{4h}
\cot\left(\tfrac{\pi y_\star}{h}\right)
\]
and
\[
V_1(\theta)
=
\tfrac{\Gamma}{2}
\left[
\partial_y\mathcal R_1(\theta,y_\star)
-
\frac{R_0''(y_\star)}{R_0'(y_\star)}
\mathcal R_1(\theta,y_\star)
\right].
\]
Since
\[
R_0''(y)
=
-\tfrac{\pi}{2h^2}
\csc^2\left(\tfrac{\pi y}{h}\right),
\]
we obtain
\[
\frac{R_0''(y_\star)}{R_0'(y_\star)}
=
-\frac{2\pi}
{h\sin(2\pi y_\star/h)},
\]
and therefore
\[
V_1(\theta)
=
\tfrac{\Gamma}{2}
\left[
\partial_y\mathcal R_1(\theta,y_\star)
+
\tfrac{2\pi\mathcal R_1(\theta,y_\star)}
{h\sin(2\pi y_\star/h)}
\right].
\]
We next expand the effective drift introduced in Theorem~\ref{thm-exact-qp-orbits-1}
\[
c_\eps
=
\left(
\fint_{\T^d}
\tfrac{d\theta}{\mathscr V_\eps(\theta)}
\right)^{-1}.
\]
First, we have
\[
\tfrac1{\mathscr V_\eps(\theta)}
=
\tfrac1{c_0}
-
\eps\tfrac{V_1(\theta)}{c_0^2}
+
O_{y_\star}(\eps^2).
\]
Averaging over $\T^d$ gives
\[
\fint_{\T^d}
\tfrac{d\theta}{\mathscr V_\eps(\theta)}
=
\tfrac1{c_0}
-
\tfrac{\eps}{c_0^2}
\fint_{\T^d}V_1(\theta)\,d\theta
+
O_{y_\star}(\eps^2).
\]
Therefore
\[
c_\eps
=
c_0+\eps c_1+O_{y_\star}(\eps^2),
\qquad
c_1
=
\fint_{\T^d}V_1(\theta)\,d\theta.
\]
We now turn to the quasi-periodic correction to the trajectory.
The straightening function $u_\eps$ satisfies
\[
\omega\cdot\partial_\theta u_\eps(\theta)
=
\frac{c_\eps}{\mathscr V_\eps(\theta)}-1,
\qquad
\fint_{\T^d}u_\eps(\theta)\,d\theta=0.
\]
Using the expansions of $c_\eps$ and $\mathscr V_\eps$, we obtain
\[
\frac{c_\eps}{\mathscr V_\eps(\theta)}-1
=
\eps
\frac{c_1-V_1(\theta)}{c_0}
+
O_{y_\star}(\eps^2).
\]
Hence
\begin{equation}
\label{eq-u-eps-expansion}
u_\eps(\theta)
=
\eps u_1(\theta)
+
O_{y_\star}(\eps^2),
\end{equation}
where $u_1$ is the zero-average solution of
\eqref{eq-u1-small-channel}. The right-hand side of
\eqref{eq-u1-small-channel} has zero average by the definition of
$c_1$, while the Diophantine condition on $\omega$ provides the
required solution of the cohomological equation.
\\
Next, we shall follow the proof of Theorem \ref{thm-exact-qp-orbits-1}. Recall by \eqref{eq-straightening-map-x} the straightening map
\[
\Phi_\eps(x)
=
x+u_\eps(x\omega).
\]
By \eqref{eq-u-eps-expansion},
\[
\Phi_\eps(x)
=
x+\eps u_1(x\omega)+O_{y_\star}(\eps^2),
\]
uniformly for $x\in\R$. Its inverse therefore satisfies
\begin{equation}
\label{eq-Phi-inverse-first-order}
\Phi_\eps^{-1}(s)
=
s-\eps u_1(s\omega)+O_{y_\star}(\eps^2),
\end{equation}
uniformly for $s\in\R$.
Along the point-vortex trajectory,
\[
\Phi_\eps(x(t))
=
c_\eps t+\Phi_\eps(x_0).
\]
Since
\[
\Phi_\eps(x_0)
=
x_0+\eps u_1(\theta_0)+O_{y_\star}(\eps^2),
\qquad
\theta_0=x_0\omega,
\]
we obtain
\[
c_\eps t+\Phi_\eps(x_0)
=
x_0+c_\eps t
+
\eps u_1(\theta_0)
+
O_{y_\star}(\eps^2).
\]
Applying \eqref{eq-Phi-inverse-first-order}, we find
\[
\begin{aligned}
x(t)
={}&
x_0+c_\eps t
+
\eps u_1(\theta_0)
-
\eps
u_1
\big(
\bigl(
x_0+c_\eps t+\eps u_1(\theta_0)
+O_{y_\star}(\eps^2)
\bigr)\omega
\big)
+
O_{y_\star}(\eps^2).
\end{aligned}
\]
Since $u_1$ is periodic and sufficiently regular,
\[
u_1
\left(
\bigl(
x_0+c_\eps t+\eps u_1(\theta_0)
+O_{y_\star}(\eps^2)
\bigr)\omega
\right)
=
u_1(\theta_0+c_\eps t\omega)
+
O_{y_\star}(\eps),
\]
uniformly for $t\in\R$. Consequently,
\[
x(t)
=
x_0+c_\eps t
+
\eps
\left[
u_1(\theta_0)
-
u_1(\theta_0+c_\eps t\omega)
\right]
+
O_{y_\star}(\eps^2),
\]
uniformly for all $t\in\R$.
Finally, since the trajectory lies on the invariant graph,
\[
y(t)
=
\mathscr Y_{\eps,E}(x(t)\omega).
\]
The preceding expansion gives
\[
x(t)\omega
=
\theta_0+c_\eps t\omega
+
O_{y_\star}(\eps)
\]
uniformly for $t\in\R$. Using
\eqref{eq-Y-expansion-proof-orbit}, we therefore obtain
\[
y(t)
=
y_\star
+
\eps
Y_1(\theta_0+c_\eps t\omega)
+
O_{y_\star}(\eps^2).
\]
By \eqref{eq-Y1-proof-orbit},
\[
y(t)
=
y_\star
-
2h\eps
\tan\left(\tfrac{\pi y_\star}{h}\right)
\mathcal R_1
\bigl(
\theta_0+c_\eps t\omega,y_\star
\bigr)
+
O_{y_\star}(\eps^2),
\]
uniformly for all $t\in\R$. This proves
\eqref{eq-orbit-first-order-small-channel} and completes the proof of the proposition.
\end{proof}

\begin{remark}
\label{cor:zero_mean_drift}
Under the hypotheses of Proposition~\ref{prop-first-order-orbit-small-channel}, if
\[
\int_{\mathbb T^{d_+}}g_+(\theta_+)\,d\theta_+=0,
\qquad
\int_{\mathbb T^{d_-}}g_-(\theta_-)\,d\theta_-=0,
\]
then averaging formula~\eqref{eq-c1-small-channel} over the torus yields
\[
c_1=\langle V_1\rangle=0,
\]
and therefore
\begin{equation*}
c_{\varepsilon,y_*}
=
c_{0,y_*}+O(\varepsilon^2).
\label{eq:zero_mean_drift}
\end{equation*}

\end{remark}

\section{Explicit conformal  perturbation of the flat strip}
\label{subsec-explicit-conformal-example}
In this section, we illustrate the preceding general theory of quasi-periodic perturbations of the flat strip by considering a simple class of domains obtained through an explicit conformal deformation of the flat geometry. The availability of an explicit conformal map allows us to compute the Robin function directly and derive its first-order expansion with respect to the perturbation parameter. This provides a concrete model in which both the deformation of the invariant vortex trajectories and the resulting vortex dynamics can be described explicitly.\\
Consider the
flat strip
\[
S:=\R\times(0,1)
\]
and, for $|\eps|<1$, we introduce the holomorphic map
\begin{equation*}
\label{eq-explicit-Phi}
\Phi_\eps(z)
=
z+\eps e^{iz}.
\end{equation*}
We define
\[
S_\eps:=\Phi_\eps(S).
\]
Since
\[
\Phi_\eps'(z)
=
1+i\eps e^{iz},
\]
and
\[
|i\eps e^{iz}|
=
|\eps|e^{-\Im z}
\leq |\eps|<1,
\qquad z\in S,
\]
the derivative never vanishes. Moreover, since the strip is convex and
\[
\sup_{z\in S}|(e^{iz})'|
=
\sup_{z\in S}|ie^{iz}|
\leq 1,
\]
we have by the  mean value inequality
\[
|\Phi_\eps(z_1)-\Phi_\eps(z_2)|
\geq
(1-|\eps|)|z_1-z_2|,
\]
and therefore $\Phi_\eps$ is injective. Thus $\Phi_\eps$ is a conformal
diffeomorphism from $S$ onto $S_\eps$.
\\
Writing $z=x+iy$, we have
\begin{equation*}
\label{eq-explicit-Phi-real}
\Phi_\eps(x+iy)
=
X+iY,
\end{equation*}
where
\begin{equation*}
\label{eq-XY-explicit}
X
=
x+\eps e^{-y}\cos x,
\qquad
Y
=
y+\eps e^{-y}\sin x.
\end{equation*}
In particular, the two boundary components of $S_\eps$ are parametrized
by
\[
(X,Y)
=
(x+\eps\cos x,\eps\sin x)
\]
and
\[
(X,Y)
=
(x+\eps e^{-1}\cos x,
1+\eps e^{-1}\sin x).
\]
Consequently, as graphs over the physical horizontal coordinate $X$,
they admit the expansions
\begin{equation*}
\label{eq-explicit-boundaries}
Y_-(X)
=
\eps\sin X+O(\eps^2),
\qquad
Y_+(X)
=
1+\eps e^{-1}\sin X+O(\eps^2).
\end{equation*}
Thus, in the notation of the perturbative theory,
\begin{equation}
\label{eq-explicit-gpm}
g_-(X)=\sin X,
\qquad
g_+(X)=e^{-1}\sin X.
\end{equation}


\subsection{Exact Robin function and critical points}

Recall that the Robin function of the flat strip $S=\R\times(0,1)$ is
\begin{equation*}
\label{eq-explicit-R0}
R_0(y)
=
\tfrac{1}{2\pi}
\log\left(
\tfrac{2}{\pi}\sin(\pi y)
\right).
\end{equation*}
If $R_\eps$ denotes the Robin function of $S_\eps$, conformal
covariance of the Green function yields
\begin{equation*}
\label{eq-Robin-conformal-transform}
R_\eps(\Phi_\eps(z))
=
R_0(z)
+
\tfrac{1}{2\pi}\log|\Phi_\eps'(z)|.
\end{equation*}
Since $R_0$ only depends on $\Im z$, we obtain the exact formula
\begin{equation*}
\label{eq-exact-Robin-reference}
R_\eps(\Phi_\eps(x+iy))
=
\tfrac{1}{2\pi}
\log\left(
\tfrac{2}{\pi}\sin(\pi y)
\right)
+
\tfrac{1}{2\pi}
\log
\left|
1+i\eps e^{-y}e^{ix}
\right|.
\end{equation*}
The inverse conformal map can also be written explicitly in terms of the
Lambert function. Indeed, if $w=\Phi_\eps(z)$, then
\begin{equation*}
\label{eq-Phi-inverse-Lambert}
z=\Phi_\eps^{-1}(w)
=
w+iW_0\left(i\eps e^{iw}\right),
\end{equation*}
where $W_0$ denotes the principal branch of the Lambert $W$ function.
Therefore, setting
\begin{align}\label{Zeps}
Z_\eps(w)
:=
w+iW_0\left(i\eps e^{iw}\right),
\end{align}
the Robin function in physical coordinates is exactly
\begin{equation*}
\label{eq-exact-Robin-physical}
R_\eps(w)
=
\frac{1}{2\pi}
\log\left[
\frac{2}{\pi}
\sin\big(\pi\textnormal{Im} Z_\eps(w)\big)
\right]
+
\frac{1}{2\pi}
\log
\left|
1+i\eps e^{iZ_\eps(w)}
\right|.
\end{equation*}
Notice that
\[
\Phi_\eps(z+2\pi)
=
\Phi_\eps(z)+2\pi.
\]
Hence $S_\eps$ is $2\pi$-periodic in the horizontal direction and
\[
R_\eps(X+2\pi,Y)
=
R_\eps(X,Y).
\]
This example therefore corresponds to the one-frequency case
\[
d=1,
\qquad
\omega=1.
\]
The associated Robin hull is consequently
\begin{equation*}
\label{eq-explicit-Robin-hull}
\mathscr R_\eps(\theta,Y)
:=
R_\eps(\theta,Y),
\qquad
(\theta,Y)\in\T\times(0,1),
\end{equation*}
where, more explicitly,
\begin{equation*}
\label{eq-exact-hull-Robin-explicit}
\mathscr R_\eps(\theta,Y)
=
\frac{1}{2\pi}
\log\left[
\frac{2}{\pi}
\sin\big(
\pi\textnormal{Im} Z_\eps(\theta+iY)
\big)
\right]
+
\frac{1}{2\pi}
\log
\left|
1+i\eps e^{iZ_\eps(\theta+iY)}
\right|.
\end{equation*}

\paragraph{Critical points of the Robin function.}
For $\eps\neq0,$ we shall identify the critical points of the Robin function in the
physical domain $S_\eps$. Let
\[
w=X+iY=\Phi_\eps(z),
\qquad
z=x+iy,
\]
and introduce the pullback
\[
\widetilde R_\eps(x,y)
:=
R_\eps(\Phi_\eps(x+iy)).
\]
Since $\Phi_\eps$ is a conformal diffeomorphism,
\[
\nabla_{x,y}\widetilde R_\eps
=
D\Phi_\eps(x,y)^T
\nabla_{X,Y}R_\eps,
\]
and therefore
\[
\nabla_{X,Y}R_\eps(X,Y)=0
\quad\Longleftrightarrow\quad
\nabla_{x,y}\widetilde R_\eps(x,y)=0.
\]
Thus the critical-point equations may be solved in the reference
coordinates $(x,y)$, provided that the resulting points are subsequently
mapped to the physical coordinates by $\Phi_\eps$.
\\
Using conformal covariance of the Robin function, we have
\[
\widetilde R_\eps(x,y)
=
R_0(y)
+
\tfrac{1}{4\pi}
\log\left(
1-2\eps e^{-y}\sin x+\eps^2e^{-2y}
\right).
\]
The equation $\partial_x\widetilde R_\eps=0$ gives
\[
\cos x=0,
\]
so that, modulo $2\pi$, there are two possible critical points,
corresponding to
\(
x_+=\frac{\pi}{2}\) and \(
x_-=\frac{3\pi}{2}.
\)
Their vertical reference coordinates $y_\pm$, obtained from the second
critical-point equation $\partial_y\widetilde R_\eps=0$, are determined
respectively by
\[
\pi\cot(\pi y_+)
+
\tfrac{\eps e^{-y_+}}{1-\eps e^{-y_+}}
=0,
\qquad
\pi\cot(\pi y_-)
-
\tfrac{\eps e^{-y_-}}{1+\eps e^{-y_-}}
=0.
\]
For $|\eps|\ll1$, these equations yield
\[
y_\pm
=
\tfrac12
\pm
\tfrac{\eps e^{-1/2}}{\pi^2}
+
O(\eps^2).
\]
The corresponding critical points in the physical variables are
\[
P_\pm=\Phi_\eps(x_\pm+iy_\pm).
\]
Since $\cos x_\pm=0$, we obtain
\[
w_+
=
\left(
\tfrac{\pi}{2},
y_++\eps e^{-y_+}
\right),
\qquad
w_-
=
\left(
\tfrac{3\pi}{2},
y_- -\eps e^{-y_-}
\right),
\]
and hence
\begin{equation*}
\label{eq-physical-critical-points}
{
\begin{aligned}
w_+
&=
\left(
\tfrac{\pi}{2},
\tfrac12+
\eps e^{-1/2}
\left(1+\tfrac1{\pi^2}\right)
\right)
+
O(\eps^2),
\\
w_-
&=
\left(
\tfrac{3\pi}{2},
\tfrac12-
\eps e^{-1/2}
\left(1+\tfrac1{\pi^2}\right)
\right)
+
O(\eps^2).
\end{aligned}}
\end{equation*}
Thus the critical centerline $y=1/2$ of the flat strip is broken by the
periodic perturbation into two isolated critical points per spatial
period. For $\eps>0$, we can check that $w_+$ is hyperbolic, whereas $w_-$ is elliptic.
\\
This preceding computation illustrates an interesting  effect of the periodic oscillation
on the degenerate critical set of the flat strip. Indeed, for $\eps=0$,
the whole centerline
\(
y=\frac12
\)
consists of critical points of the Robin function. As soon as
$\eps\neq0$, this continuous family is destroyed: the periodic
perturbation breaks the horizontal translation symmetry and leaves only
two isolated critical points in each spatial period: one
of these points is elliptic and the other is hyperbolic. Thus the
degenerate critical line of the flat geometry splits into an alternating
periodic family of elliptic and hyperbolic critical points. {The corresponding Robin level sets and point-vortex phase portrait
are shown in Figure~\ref{fig-conformal-strip-eps03} for
$\varepsilon=0.3$.}

\begin{figure}[ht]
    \centering
    \includegraphics[width=0.7\textwidth]{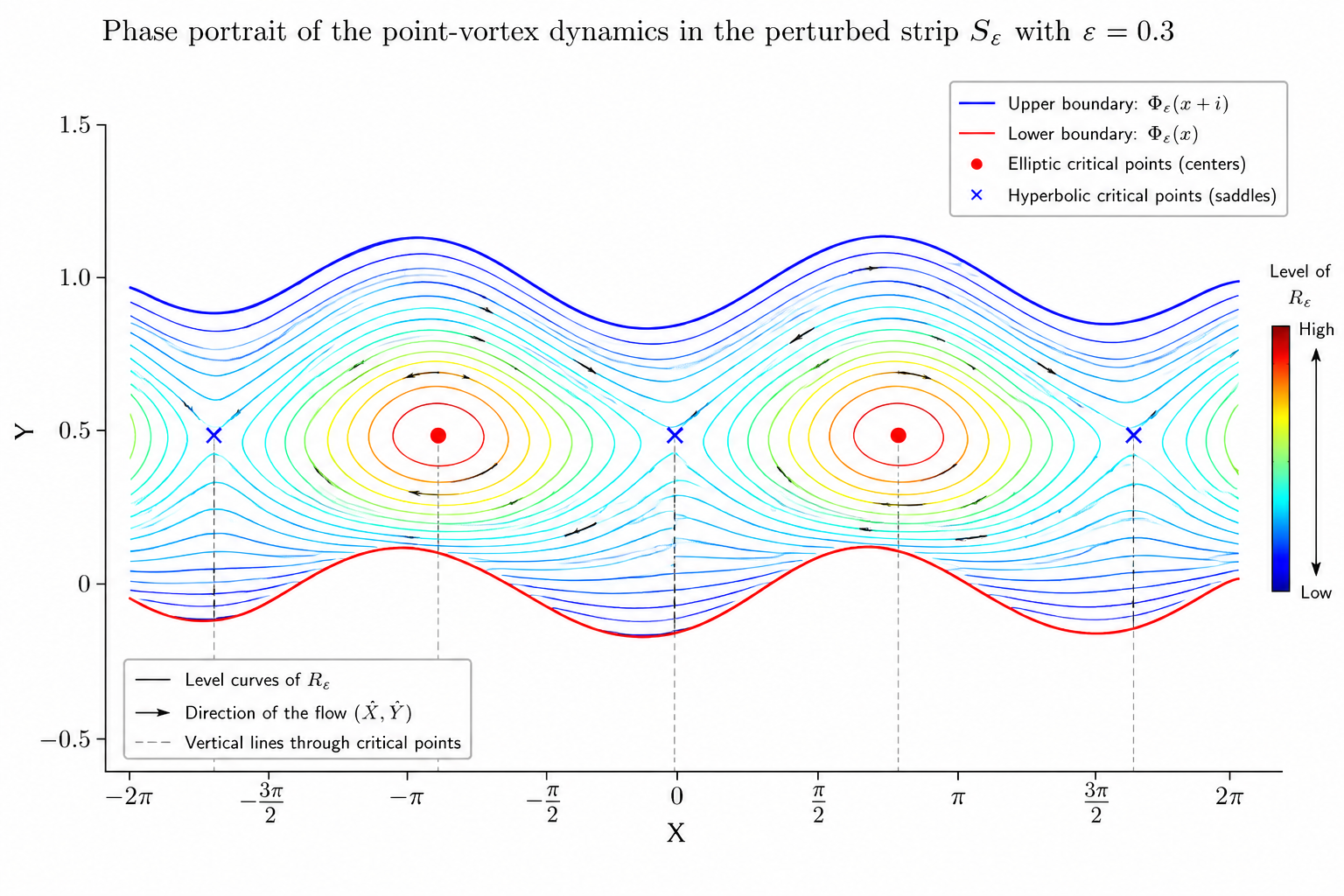}
    \caption{Phase portrait of the point-vortex dynamics in the perturbed
strip $S_\eps=\Phi_\eps(S)$, with
$\Phi_\eps(z)=z+\eps e^{iz}$ and $\eps=0.3$.
The curves are level sets of the Robin function $R_\eps$.}
    \label{fig-conformal-strip-eps03}
\end{figure}

\subsection{Geometry and first-order expansion of the vortex orbits}

The inverse map defined by \eqref{Zeps} admits the expansion
\[
Z_\eps(w)
=
w-\eps e^{iw}+O(\eps^2).
\]
For $w=X+iY$, this gives
\begin{equation*}
\label{eq-inverse-first-order}
\Im Z_\eps(X+iY)
=
Y-\eps e^{-Y}\sin X+O(\eps^2).
\end{equation*}
Furthermore,
\[
\tfrac{1}{2\pi}
\log
\left|
1+i\eps e^{iZ_\eps(w)}
\right|
=
-\tfrac{\eps}{2\pi}e^{-Y}\sin X
+
O(\eps^2).
\]
Expanding the first term of \eqref{eq-exact-Robin-physical} around $Y$
therefore yields
\[
R_0\left(
Y-\eps e^{-Y}\sin X
\right)
=
R_0(Y)
-
\eps e^{-Y}R_0'(Y)\sin X
+
O(\eps^2).
\]
Consequently,
\begin{equation*}
\label{eq-explicit-Robin-first-order}
R_\eps(X,Y)
=
R_0(Y)
+
\eps R_1(X,Y)
+
O(\eps^2),
\end{equation*}
where
\begin{align}
R_1(X,Y)
&=
-e^{-Y}
\left(
R_0'(Y)+\frac{1}{2\pi}
\right)\sin X
\nonumber\\
&=
-\tfrac{e^{-Y}}{2\pi}
\left(
\pi\cot(\pi Y)+1
\right)\sin X.
\label{eq-explicit-R1}
\end{align}
Accordingly, the first-order Robin hull is
\begin{equation*}
\label{eq-explicit-hull-first-order}
\mathscr R_\eps(\theta,Y)
=
R_0(Y)
-
\tfrac{\eps}{2\pi}{e^{-Y}}
\left(
\pi\cot(\pi Y)+1
\right)\sin\theta
+
O(\eps^2).
\end{equation*}
This formula can also be recovered directly from the Hadamard formula \eqref{eq-R1-hull-final} giving $\mathcal{R}_1$, using the boundary perturbations
\eqref{eq-explicit-gpm}. The conformal representation therefore provides
an explicit verification of the general first-order shape derivative.

\paragraph{Explicit illustration of Proposition~ \ref{cor-small-qp-exact-orbits}.}
Fix
\(
y_\star\in(0,1)\setminus\left\{\tfrac12\right\}
\)
and consider 
\[
E=R_0(y_\star).
\]
Since
\[
R_0'(y_\star)
=
\tfrac12\cot(\pi y_\star)
\neq0,
\]
From \eqref{eq-Y-leading-general}, the
invariant graph satisfies
\[
\mathscr Y_{\eps,E}(\theta)
=
y_\star
-
2\eps\tan(\pi y_\star)
 R_1(\theta,y_\star)
+
O(\eps^2).
\]
Substituting \eqref{eq-explicit-R1}, we obtain
\begin{align*}
\mathscr Y_{\eps,E}(\theta)
&=
y_\star
+
\eps e^{-y_\star}
\tan(\pi y_\star)
\left(
\cot(\pi y_\star)+\tfrac{1}{\pi}
\right)
\sin\theta
+
O(\eps^2)
\nonumber\\
&=
y_\star
+
\eps e^{-y_\star}
\left(
1+\tfrac{1}{\pi}\tan(\pi y_\star)
\right)
\sin\theta
+
O(\eps^2).
\label{eq-explicit-invariant-graph}
\end{align*}
Thus, in physical coordinates,
\begin{equation*}
\label{eq-explicit-invariant-graph-physical}
Y_{\eps,E}(X)
=
y_\star
+
\eps A(y_\star)\sin X
+
O(\eps^2),
\end{equation*}
where
\begin{equation*}
\label{eq-explicit-Aystar}
A(y_\star)
=
e^{-y_\star}
\left(
1+\tfrac{1}{\pi}\tan(\pi y_\star)
\right).
\end{equation*}
Hence the horizontal invariant line $Y=y_\star$ of the flat strip is
deformed into an explicitly sinusoidal invariant curve.

\paragraph{Explicit illustration of Proposition~\ref{prop-first-order-orbit-small-channel}.}

We next compute the first-order time parametrization. 
According to Proposition \ref{prop-first-order-orbit-small-channel},
\[
V_1(\theta)
=
\tfrac{\Gamma}{2}
\left[
\partial_y R_1(\theta,y_\star)
+
\tfrac{2\pi  R_1(\theta,y_\star)}
{\sin(2\pi y_\star)}
\right].
\]
Since $R_1(\theta,y)$ is proportional to $\sin\theta$, we may write
\[
V_1(\theta)=B(y_\star)\sin\theta.
\]
In particular,
\[
\int_{\T}V_1(\theta)\,d\theta=0,
\]
and therefore the first-order correction to the effective drift vanishes:
\begin{equation*}
\label{eq-explicit-c1-zero}
c_1=0.
\end{equation*}
Consequently,
\begin{equation*}
\label{eq-explicit-effective-drift}
c_\eps
=
\tfrac{\Gamma}{4}\cot(\pi y_\star)
+
O(\eps^2).
\end{equation*}
The first-order straightening function $u_1$ is the unique zero-average
solution of
\[
\partial_\theta u_1(\theta)
=
-\tfrac{V_1(\theta)}{c_0}.
\]
A direct computation gives
\begin{equation*}
\label{eq-explicit-u1}
u_1(\theta)
=
e^{-y_\star}
\left[
\tfrac{\tan(\pi y_\star)}{\pi}
-
\tan^2(\pi y_\star)
\right]
\cos\theta.
\end{equation*}
The expansion of Proposition~\ref{prop-first-order-orbit-small-channel} then becomes
\begin{align*}
x(t)
={}&
x_0+c_0t
+
\eps e^{-y_\star}
\left[
\tfrac{\tan(\pi y_\star)}{\pi}
-
\tan^2(\pi y_\star)
\right]
\nonumber\\
&\qquad\times
\left[
\cos x_0
-
\cos(x_0+c_0t)
\right]
+
O(\eps^2),\quad \theta_0=x_0.
\label{eq-explicit-x-orbit}
\end{align*}
For the vertical component,
\begin{equation*}
\label{eq-explicit-y-orbit}
y(t)
=
y_\star
+
\eps e^{-y_\star}
\left(
1+\tfrac{\tan(\pi y_\star)}{\pi}
\right)
\sin(x_0+c_0t)
+
O(\eps^2).
\end{equation*}
Thus the complete point-vortex trajectory takes the explicit form
\begin{equation*}
\label{eq-explicit-full-orbit}
{
\begin{aligned}
x(t)
={}&
x_0
+
\tfrac{\Gamma}{4}\cot(\pi y_\star)t
\\
&+
\eps e^{-y_\star}
\left[
\tfrac{\tan(\pi y_\star)}{\pi}
-
\tan^2(\pi y_\star)
\right]
\left[
\cos x_0
-
\cos\left(
x_0+\tfrac{\Gamma}{4}\cot(\pi y_\star)t
\right)
\right]
+
O(\eps^2),
\\[1ex]
y(t)
={}&
y_\star
+
\eps e^{-y_\star}
\left(
1+\tfrac{\tan(\pi y_\star)}{\pi}
\right)
\sin\left(
x_0+\tfrac{\Gamma}{4}\cot(\pi y_\star)t
\right)
+
O(\eps^2).
\end{aligned}
}
\end{equation*}
This example gives an explicit realization of Propositions~\ref{cor-small-qp-exact-orbits}
and~\ref{prop-first-order-orbit-small-channel}. The invariant horizontal line of the flat strip develops a
sinusoidal transverse oscillation of order $\eps$, while the constant
horizontal motion acquires a bounded periodic modulation. Since the
first-order velocity correction has zero average, the effective drift
does not change at order $\eps$ in this particular example.

{
\subsection{Periodic phase portraits beyond the conformal example}
\label{sec:periodic_phase_portraits}
The conformal example above provides an explicit realization of the
breaking of the degenerate centerline of the flat strip into isolated
elliptic and hyperbolic equilibria.  The same qualitative Hamiltonian
organization is not restricted to geometries admitting an explicit
conformal representation.
To illustrate this point, we compute the full Robin function for two
periodic channels, one preserving the reflection symmetry of the flat
strip and one breaking it.  Figure~\ref{fig:periodic_robin_portraits}
shows the resulting phase portraits.  In both cases the colored curves
are level sets of the full Robin function $R_\varepsilon$ and therefore
point-vortex trajectories.
For the symmetric channel
\[
y_\pm(x)=\pm\bigl(1+0.1\cos x\bigr),
\]
the identity
\[
R_\varepsilon(x,-y)=R_\varepsilon(x,y)
\]
forces the critical points to remain on the centerline $y=0$.
The continuous critical line of the flat strip is replaced by an
alternating periodic chain of elliptic and hyperbolic equilibria.
Closed Robin levels surround the elliptic points, while the critical
levels passing through the hyperbolic points form separatrix-type
boundaries between trapped and transporting trajectories.
\\
Breaking the vertical reflection symmetry, for example with
\[
y_+(x)=1+0.1\cos x,
\qquad
y_-(x)=-1-0.2\sin x,
\]
removes this constraint.  The elliptic and hyperbolic equilibria are
then displaced away from the centerline and the associated islands
become asymmetric, while the same basic Hamiltonian organization
persists.
\begin{figure}[h]
\centering
\begin{minipage}{0.49\textwidth}
\centering
\includegraphics[width=\textwidth]{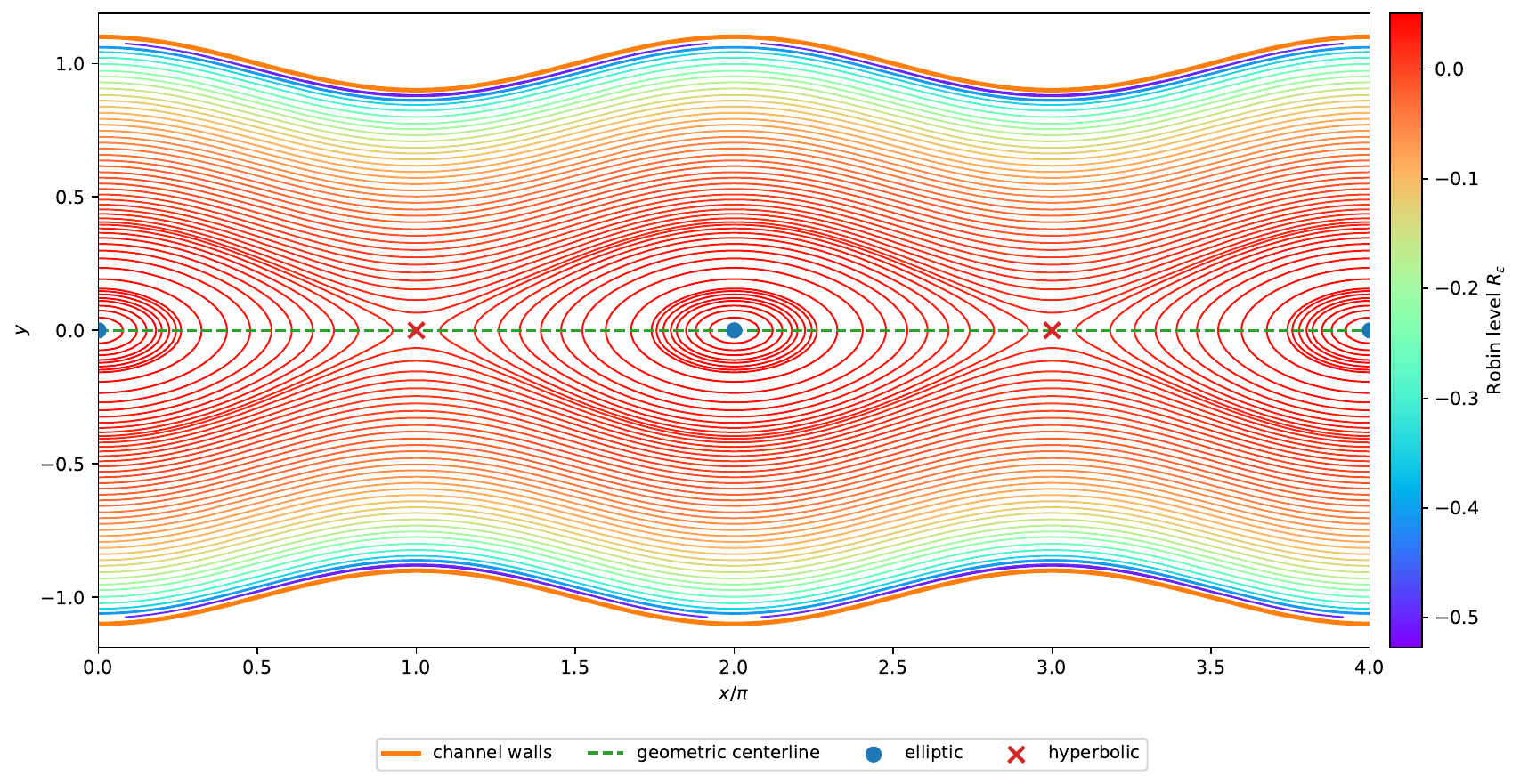}
\\[-1mm]
{\small (a) Symmetric periodic channel}
\end{minipage}
\hfill
\begin{minipage}{0.49\textwidth}
\centering
\includegraphics[width=\textwidth]{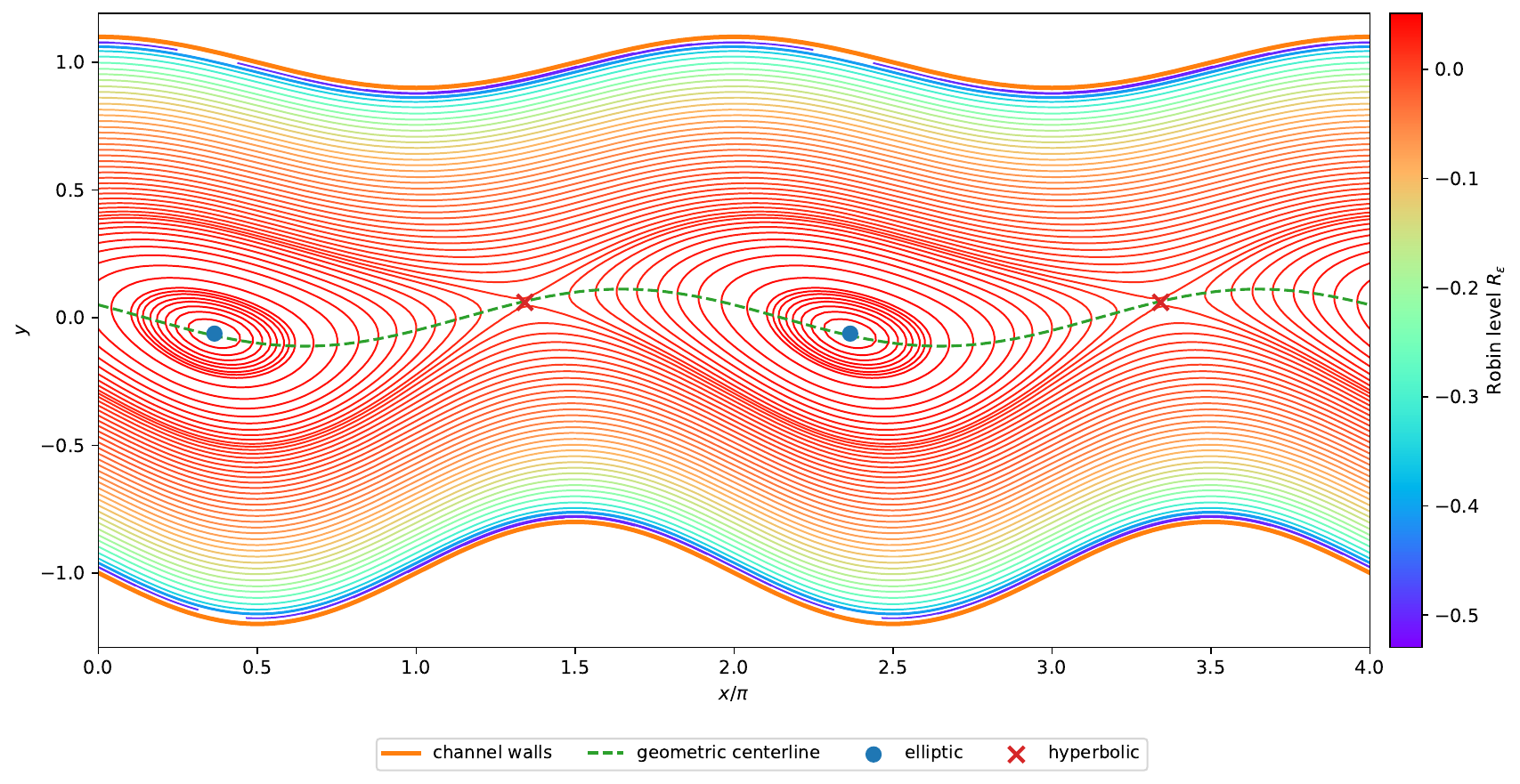}
\\[-1mm]
{\small (b) Asymmetric periodic channel}
\end{minipage}
\caption{
Full-Robin phase portraits for two periodic channels.
In panel~(a), reflection symmetry keeps the critical points on the
centerline and produces an alternating chain of elliptic and hyperbolic
equilibria.  In panel~(b), breaking this symmetry displaces the
equilibria and deforms the trapped islands.  The discrete colored curves
are level sets of the full Robin function $R_\varepsilon$, with colors
assigned from the continuous Robin-value scale.  Closed levels represent
trapped vortex motion, while the critical levels through hyperbolic
points separate these islands from transporting trajectories.
}
\label{fig:periodic_robin_portraits}
\end{figure}
These examples isolate the generic periodic mechanism: a boundary
perturbation resolves the degenerate critical centerline of the flat
strip into isolated equilibria and reorganizes the nearby Robin level
sets.  The next question is whether the number of such equilibria within
one period is itself stable when several spatial scales are present.
This leads naturally to the two-mode problem studied in the next
section, where the relative modal amplitude becomes a bifurcation
parameter.
}

\section{Periodic boundary oscillation effects on the distribution of critical points}
\label{sec-boundary-oscillation-critical-points}
We now investigate more systematically how the interaction between
different spatial modes in a \emph{periodic} boundary perturbation affects the
critical-point structure of the Robin function. In contrast with the
single-mode periodic examples considered above, the presence of a second
Fourier mode introduces a relative amplitude parameter. Varying this
parameter may change the number and location of the critical points generated
from the degenerate centerline of the flat strip.
\\
We first derive the corresponding first-order Robin function for this
two-mode periodic perturbation. We then analyze the critical points lying near
the centerline and finally describe the bifurcations in their number as the
relative amplitude of the two periodic modes varies.

\subsection{Two-mode boundary perturbations and the Robin expansion}
\label{subsec-two-mode-Robin-expansion}

We consider a flat lower boundary and perturb only the upper interface
by two Fourier modes. More precisely, let
\begin{equation}
\label{eq-two-mode-periodic-channel}
\Omega_\varepsilon
=
\left\{
(x,y)\in\mathbb R^2:
0<y<1+\varepsilon g_+(x)
\right\},
\end{equation}
where
\begin{equation*}
\label{eq-two-mode-periodic-boundary}
g_+(x)
=
\cos x+a\cos(nx),
\qquad
g_-(x)=0,
\qquad
n\in\mathbb N,\quad n\ge2,
\end{equation*}
and $a\in\mathbb R$ is a parameter measuring the relative amplitude
of the higher Fourier mode. We assume throughout that
$|\varepsilon|$ is sufficiently small so that
$\Omega_\varepsilon$ remains a smooth channel.
\\
For the perturbed channel \eqref{eq-two-mode-periodic-channel},
Lemma~\ref{lem-robin-expansion-perturbed} gives
\begin{equation*}
\label{eq-Robin-two-mode-expansion}
R_{\Omega_\varepsilon}(x,y)
=
R_0(y)
+
\varepsilon R_1(x,y)
+
O(\varepsilon^2),
\end{equation*}
uniformly on any compact subsets of the flat strip. In view of the
 Hadamard formula seen in Lemma \ref{lem-robin-expansion-perturbed}, we have
\begin{equation*}
\label{eq-R1-two-mode-integral}
R_1(x,y)
=
\int_{\mathbb R}
\left[
\cos(s+x)+a\cos(n(s+x))
\right]
P_+^2(s,y)\,ds,
\end{equation*}
where
\begin{equation*}
\label{eq-Pplus-two-mode}
P_+(s,y)
=
\frac12
\frac{\sin(\pi y)}
{\cosh(\pi s)+\cos(\pi y)}.
\end{equation*}
Since $P_+^2(\cdot,y)$ is even, the odd Fourier contributions vanish.
Consequently,
\begin{equation*}
\label{eq-R1-two-mode}
R_1(x,y)
=
A_1(y)\cos x
+
aA_n(y)\cos(nx),
\end{equation*}
where, for every integer $k\ge1$,
\begin{equation}
\label{eq-Ak-two-mode}
A_k(y)
:=
\int_{\mathbb R}
\cos(ks)P_+^2(s,y)\,ds
=
\frac{1}{2\sinh k}
\left[
\frac{k}{\pi}\cosh(ky)
-
\sinh(ky)\cot(\pi y)
\right].
\end{equation}
In particular,
\begin{equation}
\label{eq-Ak-centerline}
A_k\left(\tfrac12\right)
=
\tfrac{k}{4\pi\sinh(k/2)}.
\end{equation}
Differentiating \eqref{eq-Ak-two-mode}, we obtain
\begin{equation*}
\label{eq-Ak-prime-two-mode}
A_k'(y)
=
\frac{1}{2\sinh k}
\left[
\frac{k^2}{\pi}\sinh(ky)
-
k\cosh(ky)\cot(\pi y)
+
\pi\sinh(ky)\csc^2(\pi y)
\right],
\end{equation*}
and therefore
\begin{equation}
\label{eq-Ak-prime-centerline}
A_k'\left(\tfrac12\right)
=
\tfrac{k^2+\pi^2}{4\pi\cosh(k/2)}.
\end{equation}
We shall work with the first-order approximation
\begin{equation}
\label{eq-approx-Robin-two-mode}
\widetilde R_\varepsilon(x,y)
:=
R_0(y)
+
\varepsilon A_1(y)\cos x
+
\varepsilon aA_n(y)\cos(nx).
\end{equation}
The flat strip possesses the whole line $y=1/2$ as a critical
manifold. The purpose of the next subsection is to describe how
\eqref{eq-approx-Robin-two-mode} resolves this degeneracy into isolated
critical points.

\subsection{Critical points near the centerline}
\label{subsec-critical-points-centerline-periodic}

Differentiating \eqref{eq-approx-Robin-two-mode}, the critical-point
equations are
\begin{equation}
\label{eq-critical-system-two-mode}
\left\{
\begin{aligned}
&A_1(y)\sin x
+
anA_n(y)\sin(nx)=0,
\\
&{\tfrac12\cot(\pi y)}
+
\varepsilon
\left[
A_1'(y)\cos x
+
aA_n'(y)\cos(nx)
\right]
=0.
\end{aligned}
\right.
\end{equation}
Using
\[
\sin(nx)
=
\sin x\,U_{n-1}(\cos x),
\]
where $U_{n-1}$ denotes the Chebyshev polynomial of the second kind,
the first equation becomes
\begin{equation}
\label{eq-horizontal-factorization-two-mode}
\sin x
\left[
A_1(y)
+
anA_n(y)U_{n-1}(\cos x)
\right]
=0.
\end{equation}
It is therefore natural to distinguish the axial branches
$x=0,\pi$ from the non-axial critical points.
Define
\begin{equation*}
\label{eq-Lambda-n-definition}
\Lambda_n(a)
:=
-
\tfrac{\sinh(n/2)}
{a n^2\sinh(1/2)}.
\end{equation*}
The main goal is to prove the following result.
\begin{proposition}
\label{prop-critical-branches-two-mode}
Let $n\ge2$ and $a\neq0$. Assume that
\begin{equation*}
\label{eq-simple-Chebyshev-roots}
U_{n-1}'(c)\neq0
\end{equation*}
for every $c\in(-1,1)$ satisfying
\begin{equation}
\label{eq-Chebyshev-level-equation}
U_{n-1}(c)=\Lambda_n(a),
\end{equation}
and assume moreover that
\begin{equation}
\label{eq-axial-nondegeneracy-condition}
\Lambda_n(a)
\notin
\left\{
n,\,(-1)^{n-1}n
\right\}.
\end{equation}
Then, for $|\varepsilon|$ sufficiently small, all the critical points
of $\widetilde R_\varepsilon$ near the centerline are isolated and
nondegenerate.
There are always two axial critical points
\[
P_{0,\varepsilon}
=
(0,y_{0,\varepsilon}),
\qquad
P_{\pi,\varepsilon}
=
(\pi,y_{\pi,\varepsilon}),
\]
with
\begin{align*}
\label{eq-axial-critical-expansions}
y_{0,\varepsilon}
&=
\tfrac12
{+
\tfrac{\varepsilon}{2\pi^2}}
\left[
\tfrac{1+\pi^2}{\cosh(1/2)}
+
a\tfrac{n^2+\pi^2}{\cosh(n/2)}
\right]
+
O(\varepsilon^2),
\\
y_{\pi,\varepsilon}
&=
\tfrac12
{-
\tfrac{\varepsilon}{2\pi^2}}
\left[
\tfrac{1+\pi^2}{\cosh(1/2)}
-
a(-1)^n
\tfrac{n^2+\pi^2}{\cosh(n/2)}
\right]
+
O(\varepsilon^2).
\end{align*}
In addition, every solution
$c_j\in(-1,1)$ of
\eqref{eq-Chebyshev-level-equation} generates two non-axial critical
points
\begin{equation*}
\label{eq-off-axis-critical-points}
P_{j,\varepsilon}^{\pm}
=
\left(
\pm\arccos(c_j)+O(\varepsilon),
\tfrac12+\varepsilon Y_j+O(\varepsilon^2)
\right)
\pmod{2\pi},
\end{equation*}
where
\begin{equation*}
\label{eq-Yj-two-mode}
Y_j
=
\tfrac{1}{2\pi^2}
\left[
\tfrac{1+\pi^2}{\cosh(1/2)}\,c_j
+
a
\tfrac{n^2+\pi^2}{\cosh(n/2)}
T_n(c_j)
\right],
\end{equation*}
and $T_n$ denotes the Chebyshev polynomial of the first kind.
\end{proposition}

\begin{proof}
We first consider the solutions of
\eqref{eq-horizontal-factorization-two-mode} for which
$\sin x=0$. Modulo $2\pi$, these are
\[
x=0,
\qquad
x=\pi.
\]
For $x=0$, the second equation of
\eqref{eq-critical-system-two-mode} becomes
\[
\tfrac12\cot(\pi y)
+
\varepsilon
\left[
A_1'(y)+aA_n'(y)
\right]
=0.
\]
At $\varepsilon=0$, this equation has the unique solution
$y=1/2$. Moreover,
\begin{align*}
R_0''\left(\tfrac12\right)=-\tfrac{\pi}{2}<0.
\end{align*}
Therefore, the implicit function theorem gives a unique solution
$y_{0,\varepsilon}$ close to $1/2$. Expanding with respect to
$\varepsilon$ and using
\eqref{eq-Ak-prime-centerline} gives the suitable  expansion of $y_{0,\eps}$.
The same argument at $x=\pi$ gives the second axial critical point
$P_{\pi,\varepsilon}$ and the corresponding expansion.
\\
Let us verify their nondegeneracy. Since the mixed derivative vanishes
at $x=0,\pi$, we obtain
\begin{equation*}
\label{eq-Hessian-axial-zero}
\det D^2\widetilde R_\varepsilon(P_{0,\varepsilon})
=
\frac{\pi\varepsilon}{2}
\left[
A_1\left(\tfrac12\right)
+
an^2A_n\left(\tfrac12\right)
\right]
+
O(\varepsilon^2),
\end{equation*}
and
\begin{equation*}
\label{eq-Hessian-axial-pi}
\det D^2\widetilde R_\varepsilon(P_{\pi,\varepsilon})
=
\frac{\pi\varepsilon}{2}
\left[
-
A_1\left(\tfrac12\right)
+
a(-1)^n n^2A_n\left(\tfrac12\right)
\right]
+
O(\varepsilon^2).
\end{equation*}
Using \eqref{eq-Ak-centerline}, the leading coefficients vanish
precisely when
\[
\Lambda_n(a)=n
\qquad\text{or}\qquad
\Lambda_n(a)=(-1)^{n-1}n,
\]
respectively. Thus
\eqref{eq-axial-nondegeneracy-condition} guarantees that both axial
critical points are nondegenerate for $|\varepsilon|$ sufficiently
small.
We next consider the solutions for which $\sin x\neq0$. In this case,
\eqref{eq-horizontal-factorization-two-mode} is equivalent to
\begin{equation*}
\label{eq-off-axis-equation-y}
U_{n-1}(\cos x)
=
-
\frac{A_1(y)}
{anA_n(y)}.
\end{equation*}
At $\varepsilon=0$, the second equation of
\eqref{eq-critical-system-two-mode} imposes
\(
y=\frac12.
\)
Using \eqref{eq-Ak-centerline}, the limiting horizontal equation becomes
\[
U_{n-1}(\cos x)
=
-
\frac{A_1(1/2)}
{anA_n(1/2)}
=
\Lambda_n(a).
\]
Let $c_j\in(-1,1)$ satisfy
\[
U_{n-1}(c_j)=\Lambda_n(a).
\]
The corresponding horizontal positions are
\[
x_j^\pm
=
\pm\arccos(c_j)
\pmod{2\pi}.
\]
Since $c_j\in(-1,1)$ and
$U_{n-1}'(c_j)\neq0$, these are simple solutions of the limiting
horizontal equation. Moreover, the derivative with respect to $y$ of the function involved in the second equation
of \eqref{eq-critical-system-two-mode}  at
$\varepsilon=0$ is non zero, consequently, the Jacobian of the critical-point system is invertible
at
\[
\left(x_j^\pm,\tfrac12,0\right),
\]
and the implicit function theorem gives unique critical points
$P_{j,\varepsilon}^{\pm}$ nearby.
More precisely,
\begin{equation*}
\label{eq-Hessian-off-axis}
\det D^2\widetilde R_\varepsilon
\left(
x_j^\pm,\tfrac12
\right)
=-
{\varepsilon
\frac{an^2}{8\sinh(n/2)}}
(1-c_j^2)
U_{n-1}'(c_j)
+
O(\varepsilon^2),
\end{equation*}
which is nonzero for sufficiently small $|\varepsilon|$.
Finally, write
\[
x_{j,\varepsilon}^\pm
=
x_j^\pm+O(\varepsilon),
\qquad
y_{j,\varepsilon}
=
\tfrac12+\varepsilon Y_j+O(\varepsilon^2).
\]
Expanding the second equation of
\eqref{eq-critical-system-two-mode} gives
\[
R_0''\left(\tfrac12\right)Y_j
+
A_1'\left(\tfrac12\right)c_j
+
aA_n'\left(\tfrac12\right)T_n(c_j)
=
0.
\]
Combined with \eqref{eq-Ak-prime-centerline}, it yields the suitable formula for $Y_j.$
\end{proof}
\subsection{Amplitude-dependent bifurcations of the critical set}
\label{subsec-amplitude-bifurcations-critical}

We now count the critical points obtained above and describe how their
number changes as the relative amplitude $a$ varies. Define
\begin{equation}
\label{eq-Nn-definition}
N_n(\lambda)
:=
\#
\Big\{
c\in(-1,1):
U_{n-1}(c)=\lambda
\Big\}.
\end{equation}
The preceding analysis immediately yields the following result.

\begin{proposition}
\label{prop-number-critical-points-two-mode}
Let the assumptions of
Proposition~\ref{prop-critical-branches-two-mode} hold. Then there
exists $\varepsilon_0>0$ such that, for every
$0<|\varepsilon|<\varepsilon_0$, the Robin function
$R_{\Omega_\varepsilon}$ has exactly
\begin{equation}
\label{eq-number-critical-points-two-mode}
{
2+2N_n\bigl(\Lambda_n(a)\bigr)
}
\end{equation}
nondegenerate critical points in one spatial period.
\\
Moreover, for $|a|$ sufficiently small there are exactly two critical points per
period, whereas for $|a|$ sufficiently large there are exactly $2n$
critical points per period.
\end{proposition}

\begin{proof}
First, we note that each root
\(
c_j\in(-1,1)
\)
of the equation
\[
U_{n-1}(c)=\Lambda_n(a)
\]
produces exactly two critical points corresponding to
\[
x=\pm\arccos(c_j),
\]
while the axial branches $x=0,\pi$ contribute two additional critical
points. Hence
\[
\#\operatorname{Crit}(\widetilde R_\varepsilon)
=
2+2N_n\bigl(\Lambda_n(a)\bigr).
\]
Away from the exceptional values specified in
Proposition~\ref{prop-critical-branches-two-mode}, all these critical
points are nondegenerate. Since
\[
R_{\Omega_\varepsilon}
=
\widetilde R_\varepsilon
+
O(\varepsilon^2)
\]
in the required differentiability class on compact subsets, each
critical point of $\widetilde R_\varepsilon$ persists uniquely as a
nearby critical point of the full Robin function
$R_{\Omega_\varepsilon}$ for $|\varepsilon|$ sufficiently small.
The localization near the flat critical line also excludes additional
critical points in the region under consideration. Thus
$R_{\Omega_\varepsilon}$ and $\widetilde R_\varepsilon$ have the same
number of critical points in one period.
\\
The integer
\[
N_n\bigl(\Lambda_n(a)\bigr)
\]
is locally constant as a function of $a$ as long as
$\Lambda_n(a)$ does not cross a critical value of $U_{n-1}$ and does
not reach one of the endpoint values associated with $c=\pm1$.
Consequently, changes in the number of critical points occur at the
distinguished amplitudes
\begin{equation}
\label{eq-bifurcation-amplitudes-two-mode}
a_\mu
=
-
\frac{\sinh(n/2)}
{n^2\sinh(1/2)\,\mu},
\end{equation}
where $\mu\neq0$ is a critical value of $U_{n-1}$ on $(-1,1)$,
together with the endpoint degeneracies already excluded in
\eqref{eq-axial-nondegeneracy-condition}.
For small amplitudes,
\[
\lim_{a\to0}
|\Lambda_n(a)|
=
+\infty.
\]
Since $U_{n-1}$ is bounded on $[-1,1]$, it follows that for $|a|$ sufficiently small
\[
N_n\bigl(\Lambda_n(a)\bigr)=0.
\]
Hence there are exactly two critical
points per period.
\\
On the other hand,
\[
\lim_{|a|\to\infty}
|\Lambda_n(a)|
=
0.
\]
The polynomial $U_{n-1}$ has precisely $n-1$ simple zeros in
$(-1,1)$, given by
\[
\cos\left(\tfrac{j\pi}{n}\right),
\qquad
j=1,\ldots,n-1.
\]
Therefore, for $|a|$ sufficiently large,
\[
N_n\bigl(\Lambda_n(a)\bigr)=n-1,
\]
and hence
\[
\#\operatorname{Crit}(R_{\Omega_\varepsilon})
=
2+2(n-1)
=
2n.
\]
\end{proof}

\begin{remark}
\label{rem-two-mode-critical-bifurcation}
The preceding proposition describes explicitly how the degenerate
critical line $y=1/2$ of the flat strip is resolved by the geometry of
the oscillating interface. The fundamental mode alone produces two
critical points per period. As the amplitude $|a|$ of the $n$-th
harmonic increases, additional pairs of critical points appear when
the level $\Lambda_n(a)$ crosses an extremal value of
$U_{n-1}$. Thus the number of critical points is piecewise constant
with respect to $a$, with bifurcations occurring at a finite
collection of distinguished amplitudes. For sufficiently large
$|a|$, the $n$-th harmonic dominates and the maximal number $2n$ of
critical points per period is attained.
\end{remark}
\paragraph{Full-Robin verification of the bifurcation diagram.}
We examine the bifurcation mechanism of Proposition~\ref{prop-number-critical-points-two-mode} using the
full Robin function, without replacing it by the first-order
approximation \eqref{eq-approx-Robin-two-mode}.  We consider
\[
\Omega_\varepsilon
=
\left\{
(x,y)\in\mathbb{R}^2:
0<y<
1+\varepsilon\bigl(\cos x+a\cos(3x)\bigr)
\right\},
\]
with $n=3$ and $\varepsilon=0.05$.  In this case
\[
U_2(c)=4c^2-1,
\]
and the positive-amplitude bifurcation predicted by
Proposition~~\ref{prop-number-critical-points-two-mode} occurs at
\[
a_\ast
=
\frac{\sinh(3/2)}
{9\sinh(1/2)}
\simeq 0.4540179.
\]
For $a<a_\ast$, the leading-order theory predicts only the two axial
equilibria, whereas for $a>a_\ast$ two pairs of non-axial equilibria
are created and the total number of critical points becomes six.
\\
For each value of $a$, the Robin function is computed directly from
the perturbed geometry.  We first determine the numerical vertical
critical graph from
\[
\partial_yR_\varepsilon^{\rm num}(x,Y_c(x))=0,
\]
and subsequently locate the equilibria as the zeros of the derivative
of the reduced Robin function
\[
\widehat R_\varepsilon^{\rm num}(x)
=
R_\varepsilon^{\rm num}(x,Y_c(x)).
\]
Figure~\ref{periodic_bifurcation_full_robin} compares the resulting
full-Robin equilibrium branches with the first-order locations given by
the Chebyshev equation of Proposition~6.1.  The numerical computation
reproduces the transition from two to six critical points and places
the transition within the interval
\[
0.4525<a_{\rm bif}^{\rm num}<0.4550,
\]
which contains the analytical value $a_\ast\simeq0.4540179$.
Away from the immediate bifurcation neighborhood, the full-Robin
branches closely follow the analytical prediction.
\\
To quantify the asymptotic agreement, we fix $a=0.60$ and repeat the
calculation for decreasing values of $\varepsilon$.  The maximum
horizontal displacement between the full-Robin non-axial
equilibria and their first-order predictions satisfies
\[
\max_j
\left|
x_{j,\varepsilon}^{\rm num}-x_j^{(1)}
\right|
\propto \varepsilon^{0.994},
\]
in agreement with the expected $O(\varepsilon)$ accuracy of the first-order critical-point expansion (Figure~\ref{periodic_bifurcation_full_robin_convergence}).
\begin{figure}[t]
    \centering
    \includegraphics[width=0.7\textwidth]{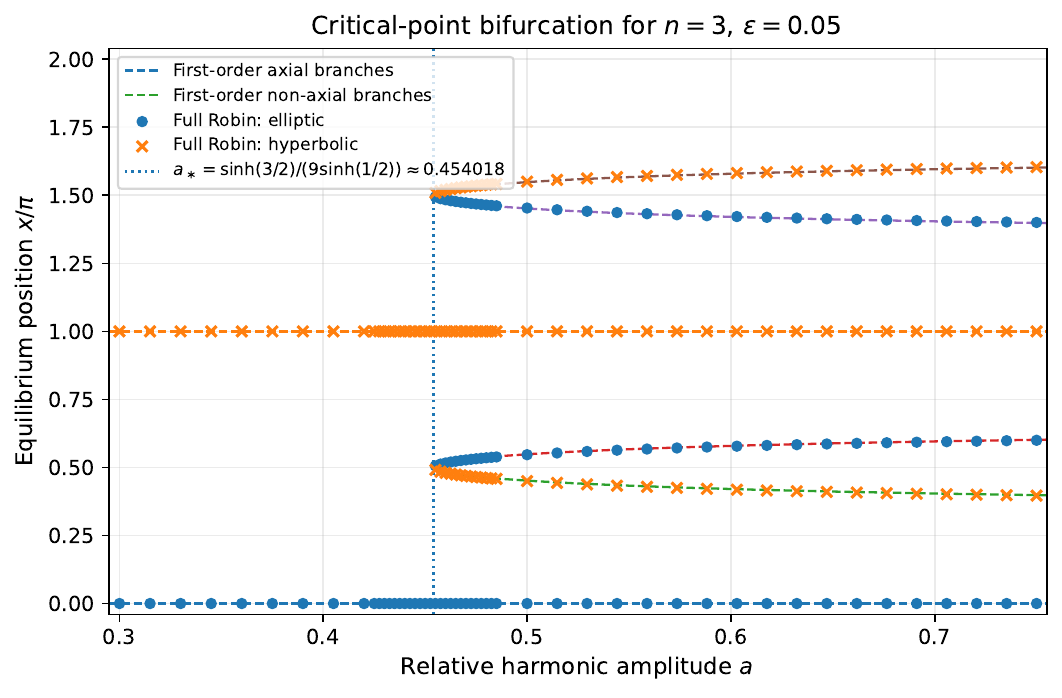}
    \caption{Bifurcation of the critical points for the two-mode periodic channel
$0<y<1+\varepsilon(\cos x+a\cos3x)$ with $\varepsilon=0.05$.
Dashed curves denote the first-order equilibrium branches obtained from the Chebyshev reduction in Proposition~\ref{prop-critical-branches-two-mode}, while the symbols denote critical points computed from the full numerical Robin function. Below the threshold only the two axial equilibria remain; above it two additional pairs of non-axial equilibria are created.}
\label{periodic_bifurcation_full_robin}
\end{figure}
\begin{figure}[h]
    \centering
    \includegraphics[width=0.7\textwidth]{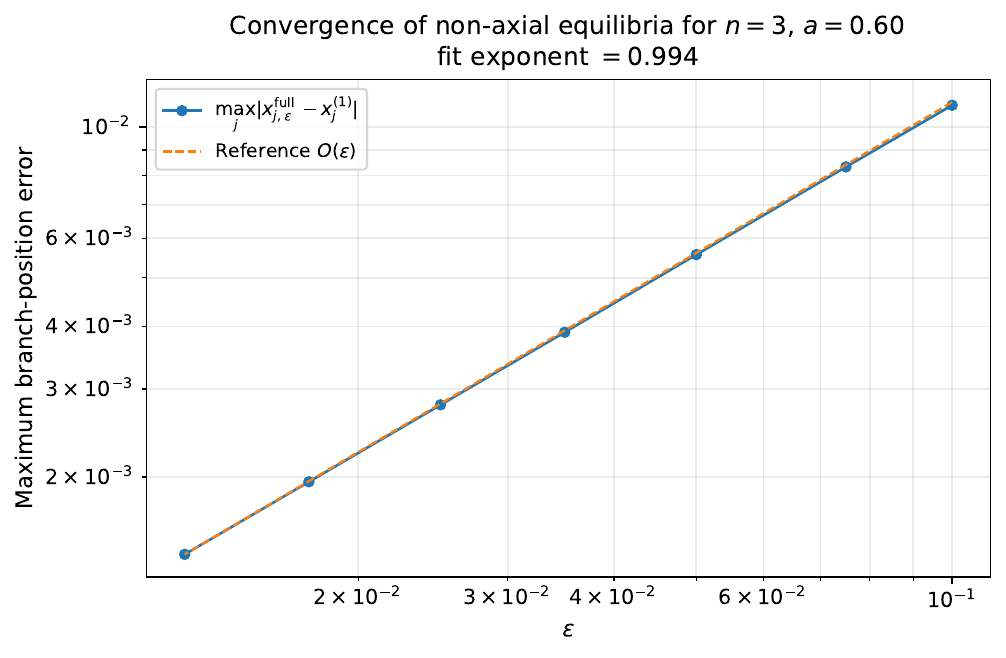}
    \caption{Convergence of the full-Robin non-axial equilibrium locations to the first-order predictions for $n=3$ and $a=0.60$. The maximum branch-position error is proportional to $\varepsilon^{0.994}$, confirming the $O(\varepsilon)$ accuracy of the perturbative critical-point location.}
    \label{periodic_bifurcation_full_robin_convergence}
\end{figure}
The periodic setting considered here provides a natural precursor to
the  quasi-periodic situation. In the present case the critical points can be counted in a single fundamental cell. For incommensurate spatial frequencies, however, no such cell exists and the appropriate quantity is instead the asymptotic number of critical
points per unit length. This motivates the ergodic analysis developed in the next section.

\section{Ergodicity and asymptotic distribution of vortex equilibria}
\label{sec-quasi-periodic-interfaces-ergodicity}

We now investigate the asymptotic distribution of the critical points of the Robin function, which correspond precisely to vortex equilibria, along a channel with quasi-periodic interfaces. In the
periodic setting, this question can be reduced to a single fundamental
cell, since the entire configuration of critical points repeats after
one spatial period. No such reduction is available in the
quasi-periodic case, where the interaction of incommensurate spatial
frequencies produces a non-repeating pattern along the channel.
Consequently, rather than attempting to locate every critical point
individually, it is more natural to study their distribution on large
spatial scales.

The main objective of this section is to show that this distribution,
although non-periodic, possesses a precise deterministic structure.
The key observation is that the quasi-periodic dependence on the
longitudinal variable can be lifted to a linear flow on a
finite-dimensional torus. The critical points can then be interpreted,
to leading order, as the successive times at which this flow intersects
a suitable hypersurface determined by the geometry of the channel.
When the underlying frequencies are rationally independent, the
corresponding linear flow is uniquely ergodic and explores the torus
uniformly in the long run. This is the mechanism through which
ergodicity naturally enters the study of the critical set.

Under an appropriate transversality condition, the crossings are
nondegenerate and their frequency can be quantified. This leads to an
asymptotic density for the number of critical points per unit length
and, more precisely, to an equidistribution law for their associated
phases on the torus. The limiting distribution is determined not only
by the ergodicity of the underlying flow, but also by the geometry and
the direction of the crossings. Thus, in the quasi-periodic setting,
unique ergodicity plays the role that spatial periodicity plays in the
periodic case, providing a natural statistical description of the
critical points despite the absence of any exact repeating pattern.

\subsection{Ergodic distribution of zeros along Kronecker flows}
\label{subsubsec-Kronecker-flow-critical}

We briefly recall the unique ergodicity of the Kronecker flow on the torus
$\T^d$ and develop a general framework for describing the asymptotic
distribution of the zeros of smooth functions \mbox{on $\T^d$} when restricted
to Kronecker trajectories. Let
\[
\omega=(\omega_1,\ldots,\omega_d)\in\mathbb R^d
\]
satisfy the nonresonance condition
\begin{equation}
\label{eq-nonresonance-Kronecker-critical}
\omega\cdot\ell\neq0,
\qquad
\forall\,\ell\in\mathbb Z^d\setminus\{0\}.
\end{equation}
We consider the associated Kronecker flow
\begin{equation*}
\label{eq-Kronecker-flow-critical}
\Phi^x_\omega:\mathbb T^d\longrightarrow\mathbb T^d,
\qquad
\Phi^x_\omega(\theta)
=
\theta+x\omega
\pmod{2\pi},
\qquad
x\in\mathbb R.
\end{equation*}
In particular, the orbit issued from the origin is
\begin{equation*}
\label{eq-Kronecker-orbit-origin-critical}
x\longmapsto
\Phi^x_\omega(0)
=
x\omega
\pmod{2\pi}.
\end{equation*}
Under the nonresonance assumption
\eqref{eq-nonresonance-Kronecker-critical}, the Kronecker flow is
minimal and uniquely ergodic, see for instance \cite{Walters1982,KatokHasselblatt1995}. Its unique invariant probability
measure is the normalized Haar measure
\begin{equation*}
\label{eq-normalized-Haar-critical}
dm_{\mathbb T^d}(\theta)
=
\tfrac{d\theta}{(2\pi)^d}.
\end{equation*}
Consequently, for every continuous function
$h\in C(\mathbb T^d)$ and every initial phase
$\theta_0\in\mathbb T^d$,
\begin{equation}
\label{eq-unique-ergodic-average-critical}
\lim_{L\to+\infty}
\frac1L
\int_0^L
h\bigl(\Phi^x_\omega(\theta_0)\bigr)\,dx
=
\fint_{\mathbb T^d}
h(\theta)\,d\theta.
\end{equation}
The convergence is uniform with respect to the initial phase
$\theta_0$.
\\
Next, let $\mathcal S:\T^d\to\R$ be a smooth function. We consider the real-valued function obtained by restricting $\mathcal S$ along a Kronecker trajectory, namely,
\begin{equation}
\label{eq-q-torus-representation}
q(x)
:=
\mathcal S\bigl(\Phi^x_\omega(0)\bigr)
=
\mathcal S(x\omega), \, x\in\R.
\end{equation}
We  introduce the zero set
\begin{equation}
\label{eq-Sigma-general-critical}
\Sigma
:=
\left\{
\theta\in\mathbb T^d:
\mathcal S(\theta)=0
\right\}.
\end{equation}
We assume that $0$ is a regular value of $\mathcal S$, namely,
\begin{equation}
 \label{eq-regular-level-set-finite-tangencies}
 \nabla\mathcal S(\theta)\neq0,
 \qquad\forall\, \theta\in\Sigma,
 \end{equation}
The zeros of $q$ are exactly the times at which the Kronecker orbit
through the origin intersects $\Sigma$:
\begin{equation}
\label{eq-zero-intersection-Kronecker}
q(x)=0
\quad\Longleftrightarrow\quad
\Phi^x_\omega(0)\in\Sigma.
\end{equation}
Thus, the zeros of $q$ correspond precisely to the successive intersections
of the Kronecker \mbox{orbit $\big{\{}\Phi_\omega^x(0):x\in\R\big\}$} with the hypersurface
$\Sigma$. Under the nonresonance assumption \eqref{eq-nonresonance-Kronecker-critical}, this orbit is dense
in $\T^d$, so the study of the zeros of $q$ can be recast as an intersection
problem between a dense Kronecker trajectory and the zero level set $\Sigma$.
We denote by
\begin{equation*}
\label{eq-flow-derivative-G}
D_\omega \mathcal S
:=
\omega\cdot\nabla_\theta \mathcal S
\end{equation*}
its derivative along the Kronecker flow. Equivalently,
\begin{equation*}
\label{eq-Domega-flow-identity-critical}
\frac{d}{dx}
\mathcal S\bigl(\Phi^x_\omega(\theta)\bigr)
=
D_\omega \mathcal S\bigl(\Phi^x_\omega(\theta)\bigr).
\end{equation*}
In particular, we have
\begin{equation*}
\label{eq-q-prime-DomegaS}
q^\prime(x)
=
D_\omega\mathcal S(x\omega).
\end{equation*}
Hence a zero $x_\ast$ of $q$ is simple precisely when the Kronecker
flow crosses $\Sigma$ transversally at the phase
\[
\Theta_\ast:=x_\ast\omega\pmod{2\pi},
\]
that is,
\begin{equation*}
\label{eq-transversal-crossing-Sigma}
\mathcal S(\Theta_\ast)=0,
\qquad
D_\omega\mathcal S(\Theta_\ast)\neq0.
\end{equation*}
Accordingly, a natural global transversality assumption is
\begin{equation}
\label{eq-global-transversality-Sigma}
D_\omega\mathcal S(\theta)\neq0,
\qquad
\forall\,\theta\in\Sigma.
\end{equation}
Under \eqref{eq-global-transversality-Sigma}, every intersection of
the Kronecker flow with $\Sigma$ is transversal, and consequently
every zero of $q$ is simple and it is isolated. In addition, since $\Sigma$ is compact, this is equivalent to the
existence of $c_\ast>0$ such that
\begin{equation}
\label{eq-global-uniform-transversality-critical}
|D_\omega\mathcal S(\theta)|
\geq c_\ast,
\qquad
\forall\,\theta\in\Sigma.
\end{equation}
In particular, all zeros of $q$ are uniformly simple.
\\Notice that \eqref{eq-global-transversality-Sigma} also implies \eqref{eq-regular-level-set-finite-tangencies}.
Therefore $0$ is a regular value of $\mathcal S$, and $\Sigma$ is a
compact smooth hypersurface of $\mathbb T^d$. This geometric
description will be the basis for the asymptotic counting and
equidistribution of the critical points of Robin functions as we shall see later.\\
For $L>0$, we define
\begin{equation}
\label{eq-N-critical-general}
N_{\mathcal S,\omega}(L)
:=
\#
\Big\{
x\in[0,L]:
q(x)=0
\Big\}.
\end{equation}
When it exists, we define the asymptotic density by
\begin{equation}
\label{eq-rho-general-critical}
\rho_{\mathcal S,\omega}
:=
\lim_{L\to\infty}
\tfrac{N_{\mathcal S,\omega}(L)}{L}.
\end{equation}
In what follows, we first consider the fully transverse setting in order to
identify the mechanism governing the asymptotic distribution of zeros and to
derive the corresponding density formula. This setting contains the essential
ergodic and geometric ingredients of the argument, while avoiding the additional
difficulties arising from tangential intersections of the Kronecker flow with
the hypersurface $\Sigma$.
In Section~\ref{subsec-isolated-finite-order-tangencies}, we substantially extend this result by allowing isolated finite-order tangencies of the Kronecker flow with the critical hypersurface. Our first main result is the following.
\begin{theorem}
\label{thm-critical-point-ergodic-distribution}
Assume that the nonresonance condition
\eqref{eq-nonresonance-Kronecker-critical} holds, that $\mathcal S$ is smooth
and satisfies the transversality condition
\eqref{eq-global-transversality-Sigma}, and that $\Sigma\neq\varnothing$. Then the
asymptotic density \eqref{eq-rho-general-critical} exists and is given
by
\begin{align*}
\label{eq-density-delta-critical}
\rho_{\mathcal S,\omega}
&=
\fint_{\mathbb T^d}
\delta_0\bigl(\mathcal S(\theta)\bigr)
\left|
D_\omega\mathcal S(\theta)
\right|
\,d\theta
\\
\nonumber
&=
\frac1{(2\pi)^d}
\int_\Sigma
|\omega\cdot n(\theta)|
\,d\sigma(\theta),
\end{align*}
where $\delta_0$ is the Dirac mass, $n$ is a normal unit vector to $\Sigma$ and $d\sigma$ denotes the surface measure induced on the hypersurface $ \Sigma.$
Let
\(
0<x_1<x_2<\cdots
\)
be the positive zeros of $q$ and define 
\begin{equation*}
\label{eq-critical-phases-general}
\Theta_n
:=
x_n\omega
\pmod{2\pi},
\end{equation*}
together with the empirical measures
\begin{equation*}
\label{eq-empirical-critical-measures-general}
\mu_N
:=
\frac1N
\sum_{n=1}^N
\delta_{\Theta_n}.
\end{equation*}
Then
\begin{equation*}
\label{eq-phase-equidistribution-weak}
{
\mu_N
\xrightharpoonup[N\to\infty]{}
\mu_{\mathcal S,\omega}
\qquad
\text{weakly in }\mathcal P(\Sigma),
}
\end{equation*}
where
\begin{equation}
\label{eq-crossing-measure-critical}
{
d\mu_{\mathcal S,\omega}(\theta)
=
\frac{1}{(2\pi)^d\rho_{\mathcal S,\omega}}
\frac{
|D_\omega\mathcal S(\theta)|
}{
|\nabla\mathcal S(\theta)|
}
\,d\sigma(\theta).
}
\end{equation}
\end{theorem}

\begin{proof}
By the nonresonance assumption, the Kronecker flow
\[
x\longmapsto x\omega\pmod{2\pi}
\]
is uniquely ergodic on $\mathbb T^d$ with respect to normalized
Lebesgue measure.
By \eqref{eq-q-torus-representation},
\[
q(x)=\mathcal S(x\omega),
\qquad
q'(x)=D_\omega\mathcal S(x\omega).
\]
From the uniform transversality assumption \eqref{eq-global-uniform-transversality-critical} and compactness of $\Sigma$ 
there exist a neighborhood $\mathcal U$ of $\Sigma$ and a constant
$c_0>0$ such that
\begin{equation}
\label{eq-transversality-neighborhood-general}
|D_\omega\mathcal S(\theta)|
\geq c_0,
\qquad
\forall \,\theta\in\mathcal U.
\end{equation}
Let $\chi\in C_c^\infty((-1,1))$ be nonnegative and normalized by
\[
\int_{\mathbb R}\chi(s)\,ds=1,
\]
and define
\[
\chi_\eta(s):=\frac1\eta\chi\left(\frac{s}{\eta}\right).
\]
By the global transversality assumption, there exist $\eta_0>0$ and
$c_0>0$ such that
\begin{equation}
\label{eq-transversality-neighborhood-general}
|D_\omega\mathcal S(\theta)|\geq c_0
\qquad\text{whenever}\qquad
|\mathcal S(\theta)|\leq\eta_0.
\end{equation}
Fix $0<\eta<\eta_0$. Then
\[
|q'(x)|
=
|D_\omega\mathcal S(x\omega)|
\geq c_0
\qquad\text{whenever}\qquad
|q(x)|<\eta.
\]
Consequently, on every connected component of the open set
\[
E_{\eta,L}
:=
\{x\in(0,L):|q(x)|<\eta\},
\]
the derivative $q'$ has a constant sign, and hence $q$ is strictly
monotone.
Since the zeros of $q$ are uniformly isolated under the transversality
assumption, $E_{\eta,L}$ has only finitely many connected components.
We may therefore write
\[
E_{\eta,L}
=
\bigcup_{i=1}^{N} I_i,
\qquad
I_i=(a_i,b_i),
\]
where the intervals are pairwise disjoint and ordered so that
\[
b_i\leq a_{i+1},
\qquad
1\leq i\leq N-1.
\]
Consider a component $I_i=(a_i,b_i)$ whose closure is contained in
$(0,L)$. By maximality of $I_i$ and continuity of $q$,
\[
|q(a_i)|=|q(b_i)|=\eta.
\]
Since $q$ is strictly monotone on $I_i$, its endpoint values are
distinct, and therefore
\[
\{q(a_i),q(b_i)\}=\{-\eta,\eta\}.
\]
It follows that
\[
q(I_i)=(-\eta,\eta),
\]
and, in particular, $q$ has exactly one zero in $I_i$. The change of
variables $s=q(x)$ then gives
\begin{equation}
\label{eq-one-crossing-contribution}
\int_{I_i}
\chi_\eta(q(x))|q'(x)|\,dx
=
\int_{-\eta}^{\eta}\chi_\eta(s)\,ds
=
1.
\end{equation}
All components of $E_{\eta,L}$ are of this type except possibly the
first and the last ones, which may intersect the endpoints $0$ and $L$,
respectively. Thus, after summing the contributions of the connected
components, we obtain
\begin{equation*}
\label{eq-regularized-zero-counting}
\int_0^L
\chi_\eta(q(x))|q'(x)|\,dx
=
N_{\mathcal S,\omega}(L)+E_{\eta,L}^{\rm end},
\end{equation*}
where $E_{\eta,L}^{\rm end}$ accounts only for the possible partial
crossings at the endpoints $0$ and $L$. Since each such partial
crossing contributes a quantity between $0$ and $1$, while it can
change the zero count by at most one, we have the uniform estimate
\begin{equation*}
\label{eq-endpoint-error-bound}
|E_{\eta,L}^{\rm end}|\leq 2.
\end{equation*}
Therefore,
\begin{equation}
\label{eq-uniform-zero-counting-density}
\left|
\frac{N_{\mathcal S,\omega}(L)}{L}
-
\frac1L
\int_0^L
\chi_\eta\bigl(\mathcal S(x\omega)\bigr)
|D_\omega\mathcal S(x\omega)|\,dx
\right|
\leq\frac{2}{L}.
\end{equation}
The crucial point is that the estimate is uniform for
$0<\eta<\eta_0$ and $L>0$.
\\
Now, for every fixed $\eta>0$, the function
\[
\theta
\longmapsto
\chi_\eta\bigl(\mathcal S(\theta)\bigr)
|D_\omega\mathcal S(\theta)|
\]
is continuous on $\mathbb T^d$. Hence, unique ergodicity result from \eqref{eq-unique-ergodic-average-critical} gives
\begin{align*}
\lim_{L\to\infty}
\frac1L
\int_0^L
&
\chi_\eta\bigl(\mathcal S(x\omega)\bigr)
|D_\omega\mathcal S(x\omega)|
\,dx
\\
&=
\frac1{(2\pi)^d}
\int_{\mathbb T^d}
\chi_\eta(\mathcal S(\theta))
|D_\omega\mathcal S(\theta)|
\,d\theta.
\end{align*}
Letting $L\to\infty$ in \eqref{eq-uniform-zero-counting-density}
yields
\begin{align}\label{cor-era}
\rho_{\mathcal S,\omega}
=
\frac1{(2\pi)^d}
\int_{\mathbb T^d}
\chi_\eta(\mathcal S(\theta))
|D_\omega\mathcal S(\theta)|
\,d\theta.
\end{align}
Since
$\chi_\eta\rightharpoonup\delta_0$ as $\eta\to0$, it follows that
\[
\rho_{\mathcal S,\omega}
=
\frac1{(2\pi)^d}
\int_{\mathbb T^d}
\delta_0(\mathcal S(\theta))
|D_\omega\mathcal S(\theta)|
\,d\theta.
\]
On the other hand, applying  the coarea formula with \eqref{cor-era},
\begin{align}
\int_{\mathbb T^d}
\chi_\eta\bigl(\mathcal S(\theta)\bigr)
|D_\omega\mathcal S(\theta)|
\,d\theta
&=
\int_{\mathbb R}
\chi_\eta(s)
\left(
\int_{\{\mathcal S=s\}}
\frac{|D_\omega\mathcal S(\theta)|}
{|\nabla\mathcal S(\theta)|}
\,d\sigma_s(\theta)
\right)ds.
\end{align}
Since $0$ is a regular value of $\mathcal S$, the function
\[
s\longmapsto
\int_{\{\mathcal S=s\}}
\frac{|D_\omega\mathcal S(\theta)|}
{|\nabla\mathcal S(\theta)|}
\,d\sigma_s(\theta)
\]
is continuous for $s$ in a neighborhood of $0$. Once again, since 
$\chi_\eta\rightharpoonup\delta_0$ as $\eta\to0$, we get
\begin{equation*}
\lim_{\eta\to0}
\int_{\mathbb T^d}
\chi_\eta\bigl(\mathcal S(\theta)\bigr)
|D_\omega\mathcal S(\theta)|
\,d\theta
=
\int_{\Sigma}
\frac{|D_\omega\mathcal S(\theta)|}
{|\nabla\mathcal S(\theta)|}
\,d\sigma(\theta).
\end{equation*}
Consequently,
\begin{equation*}
{
\rho_{\mathcal S,\omega}
=
\frac{1}{(2\pi)^d}
\int_{\Sigma}
\frac{|D_\omega\mathcal S(\theta)|}
{|\nabla\mathcal S(\theta)|}
\,d\sigma(\theta).
}
\end{equation*}
Equivalently, if
\[
n(\theta)
=
\frac{\nabla\mathcal S(\theta)}
{|\nabla\mathcal S(\theta)|}
\]
denotes the unit normal to $\Sigma$, then
\[
\frac{|D_\omega\mathcal S(\theta)|}
{|\nabla\mathcal S(\theta)|}
=
|\omega\cdot n(\theta)|,
\]
and therefore
\[
\int_{\mathbb T^d}
\delta(\mathcal S(\theta))
|D_\omega\mathcal S(\theta)|
\,d\theta
=
\int_\Sigma |\omega\cdot n(\theta)|\,d\sigma(\theta),
\]
which proves the density formula.
\\
We now prove the equidistribution statement. Let
$\psi\in C(\mathbb T^d)$. We use the same approximate identity
$\chi_\eta$ and the same decomposition into crossing intervals as in
the proof of the density formula.
Let $I_i=(a_i,b_i)$ be a complete crossing interval contained in
$(0,L)$, and denote by $x_i\in I_i$ the unique zero of $q$ in $I_i$.
Since
\[
|q'(x)|\geq c_0
\qquad\text{on }I_i,
\]
we have
\[
|x-x_i|
\leq \tfrac{|q(x)-q(x_i)|}{c_0}
\leq \tfrac{\eta}{c_0},
\qquad x\in I_i.
\]
Hence, if $\varpi_\psi$ denotes a modulus of continuity of $\psi$ on
$\mathbb T^d$, then
\[
\left|
\psi(x\omega)-\psi(x_i\omega)
\right|
\leq
\varpi_\psi\left(\tfrac{|\omega|\eta}{c_0}\right),
\qquad x\in I_i.
\]
Using \eqref{eq-one-crossing-contribution}, we therefore obtain
\begin{align*}
\left|
\int_{I_i}
\psi(x\omega)\chi_\eta(q(x))|q'(x)|\,dx
-
\psi(\Theta_i)
\right|
&\leq
\varpi_\psi\left(\tfrac{|\omega|\eta}{c_0}\right)
\int_{I_i}\chi_\eta(q(x))|q'(x)|\,dx
\nonumber\\
&\le
\varpi_\psi\left(\tfrac{|\omega|\eta}{c_0}\right),
\label{eq-weighted-crossing-error}
\end{align*}
where
\[
\Theta_i=x_i\omega\pmod{2\pi}.
\]
Summing over all complete crossing intervals contained in $(0,L)$ and
treating the two possible endpoint components as in the first part of
the proof, we obtain
\begin{align}
\left|
\frac1L\sum_{0<x_n<L}\psi(\Theta_n)
-
\frac1L\int_0^L
\psi(x\omega)\chi_\eta(q(x))|q'(x)|\,dx
\right|
&\leq
\frac{N_{\mathcal S,\omega}(L)}{L}
\varpi_\psi\left(\frac{|\omega|\eta}{c_0}\right)
+
\frac{C\|\psi\|_{L^\infty}}{L},
\label{eq-weighted-zero-counting-uniform}
\end{align}
where $C>0$ is independent of $L$ and $\eta$.
For every fixed $0<\eta<\eta_0$, the function
\[
\theta\longmapsto
\psi(\theta)
\chi_\eta\bigl(\mathcal S(\theta)\bigr)
|D_\omega\mathcal S(\theta)|
\]
is continuous on $\mathbb T^d$. Hence, by unique ergodicity of the
Kronecker flow,
\begin{align*}
\lim_{L\to\infty}
\frac1L\int_0^L
&\psi(x\omega)
\chi_\eta\bigl(\mathcal S(x\omega)\bigr)
|D_\omega\mathcal S(x\omega)|\,dx
\nonumber\\
&=
\fint_{\mathbb T^d}
\psi(\theta)
\chi_\eta\bigl(\mathcal S(\theta)\bigr)
|D_\omega\mathcal S(\theta)|
\,d\theta.
\label{eq-weighted-ergodic-mollified}
\end{align*}
Since the first part of the proof gives
\[
\frac{N_{\mathcal S,\omega}(L)}{L}
\longrightarrow
\rho_{\mathcal S,\omega},
\]
letting $L\to\infty$ in
\eqref{eq-weighted-zero-counting-uniform} yields
\begin{align}
\limsup_{L\to\infty}
\left|
\frac1L\sum_{0<x_n<L}\psi(\Theta_n)
-
\frac1{(2\pi)^d}
\int_{\mathbb T^d}
\psi(\theta)
\chi_\eta\bigl(\mathcal S(\theta)\bigr)
|D_\omega\mathcal S(\theta)|\,d\theta
\right|
\leq
\rho_{\mathcal S,\omega}
\varpi_\psi\left(\frac{|\omega|\eta}{c_0}\right).
\label{eq-weighted-limsup}
\end{align}
It remains to let $\eta\to0$. As in the first part of the proof, the
coarea formula and the regularity of the level set $\Sigma$ imply
\begin{align*}
\lim_{\eta\to0}
\int_{\mathbb T^d}
&\psi(\theta)
\chi_\eta\bigl(\mathcal S(\theta)\bigr)
|D_\omega\mathcal S(\theta)|\,d\theta
\nonumber\\
&=
\int_\Sigma
\psi(\theta)
\frac{|D_\omega\mathcal S(\theta)|}
{|\nabla\mathcal S(\theta)|}
\,d\sigma(\theta).
\label{eq-weighted-coarea-limit}
\end{align*}
Moreover,
\[
\varpi_\psi(r)\longrightarrow0
\qquad\text{as }r\to0.
\]
Letting $\eta\to0$ in \eqref{eq-weighted-limsup} therefore gives
\begin{equation}
\label{eq-weighted-critical-density-general}
\lim_{L\to\infty}
\frac1L
\sum_{0<x_n<L}
\psi(\Theta_n)
=
\frac1{(2\pi)^d}
\int_\Sigma
\psi(\theta)
\frac{|D_\omega\mathcal S(\theta)|}
{|\nabla\mathcal S(\theta)|}
\,d\sigma(\theta).
\end{equation}
Finally, since
\[
N_{\mathcal S,\omega}(L)
\sim
\rho_{\mathcal S,\omega}L,
\]
we may divide \eqref{eq-weighted-critical-density-general} by
$N_{\mathcal S,\omega}(L)/L$ to obtain
\[
\lim_{L\to\infty}
\frac1{N_{\mathcal S,\omega}(L)}
\sum_{0<x_n<L}\psi(\Theta_n)
=
\frac{1}{(2\pi)^d\rho_{\mathcal S,\omega}}
\int_\Sigma
\psi(\theta)
\frac{|D_\omega\mathcal S(\theta)|}
{|\nabla\mathcal S(\theta)|}
\,d\sigma(\theta).
\]
Equivalently, if the positive zeros are enumerated as
\[
0<x_1<x_2<\cdots,
\]
then
\[
\lim_{N\to\infty}
\frac1N\sum_{n=1}^N\psi(\Theta_n)
=
\int_\Sigma\psi(\theta)\,
d\mu_{\mathcal S,\omega}(\theta),
\]
where
\[
d\mu_{\mathcal S,\omega}(\theta)
=
\frac{1}{(2\pi)^d\rho_{\mathcal S,\omega}}
\frac{|D_\omega\mathcal S(\theta)|}
{|\nabla\mathcal S(\theta)|}
\,d\sigma(\theta).
\]
This proves \eqref{eq-crossing-measure-critical}. This completes the proof of the theorem.

\end{proof}



\subsection{Isolated finite-order tangencies}
\label{subsec-isolated-finite-order-tangencies}

The purpose of this subsection is to show that global transversality, given by \eqref{eq-global-transversality-Sigma}, is stronger than necessary. We allow the Kronecker flow to be tangent
to $\Sigma$ at finitely many points, provided that each tangency is
nondegenerate in the weaker sense of having finite order.
\\
More precisely, define the tangency set
\begin{equation}
\label{eq-tangency-set-finite-order}
\mathcal C
:=
\Big\{
\theta\in\Sigma:
D_\omega\mathcal S(\theta)=0
\Big\}.
\end{equation}
We assume that $\mathcal C$ is finite:
\begin{equation}
\label{eq-finite-tangency-set-main}
\mathcal C
=
\Big\{
\theta^{(1)},\ldots,\theta^{(M)}
\Big\}.
\end{equation}
At each $\theta^{(j)}\in\mathcal C$, we further assume that the
tangency has finite order. Namely, there exists an integer
$r_j\geq2$ such that
\begin{equation}
\label{eq-finite-order-flow-tangency-main}
D_\omega^\ell\mathcal S(\theta^{(j)})=0,
\quad
\forall\, 1\leq\ell<r_j\quad\hbox{and}\quad
D_\omega^{r_j}\mathcal S(\theta^{(j)})\neq0.
\end{equation}
Thus, along the Kronecker direction, the function $\mathcal S$
vanishes at $\theta^{(j)}$ with finite order $r_j$. In particular,
although the first derivative in the flow direction vanishes at a
tangency point, the restriction of $\mathcal S$ to the corresponding
orbit cannot vanish to infinite order there.

As we shall see later in the proofs, this finite-order condition has an important local consequence. Near
each $\theta^{(j)}$, every sufficiently short connected Kronecker
orbit segment can intersect $\Sigma$ only a uniformly bounded number
of times. The bound depends on the order $r_j$, but not on the
particular nearby orbit segment. This prevents zeros of $q$ from
accumulating with arbitrarily large multiplicity near a tangency.

On the other hand, a sufficiently thin transverse flow box around
the orbit segment through $\theta^{(j)}$ has volume of order
$\delta^{d-1}$, where $\delta$ denotes its transverse width. Unique
ergodicity of the Kronecker flow therefore implies that the frequency
of passages through such a box is also of order $\delta^{d-1}$.
Combining these two observations shows that the zeros generated near
a finite-order tangency have asymptotic density tending to zero as
the transverse width tends to zero.
This is the mechanism that allows us to recover the same density
formula as in the globally transverse case. Away from the finite set
$\mathcal C$, the intersections with $\Sigma$ are transverse and the
argument of
Theorem~\ref{thm-critical-point-ergodic-distribution} applies without
change. Near $\mathcal C$, the contribution is negligible in
asymptotic density. Thus isolated finite-order tangencies do not
affect the limiting zero density.
\\
For $L>0$, recall that the set $N_{\mathcal S,\omega}(L)$ is defined through \eqref{eq-N-critical-general} and counting 
the number of distinct zeros of $q$ in $(0,L)$. We emphasize that
zeros are counted without multiplicity.
\\
The following theorem makes the preceding discussion precise and
shows that the asymptotic zero-density formula described by Theorem \ref{thm-critical-point-ergodic-distribution} remains valid in the
presence of finitely many finite-order tangencies.
\begin{theorem}
\label{thm-density-finite-order-tangencies}
Assume
\eqref{eq-nonresonance-Kronecker-critical}, \eqref{eq-regular-level-set-finite-tangencies},
\eqref{eq-finite-tangency-set-main} and
\eqref{eq-finite-order-flow-tangency-main}.
Then the asymptotic density $\rho_{\mathcal S,\omega}$ introduced  in \eqref{eq-rho-general-critical} 
exists and is given by
\begin{align*}
\label{eq-density-finite-order-tangencies}
\rho_{\mathcal S,\omega}
&=
\frac1{(2\pi)^d}
\int_{\mathbb T^d}
\delta\bigl(\mathcal S(\theta)\bigr)
\left|
D_\omega\mathcal S(\theta)
\right|
\,d\theta
\\
\nonumber
&=
\frac{1}{(2\pi)^d}
\int_\Sigma
\frac{
\left|D_\omega\mathcal S(\theta)\right|
}{
\left|\nabla\mathcal S(\theta)\right|
}
\,d\sigma(\theta).
\end{align*}
Moreover, the convergence of the empirical measures remains valid.
\end{theorem}

The main point is to control the number of zeros generated by orbit
segments passing close to a tangency point.

\begin{lemma}
\label{lem-uniform-local-root-bound}
Let $\theta_*\in\mathcal C$ be a tangency of order $r$, namely
\[
D_\omega^\ell\mathcal S(\theta_*)=0,
\quad
1\leq\ell<r\quad\hbox{and}\quad
D_\omega^r\mathcal S(\theta_*)\neq0.
\]
Then there exists a neighborhood $U_*$ of $\theta_*$ such that every
sufficiently short orbit segment of the Kronecker flow contained in
$U_*$ intersects $\Sigma$ at most $r$ times.
\end{lemma}
\begin{proof}
Since
\[
D_\omega^r\mathcal S(\theta_*)\neq 0,
\]
continuity of $D_\omega^r\mathcal S$ implies that there exist a
neighborhood $U_*$ of $\theta_*$ and a constant $c_*>0$ such that
\begin{equation}
\label{eq-rth-derivative-lower-bound}
\left|
D_\omega^r\mathcal S(\theta)
\right|
\geq c_*
\qquad
\text{for every }\theta\in U_*.
\end{equation}
After shrinking $U_*$ if necessary, we may assume that $U_*$ is
contained in a coordinate neighborhood of $\mathbb T^d$, so that
orbit segments contained in $U_*$ can be lifted unambiguously to
$\mathbb R^d$.
\\
Let $\theta_0\in U_*$ and consider a segment of the Kronecker orbit
through $\theta_0$,
\[
\theta(t)
=
\theta_0+t\omega
\pmod{2\pi},
\qquad
t\in I,
\]
where $I\subset\mathbb R$ is an interval such that
\[
\theta_0+t\omega\in U_*
\qquad
\text{for every }t\in I.
\]
Define
\[
F_{\theta_0}(t)
:=
\mathcal S(\theta_0+t\omega),
\qquad
t\in I.
\]
Differentiating along the flow gives
\[
F_{\theta_0}^{(\ell)}(t)
=
D_\omega^\ell
\mathcal S(\theta_0+t\omega),
\qquad
1\leq \ell\leq r.
\]
In particular,
\[
F_{\theta_0}^{(r)}(t)
=
D_\omega^r
\mathcal S(\theta_0+t\omega).
\]
Hence, by \eqref{eq-rth-derivative-lower-bound},
\[
\left|
F_{\theta_0}^{(r)}(t)
\right|
\geq c_*
\qquad
\text{for every }t\in I.
\]
Using Rolle's theorem we can show that $F_{\theta_0}$ has at most $r$ distinct zeros in $I$.
Finally,
\[
F_{\theta_0}(t)=0
\quad\Longleftrightarrow\quad
\mathcal S(\theta_0+t\omega)=0
\quad\Longleftrightarrow\quad
\theta_0+t\omega\in\Sigma.
\]
Therefore every Kronecker orbit segment contained in $U_*$ intersects
$\Sigma$ at most $r$ times.
\end{proof}

We next establish the estimate which replaces uniform transversality
near the exceptional points.

\begin{lemma}
\label{lem-negligible-density-near-tangency}
Let $\theta_*\in\mathcal C$ be a tangency of finite order. Then there
exist $C>0, \delta_0>0$ and  a family of neighborhoods $B_\delta(\theta_*)$ of $\theta_*$, with 
$0<\delta<\delta_0$, such that
\begin{equation*}
\label{eq-negligible-local-zero-density}
\lim_{\delta\to0}
\limsup_{L\to\infty}
\frac{1}{L}
\#
\Big\{
x\in(0,L):
\mathcal S(x\omega)=0,\;
x\omega\in B_\delta
\Big\}
=
0.
\end{equation*}
More precisely,
\begin{equation*}
\label{eq-local-zero-density-O-delta}
\limsup_{L\to\infty}
\frac{1}{L}
\#
\Big\{
x\in(0,L):
\mathcal S(x\omega)=0,\;
x\omega\in B_\delta
\Big\}
\leq C\delta^{d-1}.
\end{equation*}
\end{lemma}
\begin{proof}
Choose \(a>0\) and \(\delta_0>0\) sufficiently small so that, for every
\(0<\delta<\delta_0\), the local flow box
\[
B_\delta(\theta_*)=B_\delta
:=
\left\{
\theta_*+t e_0+y:
|t|<a,\;
y\in e_0^\perp,\;
|y|<\delta
\right\},\, e_0:=\frac{\omega}{|\omega|}
\]
is contained in a single coordinate chart of \(\mathbb T^d\) and in
the neighborhood where Lemma~\ref{lem-uniform-local-root-bound}
applies.
The incoming and outgoing transverse sections of \(B_\delta\) are
\[
\Gamma_\delta^-
:=
\left\{
\theta_*-a e_0+y:
y\in e_0^\perp,\;
|y|<\delta
\right\},
\]
and
\[
\Gamma_\delta^+
:=
\left\{
\theta_*+a e_0+y:
y\in e_0^\perp,\;
|y|<\delta
\right\}.
\]
Consider the Kronecker orbit
\[
\theta(x)=x\omega \pmod{2\pi}.
\]
Since
\[
\dot\theta(x)=\omega=|\omega|e_0,
\]
the transverse coordinate \(y\in e_0^\perp\) remains constant while
the orbit stays inside \(B_\delta\). Consequently, every complete
passage through \(B_\delta\), from \(\Gamma_\delta^-\) to
\(\Gamma_\delta^+\), has exactly the same duration
\[
\tau_*=\tfrac{2a}{|\omega|}.
\]
Let \(P_\delta(L)\) denote the number of complete passages through
\(B_\delta\) occurring during the interval \((0,L)\). Apart from these
complete passages, the intersection of the orbit with \(B_\delta\)
over \((0,L)\) can contain at most two incomplete pieces: one
intersecting the initial time \(0\), and one intersecting the terminal
time \(L\). Each such incomplete piece has duration at most
\(\tau_*\). Therefore
\begin{equation}
\label{eq-passage-time-bound}
\int_0^L
\mathbf 1_{B_\delta}(x\omega)\,dx
=
\frac{2a}{|\omega|}\,P_\delta(L)
+
E_\delta(L),
\end{equation}
where
\[
0\leq E_\delta(L)\leq \frac{4a}{|\omega|}.
\]
In particular,
\[
P_\delta(L)
\leq
\frac{|\omega|}{2a}
\int_0^L
\mathbf 1_{B_\delta}(x\omega)\,dx.
\]
We next estimate the asymptotic frequency with which the orbit visits
\(B_\delta\). Since the boundary \(\partial B_\delta\) has zero
\(d\)-dimensional Lebesgue measure, and using \eqref{eq-unique-ergodic-average-critical},  we deduce that
\[
\lim_{L\to\infty}
\frac1L
\int_0^L
\mathbf 1_{B_\delta}(x\omega)\,dx
=
\frac{|B_\delta|}{(2\pi)^d}.
\]
Since \(B_\delta\) is, in these local coordinates, the product of an
interval of length \(2a\) in the \(e_0\)-direction and a
\((d-1)\)-dimensional ball of radius \(\delta\) in \(e_0^\perp\), we
have
\[
|B_\delta|
=
2a\,|B_\delta^{d-1}|
=
2a\,\kappa_{d-1}\delta^{d-1},
\]
where \(\kappa_{d-1}\) denotes the volume of the unit ball in
\(\mathbb R^{d-1}\).\\
Dividing \eqref{eq-passage-time-bound} by \(L\) and letting
\(L\to\infty\), we obtain
\[
\lim_{L\to\infty}
\frac{P_\delta(L)}{L}
=
\frac{|\omega|}{2a}
\frac{|B_\delta|}{(2\pi)^d}
=
\frac{|\omega|\kappa_{d-1}}{(2\pi)^d}
\delta^{d-1}.
\]
Hence
\begin{equation}
\label{eq-passage-density}
\lim_{L\to\infty}
\frac{P_\delta(L)}{L}
\leq
C\delta^{d-1},
\end{equation}
where \(C>0\) is independent of \(L\) and \(\delta\).\\
It remains to relate passages through the flow box to zeros of
\[
q(x):=\mathcal S(x\omega).
\]
By Lemma~\ref{lem-uniform-local-root-bound}, every connected
Kronecker orbit segment contained in \(B_\delta\) intersects
\(\Sigma=\{\mathcal S=0\}\) at most \(r\) times. Hence every complete
passage through \(B_\delta\) contributes at most \(r\) zeros of
\(q\). The at most two incomplete passages meeting the endpoints
\(0\) and \(L\) contribute at most \(2r\) additional zeros. Therefore
\[
\#
\Big\{
x\in(0,L):
q(x)=0,\;
x\omega\in B_\delta
\Big\}
\leq
rP_\delta(L)+2r.
\]
Combining with  \eqref{eq-passage-density}, we conclude that
\[
\limsup_{L\to\infty}
\frac1L
\#
\Big\{
x\in(0,L):
q(x)=0,\;
x\omega\in B_\delta
\Big\}
\leq
Cr\,\delta^{d-1}.
\]
Since \(r\) can be taken uniform  as it is related to the degree of degeneracy of $\theta_*$, then
\[
\limsup_{L\to\infty}
\frac1L
\#
\Big\{
x\in(0,L):
q(x)=0,\;
x\omega\in B_\delta
\Big\}
\leq
C\delta^{d-1}.
\]
Finally, letting \(\delta\to0\) yields
the desired result and  this completes the proof of the lemma.
\end{proof}

\begin{proof}[Proof of Theorem~$\ref{thm-density-finite-order-tangencies}$]
For $\delta>0$ sufficiently small, let
\[
U_\delta:=\bigcup_{j=1}^M U_\delta(\theta^{(j)}),
\]
where the sets $U_\delta(\theta^{(j)})$ are pairwise disjoint
geodesic balls in $\mathbb T^d$, centered at the tangency points
$\theta^{(j)}$ in \eqref{eq-finite-tangency-set-main}, and chosen so that
\[
U_\delta(\theta^{(j)})\subset B_\delta(\theta^{(j)}),
\qquad j=1,\ldots,M,
\]
where $B_\delta(\theta^{(j)})$ denotes the transverse flow box
constructed in Lemma~\ref{lem-negligible-density-near-tangency}.
The radii of these balls are chosen to converge to zero as
$\delta\to0$. Consequently,
\[
\bigcap_{\delta>0}\overline{U_\delta}
=\mathcal C
=\{\theta^{(1)},\ldots,\theta^{(M)}\}.
\]
Since $\mathcal C$ is finite,
Lemma~\ref{lem-negligible-density-near-tangency} gives
\begin{equation}
\label{eq-total-negligible-density-tangencies}
\lim_{\delta\to0}
\limsup_{L\to\infty}
\frac1L
\#
\Big\{
x\in(0,L):
q(x)=0,\;
x\omega\in U_\delta
\Big\}
=
0.
\end{equation}
Choose a cut-off function \(\chi_\delta\in C^\infty(\mathbb T^d)\) such that
\(
0\leq\chi_\delta\leq1,
\)
and vanishing 
on a neighborhood of $\mathcal C$ contained in $U_\delta$, and
\[
\chi_\delta(\theta)=1, \quad \forall \theta\notin U_\delta.
\]
Since the support of $\chi_\delta$ is a compact subset of
$\mathbb T^d\setminus\mathcal C$, there exists
$c_\delta>0$ such that
\[
|D_\omega\mathcal S(\theta)|
\geq c_\delta,\, \forall\theta\in\Sigma\cap\operatorname{supp}\chi_\delta.
\]
Applying \eqref{eq-weighted-critical-density-general} gives
\begin{equation}
\label{eq-weighted-density-away-tangencies}
\lim_{L\to\infty}
\frac1L
\sum_{\substack{0<x<L\\q(x)=0}}
\chi_\delta(x\omega)
=
\frac1{(2\pi)^d}
\int_\Sigma
\chi_\delta(\theta)
\frac{
|D_\omega\mathcal S(\theta)|
}{
|\nabla\mathcal S(\theta)|
}
\,d\sigma(\theta).
\end{equation}
On the other hand, by writing
$$
N_{\mathcal S,\omega}(L)
-
\sum_{\substack{0<x<L\\q(x)=0}}
\chi_\delta(x\omega)
=
\sum_{\substack{0<x<L\\q(x)=0}}
\bigl(1-\chi_\delta(x\omega)\bigr).
$$
we get
\[
0
\leq
N_{\mathcal S,\omega}(L)
-
\sum_{\substack{0<x<L \\q(x)=0}}
\chi_\delta(x\omega)
\leq
\#
\Big\{
x\in(0,L):
q(x)=0,\;
x\omega\in U_\delta
\Big\}.
\]
Hence, by
\eqref{eq-total-negligible-density-tangencies},
\begin{equation}
\label{eq-counting-cutoff-error}
\lim_{\delta\to0}
\limsup_{L\to\infty}
\left|
\frac{N_{\mathcal S,\omega}(L)}{L}
-
\frac1L
\sum_{\substack{0<x<L\\q(x)=0}}
\chi_\delta(x\omega)
\right|
=
0.
\end{equation}
Finally, by
\eqref{eq-regular-level-set-finite-tangencies},
the function
\[
\theta
\longmapsto
\frac{
|D_\omega\mathcal S(\theta)|
}{
|\nabla\mathcal S(\theta)|
}
\]
is continuous on the compact hypersurface $\Sigma$.
Since the neighborhoods $U_\delta\cap\Sigma$ shrink to the finite
set $\mathcal C$, which has zero surface measure,
\[
\lim_{\delta\to0}
\int_{\Sigma\cap U_\delta}
\frac{
|D_\omega\mathcal S(\theta)|
}{
|\nabla\mathcal S(\theta)|
}
\,d\sigma(\theta)
=
0.
\]
Consequently,
\begin{equation}
\label{eq-cutoff-integral-limit}
\lim_{\delta\to0}
\int_\Sigma
\chi_\delta(\theta)
\frac{
|D_\omega\mathcal S(\theta)|
}{
|\nabla\mathcal S(\theta)|
}
\,d\sigma(\theta)
=
\int_\Sigma
\frac{
|D_\omega\mathcal S(\theta)|
}{
|\nabla\mathcal S(\theta)|
}
\,d\sigma(\theta).
\end{equation}
Combining
\eqref{eq-weighted-density-away-tangencies},
\eqref{eq-counting-cutoff-error}, and
\eqref{eq-cutoff-integral-limit}, we conclude that
\[
\lim_{L\to\infty}
\frac{N_{\mathcal S,\omega}(L)}{L}
=
\frac1{(2\pi)^d}
\int_\Sigma
\frac{
|D_\omega\mathcal S(\theta)|
}{
|\nabla\mathcal S(\theta)|
}
\,d\sigma(\theta).
\]
This proves the density result. The convergence of the empirical measures can be established by adapting the argument used in the proof of Theorem~\ref{thm-critical-point-ergodic-distribution}, together with a suitable cut-off procedure near the tangency points.
\end{proof}

\subsection{Application to the density distribution of critical points of the Robin function}
\label{sec-nonperturbative-Robin-density}

We now apply the preceding ergodic zero-density results to the critical
points of the exact Robin function associated with a quasi-periodic
channel. The argument is nonperturbative: the essential ingredients are
the existence of a sufficiently regular Robin hull and a global
vertical critical graph.
Let $\Omega$ be a channel with quasi-periodic interfaces as described in Section \ref{sec-two}. We assume that the frequency  $\omega\in\mathbb R^d$ satisfies the nonresonance condition \eqref{eq-nonresonance-Kronecker-critical}.\\
In Proposition \ref{prop-exact-hull-Robin}, we proved the existence  of a hull Robin function $\mathcal{R}:\mathcal{D}\to\R$ such that
\begin{equation*}
\label{eq-exact-Robin-hull-density}
R_\Omega(x,y)
=
\mathcal R(x\omega,y),
\end{equation*}
The assumption of uniform separation \eqref{separation1}, together with the boundary regularity, seen in Proposition \ref{prop-regularity-Robin-hull}, ensures that
\[
\mathcal R:\mathbb T^d\times I\longrightarrow\mathbb R
\]
is sufficiently smooth and $I$ is a compact interval contained in the
common vertical region of the channel. As usual, $x\omega$ is
understood modulo $2\pi$. In particular,
\begin{equation*}
\label{eq-x-derivative-exact-Robin}
\partial_xR_\Omega(x,y)
=
(D_\omega\mathcal R)(x\omega,y).
\end{equation*}
We are interested in  critical points of $R_\Omega$, namely
solutions of
\begin{equation*}
\label{eq-full-Robin-criticality}
\partial_xR_\Omega(x,y)=0,
\qquad
\partial_yR_\Omega(x,y)=0.
\end{equation*}
In terms of the hull variables, these conditions read
\begin{equation*}
\label{eq-full-Robin-criticality-hull}
D_\omega\mathcal R(\theta,y)=0,
\qquad
\partial_y\mathcal R(\theta,y)=0,
\qquad
\theta=x\omega.
\end{equation*}
Assume that there exists a smooth periodic graph
\begin{equation*}
\label{eq-vertical-critical-graph}
\mathcal Y:\mathbb T^d\longrightarrow I
\end{equation*}
such that
\begin{equation}
\label{eq-vertical-critical-equation}
\partial_y\mathcal R
\bigl(\theta,\mathcal Y(\theta)\bigr)
=0,
\qquad
\theta\in\mathbb T^d.
\end{equation}
Define the reduced Robin hull
\begin{equation*}
\label{eq-reduced-Robin-hull}
\widehat{\mathcal R}(\theta)
:=
\mathcal R\bigl(\theta,\mathcal Y(\theta)\bigr).
\end{equation*}
Differentiating along the direction $\omega$ and using
\eqref{eq-vertical-critical-equation}, we obtain
\begin{align}
D_\omega\widehat{\mathcal R}(\theta)
&=
(D_\omega\mathcal R)
\bigl(\theta,\mathcal Y(\theta)\bigr)
+
\partial_y\mathcal R
\bigl(\theta,\mathcal Y(\theta)\bigr)
D_\omega\mathcal Y(\theta)
\nonumber\\
&=
(D_\omega\mathcal R)
\bigl(\theta,\mathcal Y(\theta)\bigr).
\label{eq-reduced-Robin-derivative}
\end{align}
Accordingly, set
\begin{equation*}
\label{eq-S-reduced-Robin}
\mathcal S_R(\theta)
:=
D_\omega\widehat{\mathcal R}(\theta)
\end{equation*}
and
\begin{equation*}
\label{eq-Sigma-reduced-Robin}
\Sigma_R
:=
\left\{
\theta\in\mathbb T^d:
\mathcal S_R(\theta)=0
\right\}.
\end{equation*}
The corresponding vertical critical graph in the physical variables is
\begin{equation*}
\label{eq-physical-vertical-critical-graph}
Y(x):=\mathcal Y(x\omega).
\end{equation*}
By construction,
\[
\partial_yR_\Omega(x,Y(x))=0,
\]
while \eqref{eq-reduced-Robin-derivative} gives
\[
\partial_xR_\Omega(x,Y(x))
=
\mathcal S_R(x\omega).
\]
Consequently,
\begin{equation}
\label{eq-Robin-critical-zero-equivalence}
\nabla R_\Omega(x,Y(x))=0
\quad\Longleftrightarrow\quad
\mathcal S_R(x\omega)=0.
\end{equation}
The distribution of the critical points of the Robin function along
the graph is therefore reduced exactly to the zero-distribution problem
studied in the preceding sections.
\\
For $L>0$, define
\begin{equation*}
\label{eq-Robin-critical-counting}
N_R(L)
:=
\#
\Big\{
x\in(0,L):\,
\mathcal S_R(x\omega)=0
\Big\}.
\end{equation*}
The following result is an immediate consequence of Theorem~\ref{thm-critical-point-ergodic-distribution}.
\begin{theorem}
\label{thm-nonperturbative-Robin-critical-density}
Assume \eqref{eq-nonresonance-Kronecker-critical} and suppose that the Robin
hull $\mathcal R$ admits a smooth vertical critical graph
$\mathcal Y$ satisfying \eqref{eq-vertical-critical-equation}.
 Assume moreover that  the Kronecker flow is transverse to $\Sigma_R$, namely
\begin{equation}
\label{eq-global-transversality-reduced-Robin}
D_\omega\mathcal S_R(\theta)\neq0,
\qquad
\theta\in\Sigma_R.
\end{equation}
Then the asymptotic density
\begin{equation*}
\label{eq-Robin-critical-density-definition}
\rho_R
:=
\lim_{L\to\infty}\frac{N_R(L)}{L}
\end{equation*}
exists and is given by
\begin{equation*}
\label{eq-nonperturbative-Robin-density}
{
\rho_R
=
\frac{1}{(2\pi)^d}
\int_{\Sigma_R}
\frac{
|D_\omega\mathcal S_R(\theta)|
}{
|\nabla_\theta\mathcal S_R(\theta)|
}
\,d\sigma(\theta)
=
\fint_{\mathbb T^d}
\delta_0\bigl(\mathcal S_R(\theta)\bigr)
|D_\omega\mathcal S_R(\theta)|
\,d\theta.
}
\end{equation*}
The convergence result for the empirical measures remains valid, with the same limiting measure as in Theorem~$\ref{thm-critical-point-ergodic-distribution}.$
\end{theorem}


\begin{remark}
\label{rem-Robin-density-nonperturbative}
The preceding result is entirely nonperturbative: the density argument
itself does not require any smallness assumption on the quasi-periodic
deformation of the channel. The essential geometric hypothesis is instead
the existence of a global vertical critical graph
\[
\mathcal Y:\mathbb T^d\to I.
\]
For a general quasi-periodic channel, the existence of such a graph must
be established independently. In the perturbative regime, however, when
the channel is a sufficiently small quasi-periodic deformation of a flat
strip, the critical graph can be constructed by means of the implicit
function theorem in an appropriate space of periodic functions. This
construction will be carried out below.
\end{remark}

\begin{remark}
The preceding theorem assumes global transversality. This hypothesis
can be relaxed by using the finite-order tangency version of Theorem \ref{thm-density-finite-order-tangencies}.
Let
\[
\mathcal C_R
:=
\left\{
\theta\in\Sigma_R:
D_\omega\mathcal S_R(\theta)=0
\right\}
\]
be the tangency set. Suppose that $\mathcal C_R$ is finite and that,
for every $\theta_*\in\mathcal C_R$, there exists an integer
$r=r(\theta_*)\geq2$ such that
\begin{equation}
\label{eq-finite-order-Robin-tangency}
D_\omega^j\mathcal S_R(\theta_*)=0,
\qquad
1\leq j<r,
\qquad
D_\omega^r\mathcal S_R(\theta_*)\neq0.
\end{equation}
Then the same density formula remains valid:
\begin{equation}
\label{eq-Robin-density-finite-tangencies}
{
\rho_R
=
\frac{1}{(2\pi)^d}
\int_{\Sigma_R}
\frac{
|D_\omega\mathcal S_R(\theta)|
}{
|\nabla_\theta\mathcal S_R(\theta)|
}
\,d\sigma(\theta)
=
\frac{1}{(2\pi)^d}
\int_{\Sigma_R}
|\omega\cdot n_R(\theta)|
\,d\sigma(\theta).
}
\end{equation}
This extension is particularly relevant in the nonperturbative setting.
Indeed, for a general quasi-periodic channel, there is no reason to expect
the Kronecker flow to intersect the critical hypersurface $\Sigma_R$
everywhere transversally. The finite-order result shows that isolated
tangencies do not affect the leading asymptotic density, provided that
their degeneracy is of finite order. As we shall see, such degeneracies
already arise in the perturbative regime. In particular,
Proposition~\ref{prop-two-frequency-geometry} provides examples for which
\eqref{eq-finite-order-Robin-tangency} holds with tangency order strictly
larger than two.
\end{remark}
\subsection{Application of the
nonperturbative density theorem}
\label{subsec-centerline-exact-density}

We now apply Theorem~\ref{thm-nonperturbative-Robin-critical-density}
to a small quasi-periodic perturbation of the flat strip. The main
point is that the perturbative analysis is used only to construct the
vertical critical graph and to verify the geometric hypotheses of the
general theorem. Once this has been done, the asymptotic density of
critical points follows directly from the exact Robin hull and is
therefore an exact statement for the full Robin function.
We consider
\begin{equation*}
\label{eq-QP-upper-perturbation-centerline}
g_-(x)=0,
\qquad
g_+(x)
=
\sum_{j=1}^d f_j(\omega_jx),
\end{equation*}
where
\begin{align}\label{REg-smooth}
f_j\in C^{m+3,\alpha}(\mathbb T;\mathbb R),
\qquad
0<\alpha<1, m\geq2
\end{align}
are $2\pi$-periodic, and the frequency 
\(
\omega=(\omega_1,\ldots,\omega_d)
\)
satisfies the nonresonance condition \eqref{eq-nonresonance-Kronecker-critical}.
Let $\mathcal R_\varepsilon(\theta,y)$ denote the exact Robin hull
of the perturbed channel, and recall the identity
\begin{align}\label{robin-link}
R_{\Omega_\varepsilon}(x,y)
=
\mathcal R_\varepsilon(x\omega,y).
\end{align}
For every fixed
\(
0<\delta<\tfrac12,
\)
the perturbative construction developed in Lemma \ref{lem-robin-expansion-perturbed} gives
\begin{equation}
\label{eq-Robin-hull-expansion-centerline}
\mathcal R_\varepsilon(\theta,y)
=
R_0(y)
+
\varepsilon\mathcal R_1(\theta,y)
+
\varepsilon^2\mathcal R_{2,\varepsilon}(\theta,y), \,
(\theta,y)
\in
\mathbb T^d
\times
\left[
\tfrac12-\delta,\tfrac12+\delta
\right],
\end{equation}
where the remainder is uniformly bounded in the $C^{m+2}\left(\mathbb T^d
\times
\left[
\tfrac12-\delta,\tfrac12+\delta
\right]\right)$ for
$|\varepsilon|$ sufficiently small. On the other hand,
\begin{equation}
\label{eq-flat-centerline-nondegeneracy}
R_0'\left(\tfrac12\right)=0,
\qquad
R_0''\left(\tfrac12\right)=-\tfrac{\pi}{2}<0.
\end{equation}
The nondegeneracy in
\eqref{eq-flat-centerline-nondegeneracy} allows us first to solve the
vertical critical-point equation globally on the phase torus.

\begin{proposition}
\label{prop-exact-vertical-critical-graph-centerline}
There exist $\varepsilon_0>0$ and $\delta_0>0$ such that, for every
$|\varepsilon|<\varepsilon_0$, there exists a unique periodic function
\[
\mathcal Y_\varepsilon:\mathbb T^d\longrightarrow
\left(
\tfrac12-\delta_0,\tfrac12+\delta_0
\right)
\]
of class $C^m(\T^d)$ and satisfying
\begin{equation}
\label{eq-exact-vertical-critical-graph}
\partial_y\mathcal R_\varepsilon
\bigl(
\theta,\mathcal Y_\varepsilon(\theta)
\bigr)
=
0,
\qquad
\theta\in\mathbb T^d.
\end{equation}
Moreover,
\begin{equation}
\label{eq-Yepsilon-expansion-exact}
\mathcal Y_\varepsilon(\theta)
=
\tfrac12
+
\tfrac{2\varepsilon}{\pi}
\partial_y\mathcal R_1\left(\theta,\tfrac12\right)
+
\varepsilon^2\mathcal Y_{2,\varepsilon}(\theta),
\end{equation}
where
\[
\mathcal Y_{2,\varepsilon}\in C^m(\T^d)
\]
and, for some constant $C>0$ independent of $\varepsilon$,
\begin{equation}
\label{eq-Y2-uniform-C1}
\sup_{|\varepsilon|\leq\varepsilon_0}
\|\mathcal Y_{2,\varepsilon}\|_{C^m(\T^d)}
\leq C.
\end{equation}

\end{proposition}

\begin{proof}
The point is to solve the vertical critical-point equation directly
in a Banach space of periodic functions, rather than solving it
pointwise in $\theta$. Let $m\ge0$ as in \eqref{REg-smooth} and 
set
\[
X:=C^m(\mathbb T^d),\quad
Y_0(\theta)\equiv\tfrac12.
\]
For $\eta>0$, introduce the open neighborhood
\[
\mathcal U_\eta
:=
\left\{
Y\in X:
\left\|Y-Y_0\right\|_{C^m}<\eta
\right\}.
\]
Choosing $\eta<\delta$, every $Y\in\mathcal U_\eta$ takes its values
in
\(
\left(
\tfrac12-\delta,\tfrac12+\delta
\right).
\)
Define the nonlinear operator
\begin{equation*}
\label{eq-functional-IFT-operator}
\mathscr F:
(-\varepsilon_0,\varepsilon_0)
\times\mathcal U_\eta
\longrightarrow X
\end{equation*}
by
\begin{equation*}
\label{eq-functional-F-definition}
\mathscr F(\varepsilon,Y)(\theta)
:=
\partial_y\mathcal R_\varepsilon
\bigl(
\theta,Y(\theta)
\bigr),\,\,\theta\in\T^d.
\end{equation*}
By \eqref{REg-smooth} and the  regularity of the Robin hull following from Proposition \ref{prop-regularity-Robin-hull} together with the standard
regularity of composition operators on $C^m(\mathbb T^d)$,
$\mathscr F$ is of class $C^1$.
For the flat strip,
\[
\mathcal R_0(\theta,y)=R_0(y),
\]
and therefore
\[
\mathscr F(0,Y_0)(\theta)
=
R_0'\left(\tfrac12\right)
=
0.
\]
We now compute the differential with respect to the function
variable $Y$. For $H\in X$,
\begin{align*}
\partial_Y\mathscr F(\varepsilon,Y)[H](\theta)
&=
\partial_{yy}\mathcal R_\varepsilon
\bigl(
\theta,Y(\theta)
\bigr)
H(\theta).
\label{eq-linearized-functional-F}
\end{align*}
Consequently,  we get by \eqref{eq-flat-centerline-nondegeneracy}
\begin{eqnarray*}
\partial_Y\mathscr F(0,Y_0)[H](\theta)
&=&
R_0''\left(\tfrac12\right)H(\theta)
\\
&=&
-\tfrac{\pi}{2}H(\theta).
\end{eqnarray*}
Thus
\begin{equation}\label{oso-dec}
\partial_Y\mathscr F(0,Y_0)
=
-\tfrac{\pi}{2} I_X.
\end{equation}
This is an isomorphism from $X$ onto itself.
The implicit function theorem in Banach spaces therefore yields
$\varepsilon_1>0$ and a unique $C^1$ map
\[
\varepsilon\in (-\varepsilon_1,\varepsilon_1)
\longmapsto
\mathcal Y_\varepsilon\in X
\]
such that
\[
\mathcal Y_0=Y_0=\tfrac12
\quad\hbox{and}\quad
\mathscr F(\varepsilon,\mathcal Y_\varepsilon)=0.
\]
Equivalently,
\[
\partial_y\mathcal R_\varepsilon
\bigl(
\theta,\mathcal Y_\varepsilon(\theta)
\bigr)
=
0
\qquad
\text{for every }\theta\in\mathbb T^d.
\]
Since the unknown belongs to
$C^m(\mathbb T^d)$, the function
$\mathcal Y_\varepsilon$ is automatically periodic in every phase
variable. 
We now derive the first-order expansion. It follows directly from the implicit function theorem with parameter, applied at the unperturbed configuration.
Differentiate the implicit equation
\[
\mathscr F(\varepsilon,\mathcal Y_\varepsilon)=0
\]
with respect to $\varepsilon$ at $\varepsilon=0$, we obtain
\[
\partial_\varepsilon
\mathscr F(0,Y_0)
+
\partial_Y\mathscr F(0,Y_0)
\left[
\left.
\partial_\varepsilon
\mathcal Y_\varepsilon
\right|_{\varepsilon=0}
\right]
=
0.
\]
By the expansion
\eqref{eq-Robin-hull-expansion-centerline},
\[
\partial_\varepsilon
\mathscr F(0,Y_0)(\theta)
=
\partial_y\mathcal R_1
\left(\theta,\tfrac12\right).
\]
Using \eqref{oso-dec},
we find
\[
\left.
\partial_\varepsilon
\mathcal Y_\varepsilon(\theta)
\right|_{\varepsilon=0}
=
\tfrac{2}{\pi}
\partial_y\mathcal R_1
\left(\theta,\tfrac12\right).
\]
Therefore
\[
\mathcal Y_\varepsilon(\theta)
=
\tfrac12
+
\tfrac{2\varepsilon}{\pi}
\partial_y\mathcal R_1
\left(\theta,\tfrac12\right)
+
\varepsilon^2\mathcal Y_{2,\varepsilon}(\theta),
\]
which proves the desired decomposition.
\end{proof}
We now reduce the full critical-point problem to the exact reduced
Robin hull using the result of Proposition \ref{prop-exact-vertical-critical-graph-centerline}. Define
\begin{equation}
\label{eq-reduced-Robin-epsilon}
\widehat{\mathcal R}_\varepsilon(\theta)
:=
\mathcal R_\varepsilon
\bigl(
\theta,\mathcal Y_\varepsilon(\theta)
\bigr)\quad \hbox{and}\quad
\mathcal S_\varepsilon(\theta)
:=
D_\omega
\widehat{\mathcal R}_\varepsilon(\theta).
\end{equation}
Since
\[
\partial_y\mathcal R_\varepsilon
\bigl(
\theta,\mathcal Y_\varepsilon(\theta)
\bigr)
=
0,
\]
the chain rule gives the exact identity
\begin{equation}
\label{eq-exact-S-epsilon-identity}
\mathcal S_\varepsilon(\theta)
=
D_\omega\mathcal R_\varepsilon
\bigl(
\theta,\mathcal Y_\varepsilon(\theta)
\bigr).
\end{equation}
By setting
\[
Y_\varepsilon(x)
:=
\mathcal Y_\varepsilon(x\omega),
\]
and recalling \eqref{robin-link} we find that
\[
\partial_yR_{\Omega_\varepsilon}
\bigl(x,Y_\varepsilon(x)\bigr)
=
0
\]
and
\[
\left(\partial_xR_{\Omega_\varepsilon}\right)
\bigl(x,Y_\varepsilon(x)\bigr)
=
\mathcal S_\varepsilon(x\omega).
\]
Hence
\begin{equation}
\label{eq-full-critical-exact-reduction}
\big(\nabla R_{\Omega_\varepsilon}\big)
\bigl(x,Y_\varepsilon(x)\bigr)
=
0
\quad\Longleftrightarrow\quad
\mathcal S_\varepsilon(x\omega)=0.
\end{equation}
Thus the critical points of the full Robin function near the centerline
are described exactly by the intersections of the Kronecker orbit
$x\mapsto x\omega$ with the hypersurface
\begin{equation*}
\label{eq-Sigma-epsilon-exact}
\Sigma_\varepsilon
:=
\left\{
\theta\in\mathbb T^d:
\mathcal S_\varepsilon(\theta)=0
\right\}.
\end{equation*}
We next identify the perturbative limit of this exact hypersurface.
At $y=1/2$, let
\begin{equation}
\label{eq-K-centerline}
K(s)
:=
P_+^2\left(s,\tfrac12\right)
=
\frac{1}{4\cosh^2(\pi s)}
\end{equation}
and define
\begin{equation}
\label{eq-filtered-profile-centerline}
F_j(\theta)
:=
\int_{\mathbb R}
K(s)
f_j(\theta+\omega_js)\,ds.
\end{equation}
Then, by \eqref{eq-R1-hull-final}
\begin{equation*}
\label{eq-R1-centerline-hull}
\mathcal R_1
\left(\theta,\tfrac12\right)
=
\sum_{j=1}^d F_j(\theta_j).
\end{equation*}
We therefore introduce
\begin{equation*}
\label{eq-S1-centerline}
\mathcal S_1(\theta)
:=
D_\omega
\mathcal R_1
\left(\theta,\tfrac12\right)
=
\sum_{j=1}^d
\omega_jF_j'(\theta_j).
\end{equation*}
To connect the exact critical-point equation with its leading-order
perturbative model, we first identify the first-order behavior of
$\mathcal S_\varepsilon$. The following lemma  deals with this asymptotics. 
\begin{lemma}
\label{lem-S-epsilon-expansion}
For small $\varepsilon$, we have the decomposition
\begin{equation*}
\label{eq-S-epsilon-expansion}
\mathcal S_\varepsilon
=
\varepsilon\mathcal S_1
+
O_{C^{m-1}}(\varepsilon^2).
\end{equation*}

\end{lemma}

\begin{proof}
Using
\eqref{eq-Robin-hull-expansion-centerline}, \eqref{eq-flat-centerline-nondegeneracy} and
\eqref{eq-Yepsilon-expansion-exact}, we have
\begin{align*}
\widehat{\mathcal R}_\varepsilon(\theta)
&=
R_0\bigl(\mathcal Y_\varepsilon(\theta)\bigr)
+
\varepsilon
\mathcal R_1
\bigl(
\theta,\mathcal Y_\varepsilon(\theta)
\bigr)
+
O_{C^m}(\varepsilon^2)
\\
&=
R_0\left(\tfrac12\right)
+
\varepsilon
\mathcal R_1\left(\theta,\tfrac12\right)
+
O_{C^m}(\varepsilon^2).
\end{align*}
Indeed, the first-order contribution of the displacement of
$\mathcal Y_\varepsilon$ to the flat part vanishes because
\[
R_0'\left(\tfrac12\right)=0.
\]
Applying $D_\omega$ gives
\[
\mathcal S_\varepsilon
=
\varepsilon
D_\omega
\mathcal R_1\left(\theta,\tfrac12\right)
+
O_{C^{m-1}}(\varepsilon^2),
\]
which gives the desired expansion of the lemma.
\end{proof}
Set
\begin{equation*}
\label{eq-Sigma1-centerline}
\Sigma_1
:=
\left\{
\theta\in\mathbb T^d:
\mathcal S_1(\theta)=0
\right\}.
\end{equation*}
We impose the following natural nondegeneracy assumption,
\begin{equation}
\label{eq-leading-transversality-centerline}
D_\omega\mathcal S_1(\theta)\neq0,
\qquad
\forall\,\theta\in\Sigma_1.
\end{equation}
By compactness of $\Sigma_1$, this condition is  uniform. Hence Lemma \ref{lem-S-epsilon-expansion} implies that, for all sufficiently
small nonzero $\varepsilon$, $\Sigma_\varepsilon$ is a smooth
hypersurface, close to $\Sigma_1$, and
\begin{equation*}
\label{eq-exact-transversality-centerline}
D_\omega\mathcal S_\varepsilon(\theta)\neq0,
\qquad
\theta\in\Sigma_\varepsilon.
\end{equation*}
Denote
\begin{equation*}
\label{eq-Nepsilon-definition}
N_\varepsilon(L)
:=
\#
\Big\{
x\in(0,L):
\mathcal S_{\eps}(x\omega)=0
\Big\},
\end{equation*}
and  the asymptotic density distribution
\begin{equation*}
\label{eq-rhoepsilon-exact-definition}
\rho_\varepsilon
:=
\lim_{L\to\infty}
\frac{N_\varepsilon(L)}{L}
\end{equation*}
when it exists.
We can therefore apply directly
Theorem~\ref{thm-nonperturbative-Robin-critical-density}.

\begin{proposition}
\label{thm-exact-density-centerline-small-perturbation}
Assume the nonresonance condition
\eqref{eq-nonresonance-Kronecker-critical} and  that
$\Sigma_1\neq\varnothing$ and that
\eqref{eq-leading-transversality-centerline} holds.
Then there exists $\varepsilon_0>0$ such that, for every
$0<|\varepsilon|<\varepsilon_0$, the critical points of the full
Robin function $R_{\Omega_\varepsilon}$ contained in a fixed
neighborhood of the centerline lie on the quasi-periodic graph
\[
y=Y_\varepsilon(x)
=
\mathcal Y_\varepsilon(x\omega),
\]
and are characterized exactly by
\[
\mathcal S_\varepsilon(x\omega)=0.
\]
Moreover, $\rho_\varepsilon$ exists and is given by the exact formula
\begin{align}
\label{eq-rhoepsilon-exact-surface}
\rho_\varepsilon
&=\frac{1}{(2\pi)^d}
\int_{\mathbb T^d}
\delta\!\left(\mathcal S_\varepsilon(\theta)\right)
\left|
D_\omega\mathcal S_\varepsilon(\theta)
\right|
\,d\theta\\
\nonumber&=
\frac{1}{(2\pi)^d}
\int_{\Sigma_\varepsilon}
\frac{
|D_\omega\mathcal S_\varepsilon(\theta)|
}{
|\nabla_\theta\mathcal S_\varepsilon(\theta)|
}
\,d\sigma_\varepsilon(\theta).
\end{align}
and decomposes as
\begin{equation}
\label{eq-rhoepsilon-leading}
\rho_\varepsilon
=
\frac{1}{(2\pi)^d}
\int_{\Sigma_1}
\frac{
|D_\omega\mathcal S_1(\theta)|
}{
|\nabla_\theta\mathcal S_1(\theta)|
}
\,d\sigma(\theta)+O(\varepsilon).
\end{equation}
Let
\[
0<x_1^\varepsilon<x_2^\varepsilon<\cdots
\]
denote the positive abscissas of these critical points, and define
their phases by
\[
\Theta_n^\varepsilon
:=
x_n^\varepsilon\omega
\pmod{2\pi}
\in\Sigma_\varepsilon.
\]
Then the empirical measures
\begin{equation}
\label{eq-empirical-measures-epsilon}
\mu_N^\varepsilon
:=
\frac1N
\sum_{n=1}^N
\delta_{\Theta_n^\varepsilon}
\end{equation}
converge weakly, as $N\to\infty$, to the probability measure
$\mu_\varepsilon$ on $\Sigma_\varepsilon$ given by
\begin{equation}
\label{eq-limiting-measure-epsilon}
d\mu_\varepsilon(\theta)
=
\frac{1}{(2\pi)^d\rho_\varepsilon}
\frac{
|D_\omega\mathcal S_\varepsilon(\theta)|
}{
|\nabla_\theta\mathcal S_\varepsilon(\theta)|
}
\,d\sigma_\varepsilon(\theta).
\end{equation}
\end{proposition}
\begin{proof}
Proposition~\ref{prop-exact-vertical-critical-graph-centerline}
provides the exact vertical critical graph required in
Theorem~\ref{thm-nonperturbative-Robin-critical-density}.
Furthermore, Lemma~\ref{lem-S-epsilon-expansion} gives
\[
\tfrac{\mathcal S_\varepsilon}{\varepsilon}
\longrightarrow
\mathcal S_1
\qquad\text{in }C^{m-1}(\mathbb T^d),\, m\geq2.
\]
The regularity and the transversality assumption 
\eqref{eq-leading-transversality-centerline}
persist for all sufficiently small
$\varepsilon\neq0$. In particular,
$\Sigma_\varepsilon$ is a smooth hypersurface and the Kronecker
direction $\omega$ is everywhere transverse to it.
All assumptions of
Theorem~\ref{thm-nonperturbative-Robin-critical-density}
are therefore satisfied, and
\eqref{eq-rhoepsilon-exact-surface} follows directly.
\\
The weighted distribution statement in
Theorem~\ref{thm-nonperturbative-Robin-critical-density}
also yields, for every
$\varphi\in C(\Sigma_\varepsilon)$,
\[
\lim_{N\to\infty}
\frac1N
\sum_{n=1}^N
\varphi(\Theta_n^\varepsilon)
=
\frac{1}{(2\pi)^d\rho_\varepsilon}
\int_{\Sigma_\varepsilon}
\varphi(\theta)
\frac{
|D_\omega\mathcal S_\varepsilon(\theta)|
}{
|\nabla_\theta\mathcal S_\varepsilon(\theta)|
}
\,d\sigma_\varepsilon(\theta).
\]
This is precisely the weak convergence
\[
\mu_N^\varepsilon\rightharpoonup\mu_\varepsilon
\]
with $\mu_\varepsilon$ given by
\eqref{eq-limiting-measure-epsilon}.
Finally, since
\[
\widetilde{\mathcal S}_\varepsilon
=
\tfrac{\mathcal S_\varepsilon}{\varepsilon}
=
\mathcal S_1+O_{C^{m-1}}(\varepsilon),
\]
the implicit function theorem for regular level sets gives a
$C^1$ family of diffeomorphisms
\[
\Psi_\varepsilon:
\Sigma_1\longrightarrow\Sigma_\varepsilon,
\qquad
\Psi_\varepsilon
=
\operatorname{Id}+O_{C^1}(\varepsilon).
\]
The density is unchanged if $\mathcal S_\varepsilon$ is multiplied
by the nonzero scalar $1/\varepsilon$. Pulling the surface integral
back to $\Sigma_1$ therefore gives
\[
\rho_\varepsilon
=
\rho_0+O(\varepsilon),
\]
which proves \eqref{eq-rhoepsilon-leading}.
\end{proof}


\begin{remark}
\label{rem-Morse-type-centerline-exact}
The preceding argument concerns the existence and spatial density of
the critical points and makes the pointwise persistence argument
unnecessary for these purposes. The local Hessian expansion remains
useful, however, for determining the Morse type of the equilibria.
If $x_\ast$ is a simple zero of
$q(x):=\mathcal S_1(x\omega)$, then the corresponding exact critical
point $(x_\varepsilon,y_\varepsilon)$ satisfies
\[
x_\varepsilon=x_\ast+O(\varepsilon),
\qquad
y_\varepsilon
=
\tfrac12
+
\tfrac{2\varepsilon}{\pi}
\partial_yR_1
\left(x_\ast,\tfrac12\right)
+
O(\varepsilon^2),
\]
and
\[
\det D^2R_{\Omega_\varepsilon}
(x_\varepsilon,y_\varepsilon)
=-\tfrac{\pi \varepsilon}{2}
 q'(x_\ast)
+
O(\varepsilon^2).
\]
Hence, for $\varepsilon>0$, the critical point is hyperbolic when
$q'(x_\ast)>0$ and  elliptic when $q'(x_\ast)<0$, with the
classification reversed for $\varepsilon<0$.
\end{remark}{
\paragraph{Full-Robin critical points in a two-frequency channel.} We illustrate the perturbative relation between the limiting critical
set and the equilibria of the full Robin function using the
two-frequency upper-wall deformation
\[
\Omega_\varepsilon
=
\left\{
(x,y)\in\mathbb R^2:
0<y<
1+\varepsilon
\left(
\cos x+a\cos(\sqrt{2}x)
\right)
\right\}.
\]
The parameter $a$ is chosen so that the effective amplitude introduced
in Section~\ref{sec-two-frequency-cosine-case} is $b=0.5$, namely
\[
a
=
\frac{b\sinh(1/\sqrt{2})}
     {2\sinh(1/2)}
\simeq0.3682259.
\]
This value lies in the globally transverse regime
$0<b<1/\sqrt{2}$.
\\
The leading longitudinal critical points are the zeros $x_*$ of
\[
q_b(x)=\sin x+b\sin(\sqrt{2}x).
\]
For each such zero, we compute the corresponding critical point
$(x_\varepsilon,y_\varepsilon)$ directly from the full finite-$\varepsilon$
Robin function by solving
\[
\partial_xR_{\Omega_\varepsilon}(x_\varepsilon,y_\varepsilon)=0,
\qquad
\partial_yR_{\Omega_\varepsilon}(x_\varepsilon,y_\varepsilon)=0.
\]
The numerical computation is independent of the first-order critical
point approximation.
\\
Figure~\ref{fig:full_robin_qp_critical_convergence} compares the full-Robin critical points with the asymptotic
predictions of Remark~7.11 as $\varepsilon$ is decreased. The longitudinal position is compared with the leading-order zero $x_\ast$, whereas the transverse position is compared with its first-order approximation. We obtain
\[
\max_n|x_n^\varepsilon-x_n^*|
\propto\varepsilon^{1.022},
\]
while
\[
\max_n\left|
y_n^\varepsilon-
\left(
\frac12+
\frac{2\varepsilon}{\pi}
\partial_yR_1(x_n^*,1/2)
\right)
\right|
\propto\varepsilon^{1.975}.
\]
These rates are consistent with the $O(\varepsilon)$ and
$O(\varepsilon^2)$ remainders in Remark~\ref{rem-Morse-type-centerline-exact}. The Morse classifications
of all computed equilibria also agree with the sign criterion based on
$q'(x_*)$.

}

\begin{figure}[ht]
    \centering
    \includegraphics[width=0.78\textwidth]{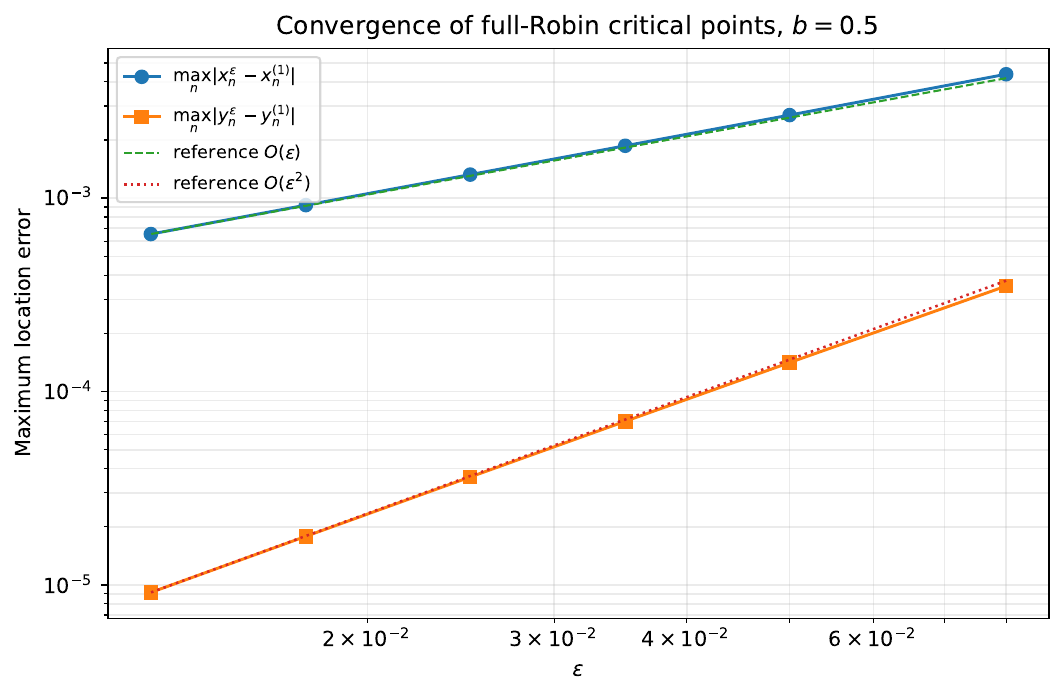}
    \caption{
Convergence of full-Robin critical points to the perturbative
predictions for the two-frequency upper-wall deformation with
$b=0.5$. The longitudinal error is measured relative to the zeros
$x_n^*$ of the leading critical equation, while the transverse error
is measured relative to the first-order vertical displacement in
Remark~\ref{rem-Morse-type-centerline-exact}. The reference slopes $O(\varepsilon)$ and
$O(\varepsilon^2)$ are shown for comparison.
}
\label{fig:full_robin_qp_critical_convergence}
\end{figure}

\subsection{Two-frequency cosine model: ergodic density and phase transition}
\label{sec-two-frequency-cosine-case}

We now specialize the preceding general framework to a simple
two-frequency quasi-periodic profile for which the critical set, its
interaction with the Kronecker flow, and the associated zero density
can be analyzed explicitly. Besides providing a concrete illustration
of the abstract results developed above, this model exhibits the main
geometric mechanisms that may affect the statistics of the equilibria:
transverse intersections, finite-order tangencies, and
parameter-dependent changes in the geometry of the critical set.
We consider
\[
d=2,
\qquad
\omega=(1,\sqrt2),
\qquad
f_1(\theta)=\cos\theta,
\qquad
f_2(\theta)=a\cos\theta.
\]
Using \eqref{eq-K-centerline} and
\eqref{eq-filtered-profile-centerline}, we introduce the Fourier
multiplier
\[
\kappa(\xi)
:=
\int_{\mathbb R}K(s)e^{i\xi s}\,ds
=
\frac{\xi}{4\pi\sinh(\xi/2)}.
\]
The corresponding filtered profiles are
\[
F_1(\theta)=\kappa(1)\cos\theta,
\qquad
F_2(\theta)=a\kappa(\sqrt2)\cos\theta,
\]
so that the first-order contribution to the Robin function on the
centerline is
\[
R_1(x,1/2)
=
\kappa(1)\cos x
+
a\kappa(\sqrt2)\cos(\sqrt2 x).
\]
Consequently, the leading-order critical equation is
\begin{equation}
\label{eq-x-critical-two-cosine}
\kappa(1)\sin x
+
\sqrt2\,a\kappa(\sqrt2)\sin(\sqrt2 x)
=
0.
\end{equation}
It is convenient to absorb the effect of the filtering operator into a
single effective amplitude. We therefore set
\[
b=b(a)
:=
\frac{\sqrt2\,a\kappa(\sqrt2)}{\kappa(1)}
=
\frac{2a\sinh(1/2)}{\sinh(1/\sqrt2)}.
\]
Since $\kappa(1)\neq0$, equation
\eqref{eq-x-critical-two-cosine} is equivalent to
\[
q_b(x)
:=
\sin x+b\sin(\sqrt2 x)
=
0.
\]
The quasi-periodic function $q_b$ is obtained by restricting to the
Kronecker orbit the function
\[
\mathcal S_b(\theta_1,\theta_2)
:=
\sin\theta_1+b\sin\theta_2,
\qquad
(\theta_1,\theta_2)\in\T^2,
\]
namely
\[
q_b(x)=\mathcal S_b(x,\sqrt2 x).
\]
We denote its zero set on the torus by
\[
\Sigma_b
:=
\left\{
(\theta_1,\theta_2)\in\T^2:
\mathcal S_b(\theta_1,\theta_2)=0
\right\}.
\]
Thus the zeros of $q_b$ correspond precisely to the intersections of
the Kronecker trajectory
\[
x\longmapsto (x,\sqrt2 x)\pmod{2\pi}
\]
with the critical set $\Sigma_b$. This formulation separates the two
ingredients entering the equilibrium statistics. The effective
amplitude $b$ determines the geometry of the curve $\Sigma_b$, whereas
the frequency vector $\omega=(1,\sqrt2)$ determines the direction
along which this curve is sampled. The distribution of the vortex
equilibria is therefore reduced to a geometric intersection problem
between a fixed linear flow and a parameter-dependent curve on
$\T^2$.
{For reference, Figure~\ref{fig:qp_full_robin_orbits} shows the level-set
geometry of the full finite-$\varepsilon$ Robin function for the same
two-frequency channel, with $b=0.8$ and $\varepsilon=0.05$. The colored curves are level sets of $R_\varepsilon(x,y)$ and hence point-vortex orbits. Closed families surround elliptic equilibria,
whereas hyperbolic equilibria organize the separatrix-like level sets.
The computation is performed directly from the full Robin function of
the perturbed domain and therefore complements the first-order critical
equation used below.
}
\begin{figure}[ht]
    \centering
   \includegraphics[width=0.8\textwidth]{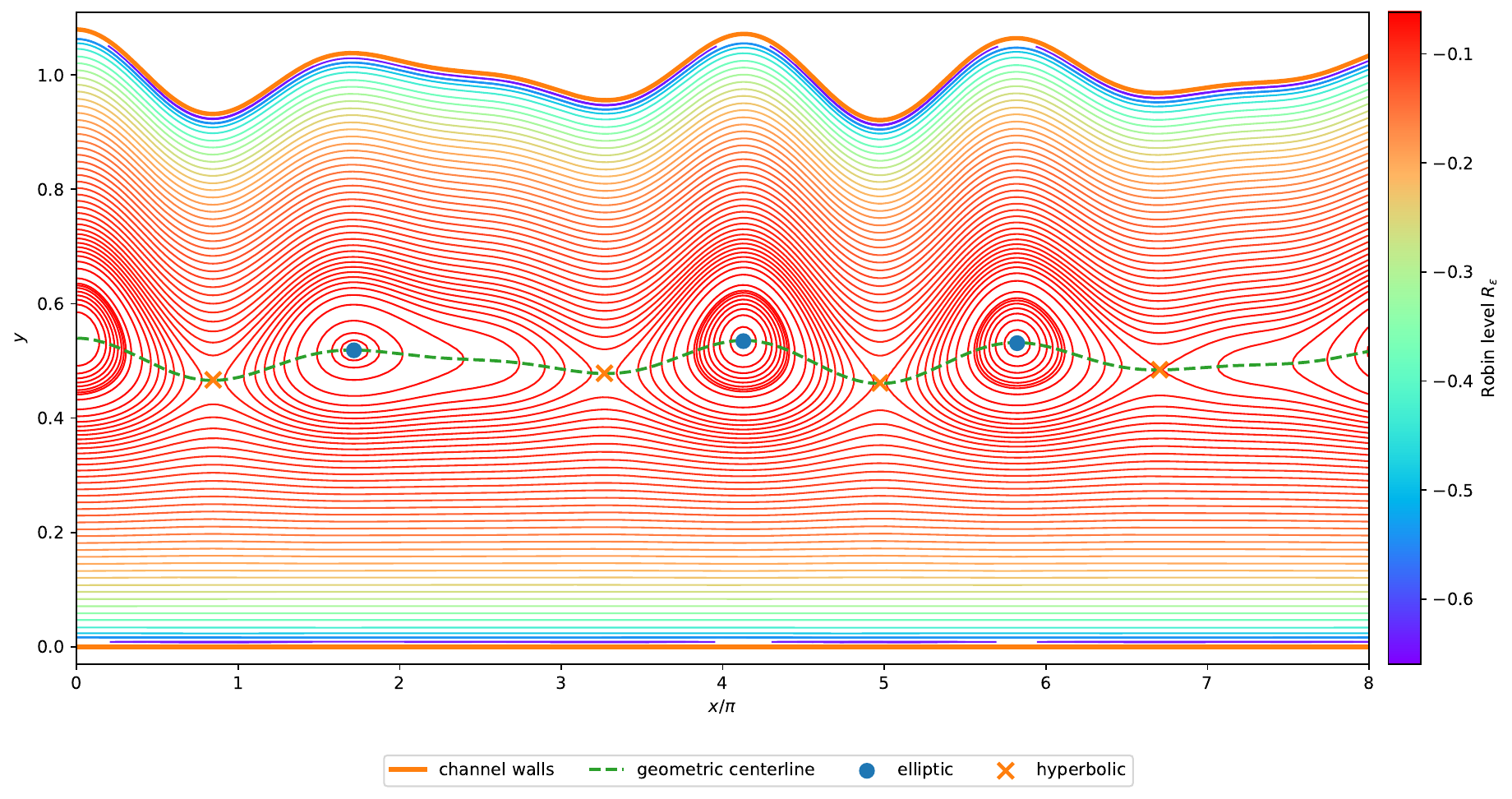}
\caption{
Level sets of the full finite-$\varepsilon$ Robin function
$R_\varepsilon(x,y)$ for the two-frequency quasi-periodic channel $0<y<1+0.05\left[\cos x+a\cos(\sqrt2 x)\right]$. The discrete colored curves are Robin level sets and therefore
point-vortex orbits; their colors are assigned from the continuous
Robin-value scale shown on the colorbar.
Elliptic and hyperbolic critical points of $R_\varepsilon$ are marked
separately. Closed orbit families surround elliptic equilibria, while
hyperbolic equilibria are associated with separatrix-like structures.}
\label{fig:qp_full_robin_orbits}
\end{figure}
\\For $0<|b|<1$, the geometry of this intersection problem can be made
particularly explicit. For each $\theta_2\in\mathbb T$, the equation
\[
\sin\theta_1=-b\sin\theta_2
\]
has two solutions. Writing
\[
A_b(\theta_2)
=
\sqrt{1-b^2\sin^2\theta_2},
\]
we decompose the critical curve as
\[
\Sigma_b=\Sigma_b^+\cup\Sigma_b^-,
\qquad
\Sigma_b^\pm
=
\left\{
(\theta_1,\theta_2)\in\Sigma_b:
\cos\theta_1=\pm A_b(\theta_2)
\right\}.
\]
Each zero of $q_b$ therefore corresponds to a crossing of the
Kronecker trajectory with one of these two branches.
Figure~\ref{fig:sigma_branches_zero_count} illustrates this
correspondence for $b=0.8$. The shifted representation
$-\pi/2\leq\theta_1\leq3\pi/2$ is used only to display both components
$\Sigma_b^\pm$ as continuous curves. The intersections of the
Kronecker trajectory with these branches correspond exactly to the
zeros of $q_b$ shown in the physical variable $x$.
\begin{figure}[!htbp]
    \centering
    \includegraphics[width=0.8\textwidth]{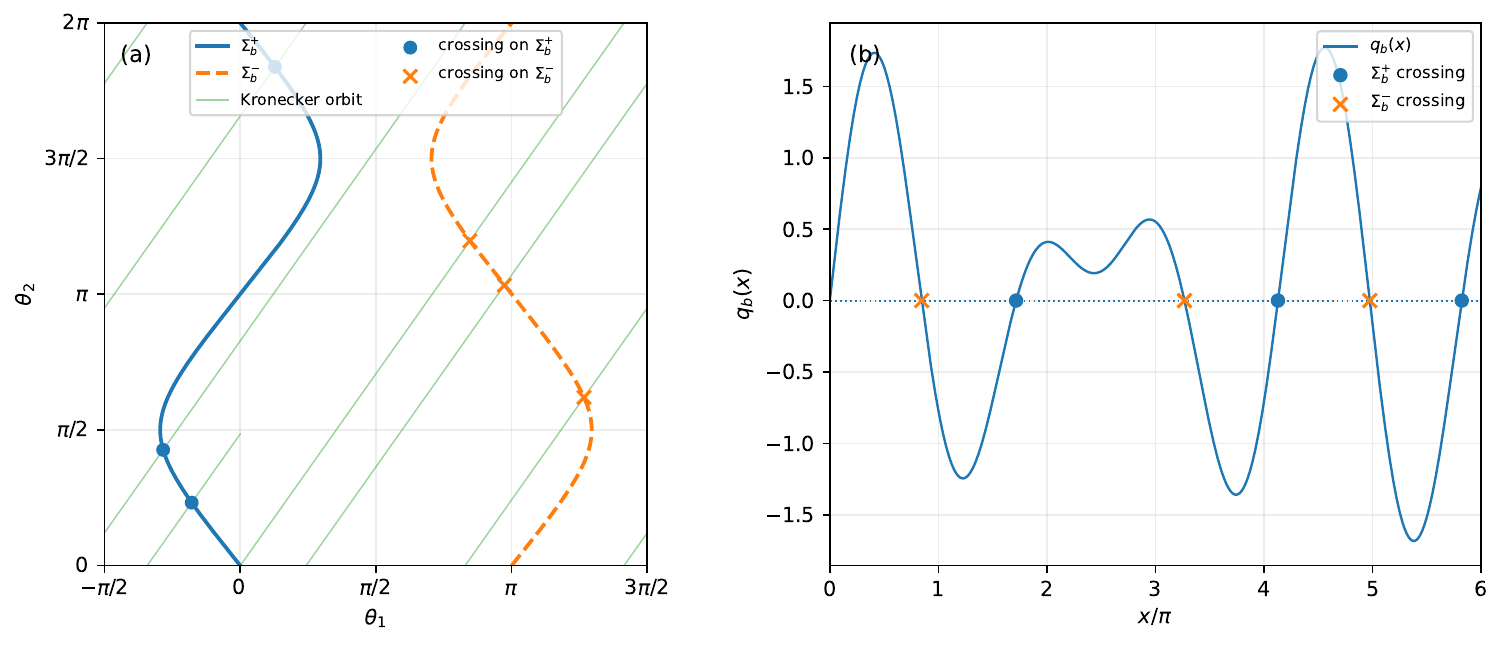}
    \caption{
    Geometric interpretation of the zero counting for the
    two-frequency model with $b=0.8$.
    Left: the two components $\Sigma_b^+$ and $\Sigma_b^-$ of the
    critical set
    $\Sigma_b=\{(\theta_1,\theta_2):
    \sin\theta_1+b\sin\theta_2=0\}$,
    together with the Kronecker trajectory
    $x\mapsto(x,\sqrt{2}x)\pmod{2\pi}$.
    The $\theta_1$ window is shifted to
    $[-\pi/2,3\pi/2]$ so that both branches are displayed continuously.
    Markers indicate intersections with $\Sigma_b^+$ and
    $\Sigma_b^-$.
    Right: the corresponding zeros of
    $q_b(x)=\sin x+b\sin(\sqrt{2}x)$ over $0<x<6\pi$,
    with the same branch classification.
    }
\label{fig:sigma_branches_zero_count}
\end{figure}
The directional derivative along this flow is
\[
D_\omega
=
\partial_{\theta_1}
+
\sqrt2\,\partial_{\theta_2},
\]
and hence
\[
D_\omega\mathcal S_b
=
\cos\theta_1+\sqrt2\,b\cos\theta_2.
\]
Accordingly, the condition
\[
D_\omega\mathcal S_b\neq0
\qquad\text{on }\Sigma_b
\]
expresses the transversality of the Kronecker orbit with the critical
set, while the simultaneous conditions
\[
\mathcal S_b=0,
\qquad
D_\omega\mathcal S_b=0
\]
characterize tangency points. Their degeneracy order is understood in
the sense of \eqref{eq-finite-order-flow-tangency-main}.
{
\paragraph{Temporal structure of the quasi-periodic motion.}
The same two-frequency geometry also provides a concrete illustration of the quasi-periodic transporting dynamics described in Proposition~\ref{prop-first-order-orbit-small-channel} and Theorem~\ref{thm-exact-qp-orbits-1}. We consider the two-frequency
upper-wall perturbation
\[
\Omega_\varepsilon
=
\left\{
(x,y)\in\mathbb R^2:
0<y<
1+\varepsilon
\left(
\cos x+a\cos(\sqrt{2}x)
\right)
\right\},
\]
with
\[
\varepsilon=0.05,
\qquad
a=0.3682259,
\]
and follow the transporting orbit issued from the flat reference
level
\[
y_\star=0.25.
\]
The value of $a$ corresponds to the effective two-frequency amplitude
$b=0.5$.\\
The trajectory is reconstructed from the first-order invariant-graph
and velocity expansions of Proposition~4.3. The corresponding
effective longitudinal drift is
\[
c_\varepsilon\simeq0.249686.
\]
Figure~\ref{fig:temporal_qp_motion} shows the transverse displacement
$y(t)$ together with the bounded longitudinal modulation
\[
x(t)-c_\varepsilon t.
\]
Both quantities exhibit persistent oscillations without a finite
temporal period, consistently with the quasi-periodic representation
\[
z(t)
=
c_\varepsilon t\,e_1
+
Z_\varepsilon(c_\varepsilon t\,\omega),
\qquad
\omega=(1,\sqrt{2}).
\]
The corresponding temporal spectra are discrete and are organized by
integer combinations of the two basic frequencies. In particular,
the dominant peaks occur near
\[
\Omega_\ell
=
\left|
c_\varepsilon
\left(
\ell_1+\sqrt{2}\ell_2
\right)
\right|,
\qquad
(\ell_1,\ell_2)\in\mathbb Z^2.
\]
For the present trajectory, the principal measured peaks include
\[
\Omega\simeq
0.249685,\quad
0.353055,\quad
0.103370,\quad
0.499371,\quad
0.602741,
\]
in close correspondence with
\[
c_\varepsilon,\qquad
\sqrt{2}c_\varepsilon,\qquad
(\sqrt{2}-1)c_\varepsilon,\qquad
2c_\varepsilon,\qquad
(1+\sqrt{2})c_\varepsilon,
\]
respectively. The computation therefore illustrates how the spatial
quasi-periodicity of the channel is transferred to the temporal
frequency content of a transporting vortex trajectory.
}

\begin{figure}[h]
    \centering
    \includegraphics[width=0.7\textwidth]{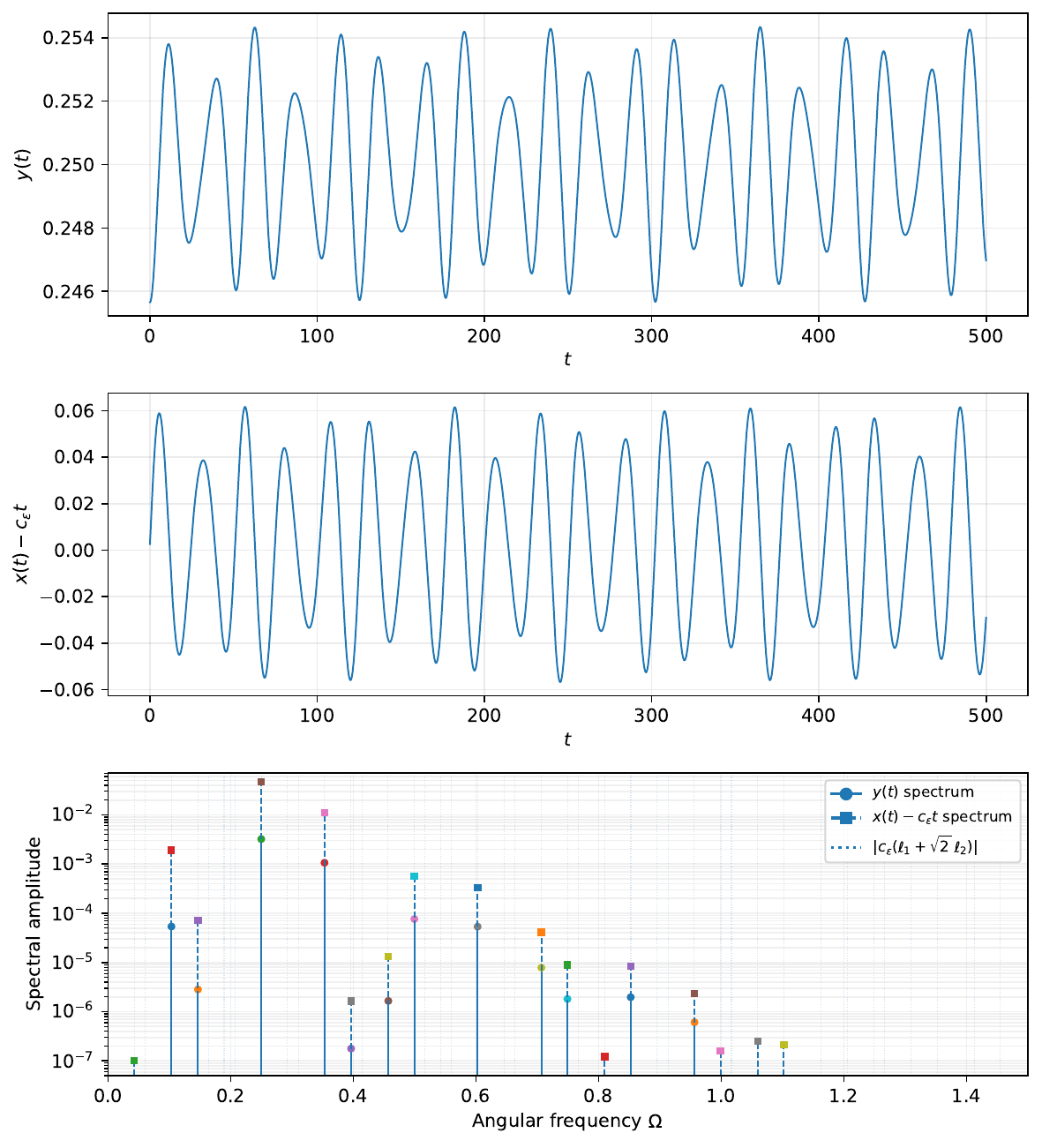}
    \caption{Temporal structure of a transporting trajectory for the two-frequency upper-wall perturbation $0<y<1+\varepsilon[\cos x+a\cos(\sqrt{2}x)]$ with $\varepsilon=0.05$, $a=0.3682259$, and reference level $y_\star=0.25$. Top: transverse position $y(t)$. Middle: bounded longitudinal modulation $x(t)-c_\varepsilon t$ after removal of the effective drift $c_\varepsilon\simeq0.249686$. Bottom: discrete temporal spectra of both signals; vertical reference lines indicate frequencies of the form $|c_\varepsilon(\ell_1+\sqrt{2}\ell_2)|$, $(\ell_1,\ell_2)\in\mathbb Z^2$. Only resolved spectral peaks with normalized amplitude
$A_k>10^{-7}$ are displayed. The trajectory starts from
$(x_0,y_0)=(0,0.2456603)$, corresponding to the reference level
$y_\star=0.25$ on the first-order invariant graph.
}
\label{fig:temporal_qp_motion}
\end{figure}

\subsubsection{Geometry of the critical set and tangencies}

We first describe the geometry of $\Sigma_b$ and classify the possible
tangencies. This geometric analysis identifies the parameter values at
which global transversality is lost and will subsequently explain the
changes in the asymptotic density and in the distribution of the
critical phases.

\begin{proposition}
\label{prop-two-frequency-geometry}
Assume that $|b|\notin\{0,1\}$. Then $\Sigma_b$ is a smooth compact
curve in $\T^2$. Moreover:
\begin{enumerate}
\item if
\(
0<|b|<\frac1{\sqrt2}
\qquad\text{or}\qquad
|b|>1,
\)
then $\Sigma_b$ is everywhere transverse to the Kronecker flow;

\item if
\(
\frac1{\sqrt2}<|b|<1,
\)
the tangencies are of order two;

\item if
\(
|b|=\frac1{\sqrt2},
\)
the tangencies are of order three.
\end{enumerate}
\end{proposition}

\begin{proof}
Differentiating yields
\[
\nabla\mathcal S_b
=
(\cos\theta_1,b\cos\theta_2).
\]
Therefore, for $b\neq0,$ a singular point of $\Sigma_b$ would satisfy
\[
\cos\theta_1=\cos\theta_2=0.
\]
Since $\mathcal S_b=0$, this would imply $|b|=1$. Thus
\[
\nabla\mathcal S_b\neq0
\qquad\text{on }\Sigma_b
\]
for $|b|\notin\{0,1\}$, and hence $\Sigma_b$ is a smooth compact curve.
A tangency point is characterized by
\[
\sin\theta_1+b\sin\theta_2=0,
\qquad
\cos\theta_1+\sqrt2\,b\cos\theta_2=0.
\]
Squaring these identities and adding them gives
\[
1
=
b^2\sin^2\theta_2
+
2b^2\cos^2\theta_2,
\]
which is equivalent to
\begin{equation}
\label{tangency-one}
\sin^2\theta_2
=
2-\frac1{b^2}.
\end{equation}
This equation has no solutions if
\[
|b|<\frac1{\sqrt2}
\qquad\text{or}\qquad
|b|>1.
\]
Thus global transversality holds in these two regimes.
Furthermore,
\[
D_\omega^2\mathcal S_b
=
-\sin\theta_1-2b\sin\theta_2.
\]
At a tangency point, using
\[
\sin\theta_1=-b\sin\theta_2,
\]
we obtain
\[
D_\omega^2\mathcal S_b
=
-b\sin\theta_2.
\]
If
\[
\frac1{\sqrt2}<|b|<1,
\]
relation \eqref{tangency-one} implies
$\sin\theta_2\neq0$, so the tangencies are quadratic.
\\In the case
\[
|b|=\frac1{\sqrt2},
\]
we get
\[
\sin\theta_1=\sin\theta_2=0,
\]
and consequently
\[
D_\omega\mathcal S_b
=
D_\omega^2\mathcal S_b
=
0.
\]
On the other hand,
\[
D_\omega^3\mathcal S_b
=
-\cos\theta_1-2\sqrt2\,b\cos\theta_2.
\]
Using the tangency relation
\[
\cos\theta_1=-\sqrt2\,b\cos\theta_2,
\]
we obtain
\[
D_\omega^3\mathcal S_b
=
-\sqrt2\,b\cos\theta_2
\neq0.
\]
Hence, the tangencies are cubic at the threshold. This completes the proof.
\end{proof}
Proposition~\ref{prop-two-frequency-geometry} identifies the geometric
mechanism behind the transition in the equilibrium statistics. For
\[
0<|b|<\frac1{\sqrt2},
\]
every intersection is transverse. At the threshold
$|b|=1/\sqrt2$, the first tangencies appear and are cubic. Once
$1/\sqrt2<|b|<1$, quadratic tangencies are present. Thus
$|b|=1/\sqrt2$ marks the first loss of global transversality. We now
show that the same threshold appears naturally in the asymptotic
density of the vortex equilibria.

\subsubsection{Asymptotic density of equilibria}

The preceding geometric classification can now be converted into a
counting law for the physical equilibria. Since all the tangencies
described in Proposition~\ref{prop-two-frequency-geometry} are of
finite order, the zero-counting theory developed above remains
applicable. In particular, the geometric flux formula continues to
determine the leading density even in the intermediate regime in which
global transversality fails.

\begin{proposition}
\label{prop-two-frequency-density}
Let
\(
b\in\mathbb R\setminus\{0,\pm1\}.
\)
Then the leading term $\rho_0(b)$  of the asymptotic density  in \eqref{eq-rhoepsilon-leading} is given as follows.
If $0<|b|<1$, then
\begin{equation}
\label{eq-rho-b-integral}
{
\rho_0(b)
=
\frac1{2\pi^2}
\int_0^{2\pi}
\max\left\{
1,
\frac{\sqrt2\,|b\cos\theta|}
{\sqrt{1-b^2\sin^2\theta}}
\right\}
\,d\theta .
}
\end{equation}
In particular,
\begin{equation*}
\label{eq-rho-small-b}
{
\rho_0(b)
=
\frac1\pi,
\qquad
0<|b|\leq\frac1{\sqrt2},
}
\end{equation*}
whereas, for
\(
\frac1{\sqrt2}<|b|<1,
\)
one has
\begin{equation}
\label{eq-rho-intermediate-final}
{
\rho_0(b)
=
\frac1\pi
+
\frac{2}{\pi^2}
\left[
\sqrt2\,
\arcsin\sqrt{2b^2-1}
-
\arcsin\sqrt{2-\tfrac1{b^2}}
\right].
}
\end{equation}
If $|b|>1$, then
\begin{equation}
\label{eq-rho-large-b}
{
\rho_0(b)
=
\frac{\sqrt2}{\pi}.
}
\end{equation}
Moreover, for every
\(
b\in\mathbb R\setminus\{0,\pm1\},
\)
the limiting empirical measure \eqref{eq-limiting-measure-epsilon} satisfies
\[
\mu_\varepsilon
=
\mu_0+O_{\mathrm{weak}}(\varepsilon).
\]
with
\begin{equation}
\label{eq-two-cosine-leading-empirical-measure}
{
d\mu_0(\theta)
=
\frac1{4\pi^2\rho_0(b)}
\frac{
|\cos\theta_1+\sqrt2\,b\cos\theta_2|
}{
\sqrt{\cos^2\theta_1+b^2\cos^2\theta_2}
}
\,d\sigma_b(\theta).
}
\end{equation}
Furthermore,
\[
\rho_0\in C_{\mathrm{loc}}^{1,\frac12}((-1,1)),
\qquad
\rho_0\in C^{0,\frac12}([-1,1]).
\]

\end{proposition}


\begin{proof}
Proposition~\ref{prop-two-frequency-geometry} shows that the critical
set is either transverse to the Kronecker flow or has only
finite-order tangencies. Applying the flux formula
\eqref{eq-rhoepsilon-exact-surface} to the limiting reduced function
$\mathcal S_b$ gives
\[
\rho_0(b)
=
\frac1{4\pi^2}
\int_{\T^2}
\delta_0\bigl(
\sin\theta_1+b\sin\theta_2
\bigr)
\left|
\cos\theta_1+\sqrt2\,b\cos\theta_2
\right|
\,d\theta_1d\theta_2.
\]
To evaluate this integral, fix $\theta_2$. Since $|b|<1$, the equation
\[
\sin\theta_1=-b\sin\theta_2
\]
has two solutions modulo $2\pi$, at which
\[
|\cos\theta_1|
=
A(\theta_2)
:=
\sqrt{1-b^2\sin^2\theta_2}.
\]
Using the formula
\begin{align}\label{dirac-mass}
|A+B|+|A-B|
=
2\max\{|A|,|B|\},
\end{align}
with
\[
B=\sqrt2\,b\cos\theta_2,
\]
gives
\[
\int_0^{2\pi}
\delta_0(\mathcal S_b)
|D_\omega\mathcal S_b|
\,d\theta_1
=
2\max\left\{
1,
\frac{\sqrt2\,|b\cos\theta_2|}
{\sqrt{1-b^2\sin^2\theta_2}}
\right\},
\]
which proves \eqref{eq-rho-b-integral}.
If
\(
|b|\leq\frac1{\sqrt2},
\)
then
\[
2b^2\cos^2\theta
\leq
1-b^2\sin^2\theta,
\]
so the maximum in \eqref{eq-rho-b-integral} is identically equal to
$1$. Therefore
\[
\rho_0(b)=\frac1\pi.
\]
Consider now the case
\(
\frac1{\sqrt2}<|b|<1
\)
and set
\[
\theta_b
=
\arcsin\sqrt{2-\tfrac1{b^2}}.
\]
By symmetry,
\[
\rho_0(b)
=
\frac{2}{\pi^2}
\left[
\sqrt2\,|b|
\int_0^{\theta_b}
\frac{\cos\theta}
{\sqrt{1-b^2\sin^2\theta}}
\,d\theta
+
\frac{\pi}{2}-\theta_b
\right].
\]
Since
\[
|b|
\int_0^{\theta_b}
\frac{\cos\theta}
{\sqrt{1-b^2\sin^2\theta}}
\,d\theta
=
\arcsin\sqrt{2b^2-1},
\]
we obtain \eqref{eq-rho-intermediate-final}.
Finally, in the globally transverse regime
$0<|b|<1/\sqrt2$, Proposition~\ref{prop-two-frequency-geometry}
and Proposition~\ref{thm-exact-density-centerline-small-perturbation}
give
\[
\rho_\varepsilon
=
\rho_0(b)+O(\varepsilon)
=
\frac1\pi+O(\varepsilon)
\]
and
\[
\mu_\varepsilon
=
\mu_b+O_{\mathrm{weak}}(\varepsilon),
\]
where
\[
d\mu_b
=
\frac{1}{(2\pi)^2\rho_0(b)}
\frac{|D_\omega\mathcal S_b|}
{|\nabla\mathcal S_b|}
\,d\sigma_b.
\]
Since
\[
|\nabla\mathcal S_b|
=
\sqrt{\cos^2\theta_1+b^2\cos^2\theta_2},
\]
this yields \eqref{eq-two-cosine-leading-empirical-measure}.
The same limiting flux expression extends across the finite-order
tangency regime by the tangency theorem developed above.
The stated regularity properties follow from the explicit formula for
$\rho_0$.
\end{proof}
It is worth emphasizing that the family $\mathcal S_b$, and hence the
underlying geometric perturbation, depends analytically on $b$, whereas
the asymptotic density $\rho_0(b)$ has only finite H\"older regularity
at the transition values. Thus analytic dependence of the geometry on
$b$ is not automatically inherited by the equilibrium statistics.
\\
The loss of regularity occurs precisely when the geometry of the zero
set becomes degenerate with respect to the Kronecker direction. Away
from such values, the relevant intersections are transverse, and the
implicit function theorem gives smooth dependence of the crossing
points on $b$. At a critical value, however, transversality is lost:
simple crossings may merge into a higher-order tangency. The
corresponding zeros may then depend on $b$ through fractional powers of
the distance to the critical parameter, and this singular dependence
is inherited by the integral formula defining $\rho_0(b)$. The loss
of regularity of the density therefore reflects a geometric
bifurcation in the intersection problem rather than a lack of
regularity of the underlying family of channels.
{
\paragraph{Computational assessment of the equilibrium density.}
We complement Proposition~\ref{prop-two-frequency-density} by a direct numerical count of the
zeros of
\[
q_b(x)=\sin x+b\sin(\sqrt{2}x)
\]
over the interval $0<x<L$, with $L=5000\pi$. For each value of $b$,
we compute
\[
\rho_L(b)
=
\frac{1}{L}
\#\left\{
x\in(0,L):q_b(x)=0
\right\},
\]
and compare it with the analytical density $\rho_0(b)$ given by
\eqref{eq-rho-b-integral}--\eqref{eq-rho-large-b}.
Figure~\ref{fig:rho_b_numerical} shows excellent agreement over the
three regimes separated by $b=1/\sqrt{2}$ and $b=1$. In particular,
the computation reproduces the constant density $1/\pi$ in the first
transverse regime, the nontrivial increase in the intermediate
tangency regime, and the limiting value $\sqrt{2}/\pi$ for $b>1$.

\begin{figure}[ht]
    \centering
    \includegraphics[width=0.78\textwidth]{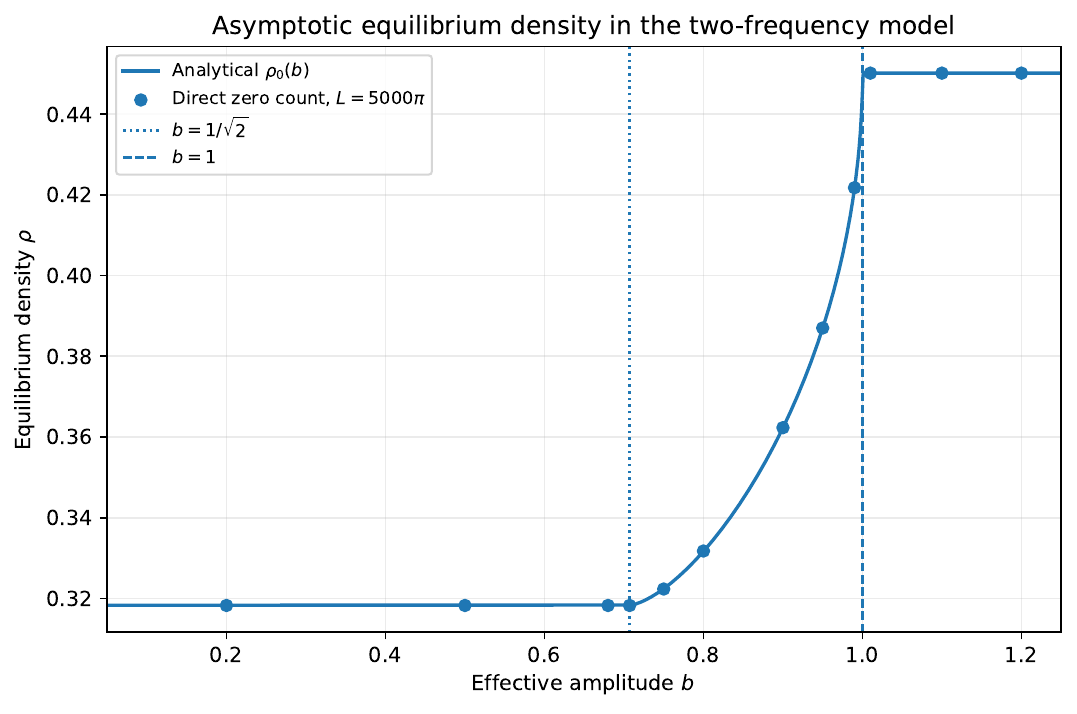}
    \caption{
    Numerical verification of the asymptotic equilibrium density for
    the two-frequency model
    $q_b(x)=\sin x+b\sin(\sqrt{2}x)$.
    The solid curve is the analytical density $\rho_0(b)$ from
    Proposition~\ref{prop-two-frequency-density}, while the symbols are obtained by direct zero
    counting over $0<x<5000\pi$.
    The vertical lines indicate the transition values
    $b=1/\sqrt{2}$ and $b=1$.
    }
    \label{fig:rho_b_numerical}
\end{figure}

}

\subsubsection{Distribution of the critical phases and phase transition}

The scalar quantity $\rho_0(b)$ measures the frequency with which
equilibria occur along the physical channel, but it does not describe
how their phases are distributed on the critical curve. We now turn to
this finer statistic. By
\eqref{eq-two-cosine-leading-empirical-measure}, the limiting phase
distribution is determined by the normalized transverse flux through
$\Sigma_b$. To obtain a more explicit description, we project this
intrinsic measure onto a global angular parameter of the critical
curve.
\\
The natural parameter depends on whether $|b|<1$ or $|b|>1$.

\paragraph{The case $0<|b|<1$.}

In this regime, for every $\theta_2\in\mathbb T$, the equation
\[
\sin\theta_1=-b\sin\theta_2
\]
has two solutions. Set
\[
A_b(\theta_2)
=
\sqrt{1-b^2\sin^2\theta_2}.
\]
Then
\[
\Sigma_b=\Sigma_b^-\cup\Sigma_b^+,
\]
where $\Sigma_b^-$ and $\Sigma_b^+$ are the two disjoint connected
components characterized by
\begin{equation}
\label{eq-branches-small-final}
\cos\theta_1
=
\pm A_b(\theta_2),
\qquad
A_b(\theta_2)
=
\sqrt{1-b^2\sin^2\theta_2}.
\end{equation}
More precisely,
\[
\Sigma_b^\pm
=
\left\{
(\theta_1,\theta_2)\in\Sigma_b:
\cos\theta_1=\pm A_b(\theta_2)
\right\}.
\]
Along $\Sigma_b$, differentiation of
\[
\sin\theta_1+b\sin\theta_2=0
\]
gives
\[
\cos\theta_1\,d\theta_1
+
b\cos\theta_2\,d\theta_2
=
0.
\]
Therefore
\[
\frac{d\theta_1}{d\theta_2}
=
-\frac{b\cos\theta_2}{\cos\theta_1},
\]
and consequently
\[
d\sigma_b
=
\sqrt{
1+
\left(
\tfrac{d\theta_1}{d\theta_2}
\right)^2
}
\,d\theta_2
=
\frac{
\sqrt{\cos^2\theta_1+b^2\cos^2\theta_2}
}{
|\cos\theta_1|
}
\,d\theta_2.
\]
Thus
\[
\frac{d\sigma_b}{|\nabla\mathcal S_b|}
=
\frac{d\theta_2}{|\cos\theta_1|}.
\]
Substituting this identity into
\eqref{eq-two-cosine-leading-empirical-measure} and using
\eqref{eq-branches-small-final}, we obtain on each branch
$\Sigma_b^\pm$
\[
d\mu_b^\pm(\theta_2)
=
\tfrac{1}{(2\pi)^2\rho_0(b)}
\left|
1
\pm
\tfrac{\sqrt2\,b\cos\theta_2}
{A_b(\theta_2)}
\right|
\,d\theta_2.
\]
Let
\[
\nu_b=(\pi_2)_\#\mu_b,
\qquad
\pi_2(\theta_1,\theta_2)=\theta_2.
\]
Adding the contributions of the two branches and using \eqref{dirac-mass}
we obtain
\begin{equation}
\label{eq-final-measure-small-general}
d\nu_b(\theta_2)
=
\frac{1}{2\pi^2\rho_0(b)}
\max\left\{
1,
\frac{\sqrt2\,|b\cos\theta_2|}
{\sqrt{1-b^2\sin^2\theta_2}}
\right\}
\,d\theta_2,
\qquad
0<|b|<1.
\end{equation}
In the case
\[
0<|b|\leq\tfrac1{\sqrt2},
\]
we have
\[
\sqrt2\,|b\cos\theta_2|
\leq
\sqrt{1-b^2\sin^2\theta_2},
\qquad
\rho_0(b)=\frac1\pi,
\]
and therefore
\begin{equation*}
\label{eq-final-measure-small-uniform}
d\nu_b(\theta_2)
=
\frac{d\theta_2}{2\pi}.
\end{equation*}
Thus the projected distribution is uniform throughout the first
globally transverse regime.
\\
For
\[
\frac1{\sqrt2}<|b|<1,
\]
the two terms inside the maximum in
\eqref{eq-final-measure-small-general} exchange dominance. The
transition occurs when
\[
\frac{\sqrt2\,|b\cos\theta_2|}
{\sqrt{1-b^2\sin^2\theta_2}}
=
1,
\]
or equivalently, by the same relation as in \eqref{tangency-one},
\[
\sin^2\theta_2
=
2-\frac1{b^2}.
\]
Consequently,
\[
d\nu_b(\theta_2)
=
\frac{1}{2\pi^2\rho_0(b)}
\max\left\{
1,
\frac{\sqrt2\,|b\cos\theta_2|}
{\sqrt{1-b^2\sin^2\theta_2}}
\right\}
\,d\theta_2,
\]
where $\rho_0(b)$ is given by
\eqref{eq-rho-intermediate-final}. In this regime the projected
measure is therefore nonuniform.\\
This formula makes the role of the tangencies particularly
transparent. Before the loss of transversality, the contributions of
the two branches compensate exactly after projection and produce the
uniform measure. Once tangencies appear, this compensation fails on
nontrivial angular sectors, and the distribution of the critical
phases becomes nonuniform.
{
\paragraph{Computed critical-phase distribution.}
We now verify the limiting phase measure by direct sampling of the
zeros of
\[
q_b(x)=\sin x+b\sin(\sqrt{2}x).
\]
For each zero $x_n$ in the interval $0<x<L$, we associate the hull
phase
\[
\Theta_n
=
\bigl(x_n,\sqrt{2}x_n\bigr)
\pmod{2\pi},
\]
and in particular its second component
\[
\theta_{2,n}
=
\sqrt{2}x_n
\pmod{2\pi}.
\]
The resulting empirical projected measure is
\[
\nu_{b,N}^{\rm num}
=
\frac1N
\sum_{n=1}^{N}
\delta_{\theta_{2,n}}.
\]
We take $L=10^5\pi$ and consider the same representative amplitudes
$b=1/\sqrt{2},\,0.8,\,0.95,$ and $0.99$ used in the analytical
discussion above. Figure~\ref{fig:phase_distribution_numerical}
compares the empirical densities obtained from the direct zero
sampling with the analytical projected density (\ref{eq-final-measure-small-general}) for different values of $b$. To separate the dependence on the parameter $b$ from the finite-sampling
fluctuations, the corresponding analytical profiles are superposed in Figure~\ref{fig:fb-profiles}.
\begin{figure}[h]
    \centering
    \includegraphics[width=0.88\textwidth]{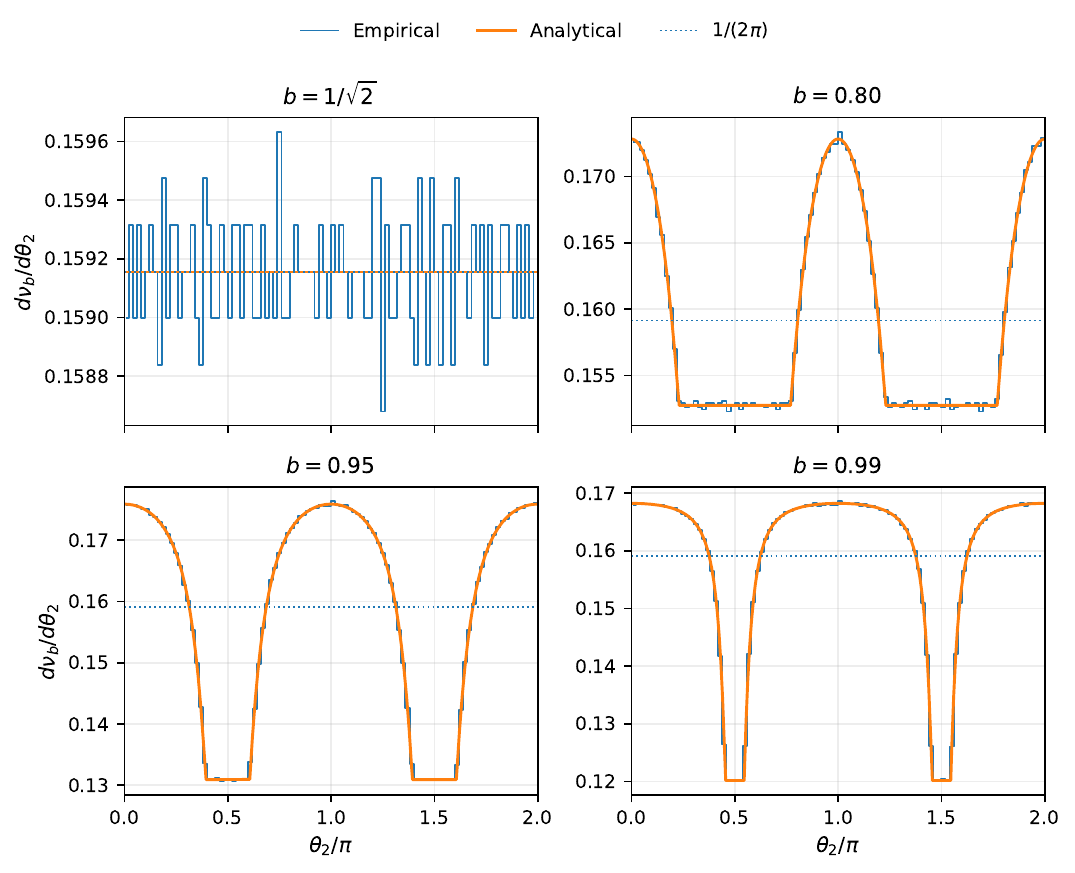}
    \caption{Projected limiting distribution of critical phases for the two-frequency model $q_b(x)=\sin x+b\sin(\sqrt{2}x)$.
The empirical densities are obtained from all zeros in
$0<x<10^5\pi$, while the smooth curves show the analytical projected
density $d\nu_b/d\theta_2$ given by~(\ref{eq-final-measure-small-general}).
The horizontal dashed line denotes the uniform density $1/(2\pi)$.
}
\label{fig:phase_distribution_numerical}
\end{figure}
\begin{figure}[h]
    \centering   \includegraphics[width=0.7\textwidth]{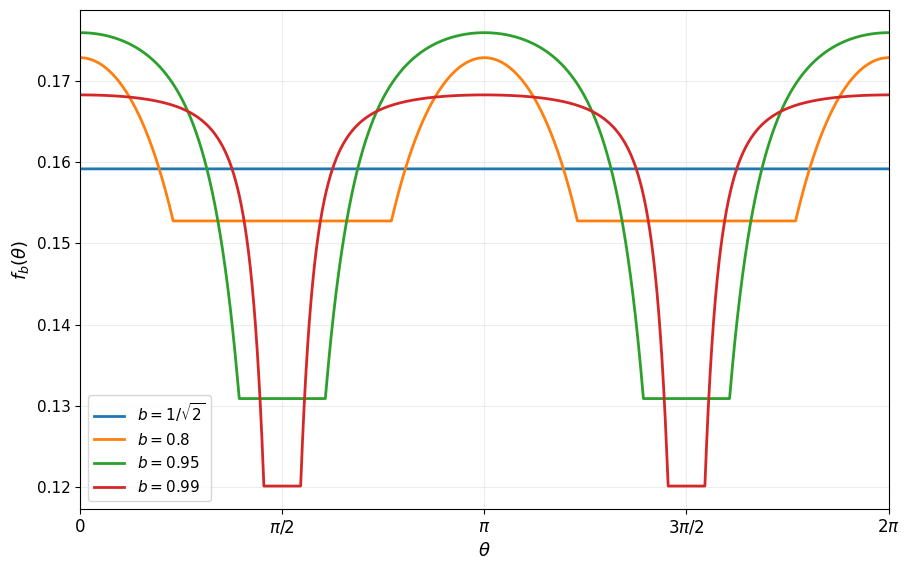}
    \caption{
    Analytical projected densities $d\nu_b/d\theta_2$ given by~(\ref{eq-final-measure-small-general}) for $b=1/\sqrt2$, $0.8$, $0.95$, and $0.99$, superposed to highlight their dependence on the parameter $b$. The case $b=1/\sqrt2$ corresponds to the uniform density $1/(2\pi)$.}
    \label{fig:fb-profiles}
\end{figure}
\\The empirical distributions closely follow the analytical prediction for all four values of $b$. Small fluctuations remain visible because
the phase sequence is sampled over a finite spatial interval, but the overall structure of the limiting distribution is already well resolved for $L=10^5\pi$. Thus the direct zero sampling provides a numerical illustration of the flux-weighted phase measure predicted by the ergodic theory.
\paragraph{The case $|b|>1$.}

When $|b|>1$, the coordinate $\theta_2$ is no longer a global
parameter. Instead, for every $\theta_1\in\mathbb T$,
\[
\sin\theta_2
=
-\tfrac{\sin\theta_1}{b}
\]
has exactly two solutions. Set
\begin{equation}
\label{eq-Cb-final}
C_b(\theta_1)
=
\sqrt{1-\tfrac{\sin^2\theta_1}{b^2}}.
\end{equation}
The two branches are characterized by
\[
\cos\theta_2
=
\pm C_b(\theta_1).
\]
Proceeding as before, but now using $\theta_1$ as the global parameter,
we obtain
\begin{equation}
\label{eq-final-measure-large-uniform}
d\nu_b(\theta_1)
=
\tfrac{d\theta_1}{2\pi},
\qquad
|b|>1,
\end{equation}
where
\[
\nu_b=(\pi_1)_\#\mu_b,
\qquad
\pi_1(\theta_1,\theta_2)=\theta_1.
\]
In the same regime,
\[
\rho_0(b)=\tfrac{\sqrt2}{\pi}.
\]
We therefore obtain a simple global picture. For $|b|<1$, the zero set
is naturally parametrized by $\theta_2$, whereas for $|b|>1$ it is
naturally parametrized by $\theta_1$. In the two outer transverse
regimes,
\[
0<|b|\leq\tfrac1{\sqrt2}
\qquad\text{and}\qquad
|b|>1,
\]
the projected limiting measure is uniform in the corresponding global
parameter. By contrast, in the intermediate regime
\[
\tfrac1{\sqrt2}<|b|<1,
\]
the appearance of tangencies produces a  nonuniform
distribution. In all regular regimes, however, the intrinsic limiting
measure remains supported on the whole critical set $\Sigma_b$; the
one-dimensional formulas above are simply its push-forwards under the
natural graph parametrizations.\\
Figure~\ref{fig:fb-profiles} makes the dependence of the projected measure on $b$ particularly transparent. \\
At
\[
b=\frac{1}{\sqrt2},
\]
the distribution is uniform. For
\[
\frac{1}{\sqrt2}<b<1,
\]
this uniformity is broken and two symmetric depleted regions develop around
\[
\theta_2=\frac{\pi}{2},
\qquad
\theta_2=\frac{3\pi}{2}.
\]
As $b\uparrow1$, these regions become increasingly narrow and pronounced, whereas away from them the density approaches the uniform value $1/(2\pi)$. \\
The restoration of uniformity across $|b|=1$ is therefore singular.
It does not arise from a smooth reversal of the deformation occurring
in the intermediate regime. Rather, the geometry of the critical set
changes at $|b|=1$, and the natural global parametrization switches
from $\theta_2$ to $\theta_1$. The transition should consequently be
understood intrinsically in terms of the curve $\Sigma_b$ and its
transverse flux measure, rather than solely through a fixed
one-dimensional projection.
\\
This explicit example illustrates the main mechanism behind the
ergodic theory developed in this section. The large-scale statistics
of vortex equilibria are controlled by the geometry of the critical
hypersurface relative to the Kronecker direction. Global
transversality leads to regular statistics and, in the present model,
to uniform projected distributions. Finite-order tangencies generate
both nonuniformity and a loss of regularity with respect to the
geometric parameter.
{
\paragraph{Folded spatial distribution of the equilibria.}
The preceding phase distributions describe the equilibria intrinsically
on the critical curve $\Sigma_b$. It is also useful to visualize how
these equilibria are organized in physical coordinates. Let $x_n$
denote the positive zeros of
\[
q_b(x)=\sin x+b\sin(\sqrt{2}x),
\]
and use the first-order critical-point expansion
\[
y_n^{(1)}
=
\frac12+
\frac{2\varepsilon}{\pi}
\partial_yR_1
\left(x_n,\frac12\right).
\]
Since the quasi-periodic channel has no fundamental spatial cell, we
introduce only for visualization the folded coordinate
\[
X_n=x_n\pmod{2\pi}, \qquad X_n \in [-\pi/2,3\pi/2)
\]
Figure~\ref{fig:folded_equilibrium_distribution} shows the resulting
sets $(X_n,y_n^{(1)})$ for several values of $b$.
\begin{figure}[h]
    \centering
    \includegraphics[width=0.92\textwidth]{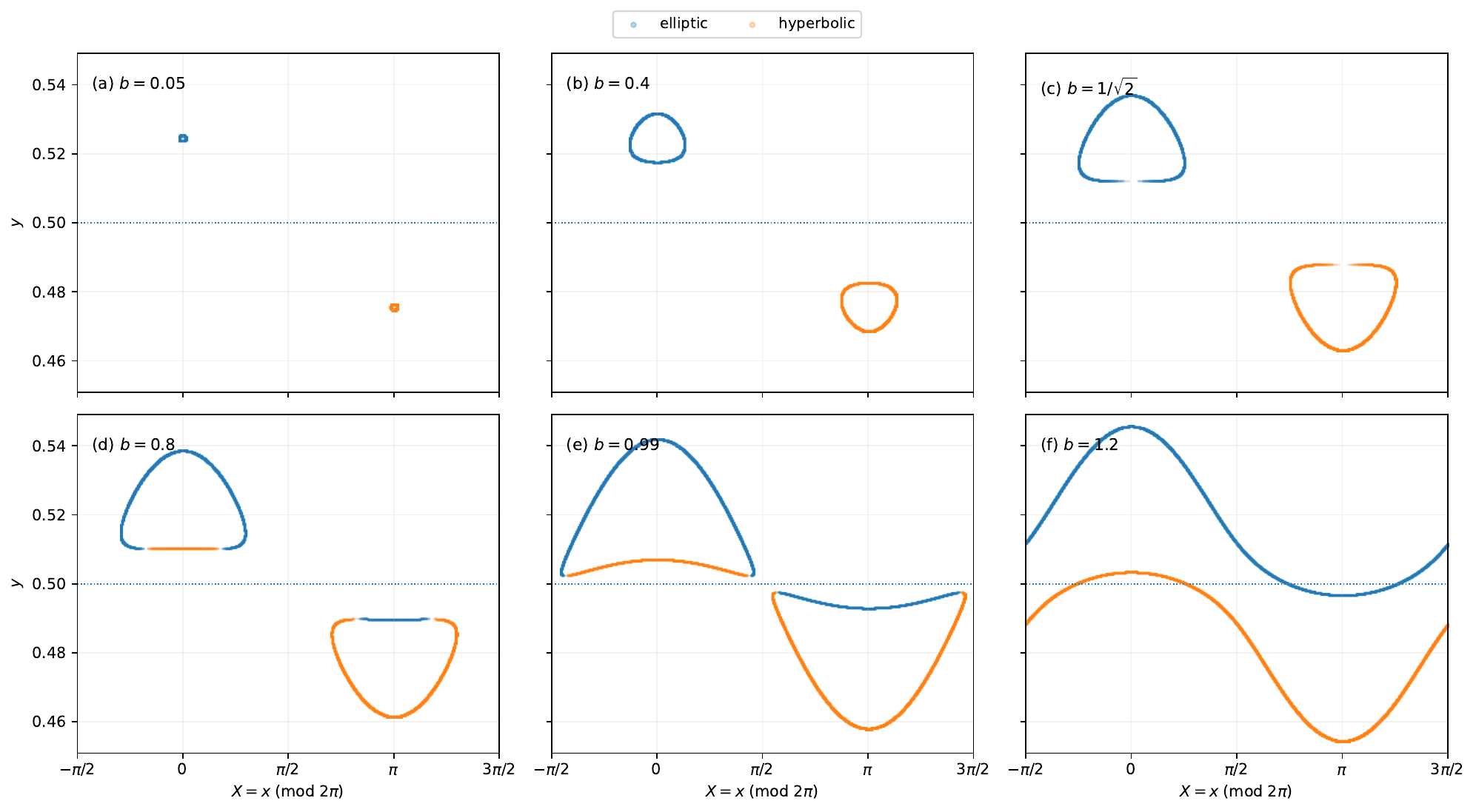}
    \caption{Folded spatial distribution of the vortex equilibria predicted by
the first-order two-frequency model for $\varepsilon=0.05$.
The longitudinal positions are zeros of
$q_b(x)=\sin x+b\sin(\sqrt{2}x)$ sampled over a long spatial interval,
while the transverse positions are obtained from the first-order
critical-point expansion.
The equilibria are represented in the shifted $2\pi$-wide window
$-\pi/2\leq X<3\pi/2$, with $X=x\pmod{2\pi}$, chosen so that the
branches are displayed without artificial cuts at the boundary.
The six panels correspond to
(a) $b=0.05$,
(b) $b=0.4$,
(c) $b=1/\sqrt{2}$,
(d) $b=0.8$,
(e) $b=0.99$, and
(f) $b=1.2$.
Elliptic and hyperbolic equilibria are distinguished according to
the sign of the leading Hessian determinant.
The folding is used only as a visualization device and does not imply
spatial periodicity of the quasi-periodic channel.
    }
    \label{fig:folded_equilibrium_distribution}
\end{figure}
For $b\ll1$, the second incommensurate mode is weak and the folded
distribution concentrates near the two equilibrium positions of the
single-mode periodic problem, $X=0$ and $X=\pi$. As $b$ increases,
these points broaden into curves reflecting the sampling of the two
components of $\Sigma_b$ by the Kronecker orbit. For
$0<b<1/\sqrt{2}$, the two components retain a fixed Morse type.
At $b=1/\sqrt{2}$ the first tangencies occur, while for
$1/\sqrt{2}<b<1$ elliptic and hyperbolic portions coexist on the
same folded branches. As $b\uparrow1$, the admissible longitudinal
support expands toward the whole $2\pi$ window.
\\
This representation should be distinguished from the projected
limiting distribution $\nu_b$ in~\eqref{eq-final-measure-small-general}. For $b<1$, the
projection of $\Sigma_b$ onto $\theta_2$ has full support, whereas
the folded physical coordinate is
$X=\theta_1=x\pmod{2\pi}$ and is constrained by
\[
|\sin X|\leq b.
\]
Consequently, a continuous, and in the globally transverse regime
uniform, limiting distribution in $\theta_2$ is fully compatible with
a spatial distribution localized on lower-dimensional subsets of the
folded $(X,y)$ plane.
As an independent consistency check, direct computations of the full
finite-$\varepsilon$ Robin function at representative parameter values
produce the same qualitative folded organization of the critical
points and close agreement with their first-order locations; these
additional computations are not shown.
}

\section*{Acknowledgment}
T. Hmidi has been supported by Tamkeen under the NYU Abu Dhabi Research Institute grant.

\section*{Data availability}
The numerical data and computational scripts supporting the figures and results presented in this study are available from the authors upon reasonable request.

\section*{Declaration on the use of generative AI}
Generative AI tools were used during the preparation of this manuscript to assist with text editing, selected mathematical and numerical calculations, and aspects of code development and debugging. The authors independently assessed and validated the resulting material and take full responsibility for the manuscript.

\appendix

\section{Auxiliary Diophantine estimates}
The following results are classical.
\begin{lemma}
\label{lem-regularity-u}
Assume that
\[
g\in C^m(\T^d),
\qquad
\widehat g_0=0,
\]
and that $\omega\in\R^d$ satisfies the Diophantine condition
\begin{equation}
\label{eq-diophantine-u}
|\omega\cdot\ell|
\geq
\frac{\gamma}{|\ell|^\tau},
\qquad
\forall\,\ell\in\Z^d\setminus\{0\},
\end{equation}
for some $\gamma>0$ and $\tau>0$. Define
\begin{equation}
\label{eq-u-Fourier-orbit}
u(\theta)
=
\sum_{\ell\in\Z^d\setminus\{0\}}
\frac{\widehat g_\ell}
{i\omega\cdot\ell}
e^{i\ell\cdot\theta}.
\end{equation}
Then, for every integer $k\geq0$ satisfying
\begin{equation}
\label{eq-regularity-condition-u}
k<m-\tau-d,
\end{equation}
one has
\[
u\in C^k(\T^d).
\]
Moreover,
\begin{equation*}
\label{eq-Ck-estimate-u}
\|u\|_{C^k}
\leq
{C}\gamma^{-1}\|g\|_{C^m},
\end{equation*}
where $C$ depends only on $d,m,\tau$ and $k$.
\end{lemma}

\begin{proof}
Since $g\in C^m(\T^d)$, its Fourier coefficients satisfy
\[
|\widehat g_\ell|
\leq
C\|g\|_{C^m}|\ell|^{-m},
\qquad
\ell\neq0.
\]
The Diophantine condition yields
\begin{equation}
\label{eq-u-Fourier-decay-Ck}
|\widehat u_\ell|
=
\frac{|\widehat g_\ell|}{|\omega\cdot\ell|}
\leq
{C}{\gamma}^{-1}
\|g\|_{C^m}
|\ell|^{-(m-\tau)}.
\end{equation}
Let $\beta\in\N^d$ with $|\beta|\leq k$, then
\[
\partial_\theta^\beta u(\theta)
=
\sum_{\ell\neq0}
(i\ell)^\beta\widehat u_\ell e^{i\ell\cdot\theta}.
\]
Using \eqref{eq-u-Fourier-decay-Ck}, we obtain
\[
\begin{aligned}
\sum_{\ell\neq0}
|\ell|^{|\beta|}|\widehat u_\ell|
&\leq
{C}{\gamma}^{-1}
\|g\|_{C^m}
\sum_{\ell\neq0}
|\ell|^{-(m-\tau-|\beta|)}
\\
&\leq
{C}{\gamma}^{-1}
\|g\|_{C^m}
\sum_{\ell\neq0}
|\ell|^{-(m-\tau-k)}.
\end{aligned}
\]
The last series converges because
\[
m-\tau-k>d,
\]
which is exactly \eqref{eq-regularity-condition-u}. Hence the Fourier
series defining $\partial_\theta^\beta u$ converges absolutely and
uniformly for every $|\beta|\leq k$. Consequently,
\[
u\in C^k(\T^d).
\]
The same estimates yield
\[
\|u\|_{C^k}
\leq
{C}{\gamma}^{-1}\|g\|_{C^m},
\]
which concludes the proof.
\end{proof}
\begin{lemma}
\label{lem-full-measure-diophantine}
Let $\tau>d-1$ and let $B\subset\R^d$ be a bounded measurable set. For
$\gamma>0$, define
\[
\mathrm{DC}(\gamma,\tau)
:=
\left\{
\omega\in B:
|\omega\cdot\ell|
\geq
\frac{\gamma}{|\ell|^\tau},
\quad
\forall\,\ell\in\Z^d\setminus\{0\}
\right\}.
\]
Then there exists a
constant $C=C(B,d,\tau)>0$ such that
\[
\bigl|B\setminus \mathrm{DC}(\gamma,\tau)\bigr|
\leq
C\gamma.
\]
Consequently,
\[
\left|
B\cap\bigcup_{\gamma>0}\mathrm{DC}(\gamma,\tau)
\right|
=
|B|.
\]
\end{lemma}

\begin{proof}
For every $\ell\in\Z^d\setminus\{0\}$, introduce the resonant strip
\[
\mathcal R_\ell(\gamma)
:=
\left\{
\omega\in B:
|\omega\cdot\ell|
<
\frac{\gamma}{\langle\ell\rangle^\tau}
\right\}.
\]
Its width in the direction orthogonal to the hyperplane
$\{\omega\cdot\ell=0\}$ is of order
\[
\frac{\gamma}{|\ell|\langle\ell\rangle^\tau}.
\]
Since $B$ is bounded, it follows that
\[
|\mathcal R_\ell(\gamma)|
\leq
C_B
\frac{\gamma}{|\ell|\langle\ell\rangle^\tau}
\leq
C_B
\frac{\gamma}{\langle\ell\rangle^{\tau+1}}.
\]
Therefore
\[
\begin{aligned}
\bigl|B\setminus \mathrm{DC}(\gamma,\tau)\bigr|
&\leq
\sum_{\ell\in\Z^d\setminus\{0\}}
|\mathcal R_\ell(\gamma)|
\\
&\leq
C_B\gamma
\sum_{\ell\in\Z^d\setminus\{0\}}
\frac{1}{\langle\ell\rangle^{\tau+1}}.
\end{aligned}
\]
Since $\tau>d-1$, one has $\tau+1>d$, and thus
\[
\sum_{\ell\in\Z^d\setminus\{0\}}
\langle\ell\rangle^{-(\tau+1)}
<\infty.
\]
Hence
\[
\bigl|B\setminus \mathrm{DC}(\gamma,\tau)\bigr|
\leq
C\gamma.
\]
Letting $\gamma\downarrow0$ gives
\[
\left|
B\setminus\bigcup_{\gamma>0}\mathrm{DC}(\gamma,\tau)
\right|
=0,
\]
which proves the result.

\end{proof}


\begin{tabular}{l}
		\textbf{Mohamed Ali} \\
		{\small Department of Mathematics}\\
		{\small New York University in Abu Dhabi} \\ 
		{\small Saadiyat Island, P.O. Box 129188, Abu Dhabi, UAE} \\ 
		{\small Email: ma8281@nyu.edu}
	\end{tabular}
	
	\vspace{0.2cm}
	
		\begin{tabular}{l}
		\textbf{Taoufik Hmidi} \\
		{\small Department of Mathematics}\\
		{\small New York University in Abu Dhabi} \\ 
		{\small Saadiyat Island, P.O. Box 129188, Abu Dhabi, UAE} \\ 
		{\small Email: th2644@nyu.edu}
	\end{tabular}
\end{document}